\documentclass[11pt]{amsart} 

\usepackage{amssymb,cite}
\usepackage[margin=1in]{geometry}
\usepackage{tensor,enumitem}
\usepackage[all,cmtip]{xy}
\usepackage{hyperref}

\newcommand{\sm}[1]{( \begin{smallmatrix} #1 \end{smallmatrix} )}
\newcommand{\relphantom}[1]{\mathrel{\phantom{#1}}}
\newcommand{\subgrp}[1]{\langle #1 \rangle}
\newcommand{\set}[1]{\left\{ #1 \right\}}
\newcommand{\abs}[1]{\left| #1 \right|}
\newcommand{\bs}[1]{\boldsymbol{#1}}
\newcommand{\wh}[1]{\widehat{ #1}}
\newcommand{\wt}[1]{\widetilde{ #1}}
\newcommand{\ul}[1]{\underline{#1}}
\newcommand{\ol}[1]{\overline{#1}}

\newcommand{\ev}{\textup{ev}}
\newcommand{\odd}{\textup{odd}}
\newcommand{\ve}{\varepsilon}
\newcommand{\dalpha}{\dot{\alpha}}
\newcommand{\ddalpha}{\ddot{\alpha}}
\newcommand{\ualpha}{\ul{\alpha}}
\newcommand{\dbeta}{\dot{\beta}}
\newcommand{\ddbeta}{\ddot{\beta}}

\newcommand{\simrightarrow}{\stackrel{\sim}{\rightarrow}}

\DeclareMathOperator{\chr}{char}
\DeclareMathOperator{\Dist}{Dist}
\DeclareMathOperator{\End}{End}
\DeclareMathOperator{\Ext}{Ext}
\DeclareMathOperator{\opH}{H}
\DeclareMathOperator{\Hom}{Hom}
\DeclareMathOperator{\id}{id}
\DeclareMathOperator{\im}{im}
\DeclareMathOperator{\Lie}{Lie}
\DeclareMathOperator{\Mat}{Mat}
\DeclareMathOperator{\Max}{Max}
\DeclareMathOperator{\res}{res}
\DeclareMathOperator{\Spec}{Spec}

\newcommand{\gotimes}{\tensor[^g]{\otimes}{}}

\renewcommand{\mod}{\, \textup{mod}\, }

\newcommand{\Hbul}{\opH^\bullet}

\newcommand{\F}{\mathbb{F}}
\newcommand{\G}{\mathbb{G}}
\newcommand{\M}{\mathbb{M}}
\newcommand{\N}{\mathbb{N}}
\renewcommand{\P}{\mathbb{P}}
\newcommand{\Z}{\mathbb{Z}}

\newcommand{\Fp}{\F_p}

\newcommand{\Ga}{\G_a}
\newcommand{\Gam}{\Ga^-}
\newcommand{\Gar}{\G_{a(r)}}
\newcommand{\Gal}{\G_{a(\ell)}}

\newcommand{\cp}{\mathcal{P}}
\newcommand{\cv}{\mathcal{V}}

\newcommand{\zero}{\ol{0}}
\newcommand{\one}{\ol{1}}
\newcommand{\0}{\ol{{0}}}
\renewcommand{\1}{\ol{1}}

\newcommand{\bsa}{\bs{A}}
\newcommand{\bsc}{\bs{c}}
\newcommand{\bse}{\bs{e}}
\newcommand{\bsF}{\bs{F}}
\newcommand{\bsg}{\bs{\Gamma}}
\newcommand{\bsi}{\bs{I}}
\newcommand{\bsir}{\bsi^{(r)}}
\newcommand{\bsirone}{{\bsi_1}^{(r)}}
\newcommand{\bsirzero}{{\bsi_0}^{(r)}}
\newcommand{\bsl}{\bs{\Lambda}}
\newcommand{\bsp}{\bs{\cp}}
\newcommand{\bss}{\bs{S}}
\newcommand{\bsv}{\bs{\cv}}
\newcommand{\bsV}{\bs{V}}
\newcommand{\bibsp}{\textup{bi-}\bsp}

\newcommand{\bfHom}{\mathbf{Hom}}

\newcommand{\fS}{\mathfrak{S}}
\newcommand{\g}{\mathfrak{g}}
\newcommand{\gl}{\mathfrak{gl}}
\newcommand{\glmn}{\gl(m|n)}
\newcommand{\glone}{\gl(m|n)_{\one}}
\newcommand{\glzero}{\gl(m|n)_{\zero}}

\newcommand{\Azero}{A_{\zero}}
\newcommand{\Aone}{A_{\one}}
\newcommand{\gone}{\g_{\one}}
\newcommand{\gzero}{\g_{\zero}}
\newcommand{\Vone}{V_{\one}}
\newcommand{\Voner}{{\Vone}^{(r)}}
\newcommand{\Vzero}{V_{\zero}}
\newcommand{\Vzeror}{{\Vzero}^{(r)}}

\newcommand{\alg}{\mathfrak{alg}}
\newcommand{\calg}{\mathfrak{calg}}
\newcommand{\smod}{\textup{smod}}
\newcommand{\svec}{\textup{svec}}
\newcommand{\csalg}{\mathfrak{csalg}}
\newcommand{\salg}{\mathfrak{salg}}
\newcommand{\sgrp}{\mathfrak{sgrp}}
\newcommand{\fsvec}{\mathfrak{svec}}
\newcommand{\fsmod}{\mathfrak{smod}}

\newcommand{\Grf}{\G_{r;f}}
\newcommand{\Matmn}{\Mat_{m|n}}
\newcommand{\Mmn}{M_{m|n}}
\newcommand{\Monef}{\M_{1;f}}
\newcommand{\Monefeta}{\M_{1;f,\eta}}
\newcommand{\Mones}{\M_{1;s}}
\newcommand{\Monet}{\M_{1;t}}
\newcommand{\Mr}{\M_r}
\newcommand{\Mrone}{\M_{r;1}}
\newcommand{\Mrf}{\M_{r;f}}
\newcommand{\Mrfeta}{\M_{r;f,\eta}}
\newcommand{\Mrs}{\M_{r;s}}
\newcommand{\Mrt}{\M_{r;t}}

\newcommand{\Nrs}{N_{r;s}}
\renewcommand{\Pr}{\P_r}
\newcommand{\Vr}{V_r}
\newcommand{\Vrf}{V_{r;f}}
\newcommand{\Vrfeta}{V_{r;f,\eta}}
\newcommand{\Vrs}{V_{r;s}}

\newcommand{\bsVrs}{\bsV_{r;s}}
\newcommand{\bsVrf}{\bsV_{r;f}}
\newcommand{\bsVrfeta}{\bsV_{r;f,\eta}}
\newcommand{\bsVonefeta}{\bsV_{1;f,\eta}}

\newcommand{\Vab}{V_{(\alpha|\beta)}}
\newcommand{\Vuab}{V_{(\ualpha|\beta)}}
\newcommand{\Wuab}{W_{(\ualpha|\beta)}}

\newcommand{\Amn}{A^{m|n}}
\newcommand{\kmn}{k^{m|n}}
\newcommand{\GLmn}{GL_{m|n}}
\newcommand{\GLmnone}{GL_{m|n(1)}}
\newcommand{\GLmnr}{GL_{m|n(r)}}

\newcommand{\wtsigma}{\wt{\sigma}}

\numberwithin{equation}{subsection}

\newtheorem{theorem}{Theorem}[subsection]
\newtheorem{proposition}[theorem]{Proposition}
\newtheorem{corollary}[theorem]{Corollary}
\newtheorem{lemma}[theorem]{Lemma}
\newtheorem{question}[theorem]{Question}

\newtheorem*{theorem*}{Theorem}
\newtheorem*{corollary*}{Corollary}

\theoremstyle{definition}
\newtheorem{definition}[theorem]{Definition}
\newtheorem{convention}[theorem]{Convention}
\newtheorem{notation}[theorem]{Notation}
\newtheorem{example}[theorem]{Example}

\newtheorem{remark}[theorem]{Remark}

\title[Graded analogues of one-parameter subgroups]{Graded analogues of one-parameter subgroups and applications to the cohomology of $GL_{m|n(r)}$}

\author{Christopher M.\ Drupieski}
\address{Department of Mathematical Sciences,
		DePaul University,
		Chicago, IL 60614, USA}
\email{c.drupieski@depaul.edu}

\author{Jonathan R. Kujawa}
\address{Department of Mathematics \\
		University of Oklahoma \\
		Norman, OK 73019, USA}
\email{kujawa@math.ou.edu}

\thanks{The first author was supported in part by a Simons Collaboration Grant for Mathematicians, and by NSF Grant No.\ DMS-1440140 while he was in residence at the Mathematical Sciences Research Institute in Berkeley, CA, during the Spring 2018 semester. The second author was supported in part by NSA grant
H98230-16-0055.}


\allowdisplaybreaks

\begin{document}

\begin{abstract}
We introduce a family $\Mrfeta$ of infinitesimal supergroup schemes, which we call multi\-param\-eter supergroups, that generalize the infinitesimal Frobenius kernels $\Gar$ of the additive group scheme $\Ga$. Then, following the approach of Suslin, Friedlander, and Bendel, we use functor cohomology to define characteristic extension classes for the general linear supergroup $\GLmn$, and we calculate how these classes restrict along homomorphisms $\rho: \Mrfeta \rightarrow \GLmn$. Finally, we apply our calculations to describe (up to a finite surjective morphism) the spectrum of the cohomology ring of the $r$-th Frobenius kernel $\GLmnr$ of the general linear supergroup $\GLmn$.
\end{abstract}

\maketitle

\tableofcontents

\section{Introduction}

\subsection{Overview}

Let $k$ be a field of characteristic $p \geq 3$. For more than a century researchers have investigated the modular representation theory of finite groups. In 1971, Quillen \cite{Quillen:1971} brought a groundbreaking new geometric perspective to the subject by investigating the affine variety $\abs{G}$ defined by the cohomology ring $\Hbul(G,k)$ of a finite group $G$. Quillen showed that $\abs{G}$ is a disjoint union of locally closed subsets $V_{G,E}^+$, one for each conjugacy class of elementary abelian $p$-subgroups in $G$. In particular, he showed that the geometric dimension of $\abs{G}$ is equal to the $p$-rank of $G$, demonstrating that geometric information encoded by $\Hbul(G,k)$ can be used to extract structural information about $G$. Later, Carlson \cite{Carlson:1983} introduced for each finite-dimensional $kG$-module $M$ an associated subvariety $\abs{G}_M$ of $\abs{G}$, called the support variety of $M$. For $E$ an elementary abelian $p$-subgroup of $G$, Carlson conjectured a non-cohomological, entirely representation-theoretic `rank variety' description for $\abs{E}_M$. Avrunin and Scott \cite{Avrunin:1982} subsequently proved Carlson's conjecture, and thus obtained a Quillen-type stratification for $\abs{G}_M$ in terms of the varieties $\abs{E}_M$. As a consequence, one can describe $\abs{G}_M$ without recourse to cohomology. This formulation is essential for explicit computations, and is key to proving that support varieties obey the tensor product property $\abs{G}_{M \otimes N} = \abs{G}_M \cap \abs{G}_N$. The central role of the elementary abelian subgroups is further underscored by Quillen's theorem that an element of $\Hbul(G,k)$ is nilpotent if and only if it is nilpotent upon restriction to every elementary abelian $p$-subgroup, and by Chouinard's theorem that a $kG$-module is projective if and only if it is so upon restriction to every such subgroup. 

In the mid 1980s, Friedlander and Parshall \cite{Friedlander:1986a,Friedlander:1986b,Friedlander:1987} defined and studied support varieties for finite-dimensional restricted Lie algebras, ultimately showing the support variety of a module $M$ is determined by the restriction of $M$ to cyclic $p$-nilpotent subalgebras. In the late 1990s, Suslin, Friedlander, and Bendel \cite{Suslin:1997, Suslin:1997a} extended these results to arbitrary infinitesimal group schemes. Given an affine group scheme $G$, they defined an infinitesimal one-parameter subgroup of height $\leq r$ in $G$ to be a homomorphism $\nu: \Gar \rightarrow G$ from the $r$-th Frobenius kernel of the additive group scheme $\Ga$. They showed the infinitesimal one-parameter subgroups play a fundamental role akin to that played by elementary abelian subgroups in finite group theory. One of their main results was that cohomology classes are detected (modulo nilpotence) via restriction to one-parameter subgroups. They also showed that one-parameter subgroups again give a representation-theoretic `rank variety' description for the support variety $\abs{G}_M$ of a $G$-module $M$. In particular, one-parameter subgroups can be used to detect whether or not $M$ is projective.

In the mid-2000s, Friedlander and Pevtsova \cite{Friedlander:2007} introduced the technology of $\pi$-points, which enabled them to unify and extend to arbitrary finite group schemes the previously separate support variety theories for finite groups and infinitesimal group schemes. But even with the conceptual unification provided by $\pi$-points, one-parameter subgroups continue to play an important role. For example, a long-standing open question has been to find a suitable support variety theory for the rational representations of arbitrary affine group schemes. It is not at all clear that such a theory should exist. Indeed, the rational cohomology ring of a reductive algebraic group is zero in all positive degrees, and hence will not lead to any interesting geometry. Recently, Friedlander \cite{Friedlander:2015} has shown one can use one-parameter subgroups to develop a satisfactory theory. Under some mild assumptions on $G$, he uses one-parameter subgroups to directly define support varieties for rational $G$-modules and proves these varieties exhibit many of the desirable properties expected of a support theory. However, there are also new features. For example, while an injective rational $G$-module must have trivial support, a module with trivial support need not be injective. This interesting class of modules is the class of so-called \emph{mock injective} $G$-modules \cite{Friedlander:2018,Hardesty:2017}.

\subsection{The graded setting}

This work is a contribution to the study of the representation theory of finite graded group schemes in positive characteristic; for the characteristic zero setting, see \cite[\S3]{Drupieski:2016a}. We focus on gradings by the group $\Z_2 : = \Z/2\Z$, and following the literature we use the prefix ``super'' to indicate that the object in question is $\Z_2$-graded. Besides being of intrinsic interest, $\Z_2$-gradings are natural to consider because any $\Z$-graded object can be viewed as $\Z_2$-graded by reducing the grading modulo two, and thus results for superalgebras may shed light on problems for $\Z$-graded algebras. For example, motivated by finite-dimensional Hopf subalgebras of the Steenrod algebra and by the Adams spectral sequence, in the 1980s and 1990s algebraic topologists studied the cohomology of finite-dimensional $\Z$-graded \emph{connected} cocommutative Hopf algebras. In 1981, Wilkerson \cite{Wilkerson:1981} showed that the cohomology of any such algebra is finitely-generated, but it was only with the first author's work in the $\Z_2$-setting \cite{Drupieski:2016,Drupieski:2017}, 35 years later, that it was established that the connectedness assumption can be dropped. In the 1990s, Palmieri \cite{Palmieri:1997} and Nakano and Palmieri \cite{Nakano:1998} investigated stratification results, analogous to those of Quillen and of Avrunin and Scott, for finite-dimensional graded connected cocommutative Hopf algebras in general and for finite-dimensional Hopf subalgebras of the mod-$p$ Steenrod algebra in particular. Ongoing work on supergroups by the authors \cite{Drupieski:2017b} and by Benson, Iyengar, Krause, and Pevtsova (discussed at the end of Section \ref{subsection:mainresults} below) is poised to extend those stratification results to arbitrary finite-dimensional $\Z$-graded cocommutative Hopf algebras.


Let $k$ be a fixed ground field of characteristic $p \geq 3$. Write $\csalg_k$ for the category of commutative $k$-superalgebras. For precise definitions of the objects under discussion the reader can consult Section~\ref{SS:GeneralConventions} and the references therein. An affine $k$-supergroup scheme is a representable functor from the category $\csalg_{k}$ to the category of groups. The prototypical example of an affine supergroup scheme is the general linear supergroup $\GLmn$. On a superalgebra $A = \Azero \oplus \Aone$, it is defined by setting $\GLmn(A)$ to be the set of all invertible $(m+n) \times (m+n)$ matrices $(a_{i,j})_{1 \leq i, j \leq m+n}$ such that $a_{i,j} \in \Azero$ if either $1 \leq i,j \leq m$ or $m+1 \leq i,j \leq m+n$, and such that $a_{i,j} \in \Aone$ otherwise. The Frobenius endomorphism $F: \GLmn \to \GLmn$ is defined on a matrix $g \in \GLmn(A)$ by raising the individual matrix entries of $g$ to the the $p$-th power. The $r$-th Frobenius kernel $\GLmnr$ of $\GLmn$ is the scheme-theoretic kernel of the $r$-th iterate of $F$. As for affine group schemes, the Frobenius kernels are an important family of infinitesimal supergroup schemes.

In \cite{Drupieski:2016,Drupieski:2017}, the first author proved that the cohomology ring of a finite supergroup scheme is always a finitely-generated graded-(super)commutative algebra, so its spectrum provides a natural geometric setting in which to introduce cohomological support varieties. A fundamental first step in developing this theory is to provide a concrete description of the spectrum. The authors' previous paper \cite{Drupieski:2016a} provides such a description in several natural settings. Relevant to the present work is the case of the first Frobenius kernel $\GLmnone$ of $\GLmn$. Set $\g = \glmn$. In \cite{Drupieski:2016a}, we proved that there is a finite morphism
\[
\Phi: \Max \left(\Hbul(G_{1}, k)\right) \to \glmn 
\]
with image equal to 
\begin{equation}\label{E:Frob1}
V_1(\GLmnone)(k):= \set{ (\alpha, \beta) \in \mathfrak{g}_{\0} \times \mathfrak{g}_{\1} : [\alpha, \beta]=0, \alpha^{[p]} + \tfrac{1}{2}[\beta, \beta] = 0}.\footnote{In the preprint version of \cite{Drupieski:2016a}, the last condition appears instead as $\alpha^{[p]} = \frac{1}{2}[\beta,\beta]$. This is because of an implicit sign error in the argument in the first paragraph of the proof of \cite[Proposition 5.4.2]{Drupieski:2016a}. We correct the issue here in Remark \ref{remark:matrixtohom} and in the discussion of Section \ref{section:geometricapps}.}
\end{equation}
When $n = 0$, one has $\gone = 0$, and the set $V_1(\GLmnone)(k)$ becomes the restricted nullcone. In this way our result is a natural generalization of known results in the non-graded setting. One outcome of the present work is to extend \eqref{E:Frob1} to the higher Frobenius kernels of $\GLmn$.

\subsection{Main results} \label{subsection:mainresults}

Given their importance in the classical setting, the main goal of this paper is to introduce and study graded analogues of one-parameter subgroups. In contrast to the classical setting, there does not seem to be a suitable family of infinitesimal supergroups describable by a single parameter such as the height (see Section~\ref{SS:Caution}). Instead, the role of the one-parameter subgroups seems to be played by a family of infinitesimal supergroups $\Mrfeta$ which are described by multiple parameters. It should be emphasized that the structure and representation theory of these supergroups is more varied than in the classical case. For example, if the parameter $f$ is not simply a monomial, then the group algebra $k\Mrfeta$ will contain semisimple elements.

In Section~\ref{S:multiparametersupergroups} we define these `multiparameter supergroups' by describing their coordinate super\-algebras and their group algebras, and we give a complete description of their rational cohomology rings. 
Given an algebraic $k$-supergroup scheme $G$, we also describe the affine superscheme structure on the functor $\bfHom(\Mrfeta,G)$ of all supergroup scheme homomorphisms $\rho: \Mrfeta \rightarrow G$. In Section \ref{S:Ext} we investigate the Hopf structure of the extension algebra $\Ext_{\bsp}^\bullet(\bsir,\bsir)$ of the $r$-th Frobenius twist of the identity functor in the category $\bsp$ of strict polynomial superfunctors. These results generalize those of \cite[\S3]{Suslin:1997}. Then following the approach of Suslin, Friedlander, and Bendel, in Section \ref{section:charclasses} we introduce characteristic classes corresponding to elements of $\Ext_{\bsp}^\bullet(\bsir,\bsir)$ and rational representations $\rho: G \rightarrow GL(V)$ of an affine supergroup scheme $G$. The main results of this section are in Section \ref{SS:calculationofcharacteristicclasses}, where we provide explicit descriptions of the characteristic classes arising from representations $\rho: \Mrfeta \rightarrow \GLmn$. These descriptions enable us in Theorem \ref{theorem:Extalgebrarelations} to completely specify, for the first time, the structure constants of the algebra $\Ext_{\bsp}^\bullet(\bsir,\bsir)$.

In Section~\ref{section:geometricapps} we reap the harvest, applying the explicit calculations of Section \ref{section:charclasses} to describe the cohomological spectrum of the Frobenius kernels of $\GLmn$.  Suppose $k$ is algebraically closed of characteristic $p\geq 3$. Let $\GLmnr$ denote the $r$-th Frobenius kernel of $\GLmn$, and let $\abs{\GLmnr}$ denote the spectrum of $\Hbul(\GLmnr,k)$. Then there is a finite morphism of schemes
\[
\Phi: \abs{\GLmnr} \rightarrow \Vr(\GLmn),
\] where $\Vr (\GLmn)$ is the affine scheme defined in Definition~\ref{definition:evensubschemes}; see also Definition \ref{definition:VrGLmn}. Set $\g = \glmn$. On $k$-points, $\Phi$ induces a finite-to-one, surjective morphism of varieties,
\[
\Phi(k): \abs{\GLmnr} \rightarrow \Vr(\GLmn)(k),
\] where
\begin{align*}
\Vr(\GLmn)(k)= \Big\{ (\alpha_0,\alpha_1,\ldots,\alpha_{r-1},\beta) \in &(\gzero)^{\times r} \times \gone : [\alpha_i,\alpha_j] = 0, [\alpha_i,\beta] = 0 \text{ for all $i,j$,} \\
&\alpha_i^{[p]} = 0 \text{ for $0 \leq i \leq r -2$, and } \alpha_{r-1}^{[p]} + \tfrac{1}{2}[\beta,\beta] = 0 \Big\}.
\end{align*}
This verifies a conjecture we made in \cite{Drupieski:2016a}. Note that for $n=0$, one gets $\GLmnr = GL_{m(r)}$, the $r$-th Frobenius kernel of $GL_m$, and $\Vr(\GLmn) = V_r(GL_m)$ is the variety of $r$-tuples of commuting $p$-nilpotent matrices. We thus recover one of the main results of \cite{Suslin:1997}. In light of those classical results it is reasonable to expect that $\Phi(k)$ is in fact a homeomorphism and that the multiparameter supergroups introduced in this paper (together with the one-parameter subgroups of the classical theory and the purely odd additive group scheme $\Ga^-$ described just before Remark \ref{remark:immediateMrs}) may be used to describe the spectra of other infinitesimal supergroup schemes.

The role of the multiparameter supergroups in the surjectivity of $\Phi(k)$ can be made explicit as follows. For each inseparable $p$-polynomial $0 \neq f \in k[T]$ (i.e., a $p$-polynomial with no linear term) and each $\eta \in k$, there exists a sequence of algebra homomorphisms
\[
k[\Vr(\GLmn)] \stackrel{\ol{\phi}}{\longrightarrow} H(\GLmn,k) \stackrel{\psi_{r;f,\eta}}{\longrightarrow} k[\Vrfeta(\GLmn)],
\]
which induce morphisms of schemes
\[
\Vrfeta(\GLmn) \stackrel{\Psi_{r;f,\eta}}{\longrightarrow} \abs{\GLmnr} \stackrel{\Phi}{\longrightarrow} \Vr(\GLmn).
\]
Here and in the rest of the paper, $H(G,k)$ denotes the subalgebra $\opH^{\ev}(G,k)_{\zero} \oplus \opH^{\odd}(G,k)_{\one}$ of the (full) cohomology ring $\Hbul(G,k)$.\footnote{In the classical setting when working over fields of odd characteristic, it is common to consider only the even part of the cohomology ring, since then the odd part consists of nilpotent elements and hence does not contribute anything to the spectrum. But in the super setting, one may have non-nilpotent elements of odd cohomological degree, which motivates our definition of the ring $H(G,k)$. For more details, see Section \ref{subsection:phi}.} The calculations of Section \ref{section:charclasses} enable us to identify the composite morphism $\Theta_{r;f,\eta} := \Phi \circ \Psi_{r;f,\eta}: \Vrfeta(\GLmn) \rightarrow \Vr(\GLmn)$ as the natural inclusion $\Vrfeta(\GLmn) \subset \Vr(\GLmn)$ composed with the $r$-th Frobenius twist morphism on the scheme $\Vr(\GLmn)$; see Theorem~\ref{theorem:Theta}. Since
\begin{equation}\label{E:StratificationTheorem}
\Vr(\GLmn)(k) = \bigcup_{f,\eta} \Vrfeta(\GLmn)(k),
\end{equation}
this implies the surjectivity of $\Phi(k)$. Since $\Vrfeta(\GLmn)(k)$ parametrizes homomorphisms $\rho: \Mrfeta \rightarrow \GLmn$, \eqref{E:StratificationTheorem} can be viewed as an analogue of Quillen's stratification theorem.    

It is also noteworthy that our setup incorporates the gradings of \cite[Theorem 1.14]{Suslin:1997} in a natural and transparent way. Namely, suppose $f=T^{p^s}$ for some $s \geq 1$, and suppose $\eta = 0$, so that $\Mrfeta = \Mrs$ and $\Vrfeta = \Vrs$. Then given an algebraic $k$-supergroup scheme $G$, the homomorphism $\psi_{r;f,\eta}: H(G, k) \rightarrow k[\Vrfeta(G)]$ is a map of graded algebras that multiplies degrees by $\frac{p^r}{2}$, just as in \cite[Theorem 1.14]{Suslin:1997}; see Proposition~\ref{P:grading}. However, in contrast to the classical setting, in our setup the cohomology ring $H(G, k)$ can be nonzero in both even and \emph{odd} degrees, and the natural grading for $k[\Vrs(G)]$ is by $\mathbb{Z}[\frac{p^{r}}{2}]$; see Corollary \ref{cor:VrsGgraded}.

Work in progress by Benson, Iyengar, Krause, and Pevtsova (BIKP) gives further evidence of the importance of multi\-parameter supergroups. They independently introduce detecting subalgebras equivalent to our multi\-parameter supergroups $\M_{r;T^{p^s},\eta}$ (with $\eta = 0$ if $r=1$) and prove for unipotent infini\-tesimal supergroup schemes that their subalgebras detect the nilpotency of cohomology classes and the projectivity of modules. But as we describe in Example \ref{example:semisimple}, this family alone seems to be inadequate for detecting projectivity for arbitrary infinitesimal supergroups.

\subsection{Two cautionary examples} \label{SS:Caution}

The reader may hope that one can further reduce to true one-parameter subsupergroups. The following examples show this will not be possible in general.

\begin{example}
Let $k$ be a field of characteristic $p \geq 3$, and let $A = k[u,v]/\subgrp{u^p,v^2}$ be the superalgebra with grading given by letting $u$ be even and $v$ be odd. Then $A$ is the group algebra of $\M_{1;1}$; see Proposition~\ref{prop:groupalgebras}. Alternatively, $A$ is the restricted enveloping superalgebra $V(\g)$ of the abelian restricted Lie superalgebra $\g$ spanned by $u$ and $v$ such that $u^{[p]} = 0$. Note that the only Lie sub\-superalgebras of $\g$ are $0$, $k.u$, $k.v$, and $\g$.

Let $R = R_{\zero} \oplus R_{\one}$ be the superspace such that the underlying vector spaces of $R_{\zero}$ and $R_{\one}$ are given by $R_{\zero}=R_{\one} = k[u]/\subgrp{u^p}$. Then $R$ is a free module over the subalgebra of $A$ generated by $u$.  Let $x_0$ (resp.\ $y_0$) denote the vector $1$ in $R_{\zero}$ (resp.\ $R_{\one}$). For $0 \leq i < p$, set $x_i = u^i.x_0$ (resp.\ $y_i = u^i.y_0$). Then the set $\set{ x_i, y_i : 0 \leq i < p}$ is a homogeneous basis for $R$. Define an action of $v$ on $R$ by setting $v.x_i = y_{i+1}$ and $v.y_i = x_{i+p-1}$ (with the convention that $x_k = y_k =0$ whenever $k \geq p$). A direct computation verifies that this makes $R$ into an $A$-supermodule. Now one can check (e.g., with the aid of \cite[Proposition 2.2]{Friedlander:2005}) that $R$ is projective when restricted to any proper cyclic subsuperalgebra of $A$, but that $R$ is not projective for $A$ itself. That is, no family of cyclic subsuperalgebras of $A$ will detect projectivity.
\end{example}

\begin{example} \label{example:semisimple}
Let $B = k[u,v]/\subgrp{u^p + v^2, u^p-u}$ be the superalgebra with $\Z_2$-grading given by declaring $u$ to be even and $v$ to be odd. This is the group algebra of $\M_{1;T^{p},-1}$; see Definition~\ref{Def:anotherdef}. Alternatively, $B$ is the restricted enveloping superalgebra $V(\g)$ of the restricted Lie superalgebra $\g$ spanned by $u$ and $v$ such that $u^{[p]} = u$ and $u^{[p]} + \frac{1}{2}[v,v] = 0$. The only restricted Lie sub\-super\-algebras of $\g$ are $0$, the semisimple subalgebra generated by $u$, and $\g$ itself.

The subalgebra of $B$ generated by $u$ is a semisimple Hopf (super)algebra, and the quotient algebra $B/\subgrp{u}$ identifies with the one-variable exterior algebra $\Lambda(v)$. Applying the Lyndon-Hochschild-Serre spectral sequence, it follows that the quotient map $B \twoheadrightarrow \Lambda(v)$ induces an isomorphism in cohomology $\Hbul(B,k) \cong \Hbul(\Lambda(v),k)$. Now $\Hbul(\Lambda(v),k) \cong k[y]$, with the polynomial generator $y$ of $k[y]$ located in cohomological degree $1$ and in $\Z_2$-degree $\one$.

In the height-one case of the BIKP (infinitesimal unipotent) setup, one only need consider the multi\-parameter supergroups of the form $\Mones := \M_{1;T^{p^s},0}$ for $s \geq 1$. The group algebra of $\Mones$ has the form $k\Mones = k[u,v]/\subgrp{u^p+v^2,u^{p^s}}$, with the $\Z_2$-degrees of $u$ and $v$ the same as above. Since the odd superdegree generator of $k\Mones$ is nilpotent, any superalgebra homomorphism $\rho: k\Mones \rightarrow B$ must factor through the purely even quotient $k[u]/\subgrp{u^p}$ of $k\Mones$, and consequently it follows that the induced map in cohomology $\rho^*: \Hbul(B,k) \rightarrow \Hbul(\Mones,k)$ must be trivial. This shows that for the purposes of detecting cohomology of general infinitesimal supergroup schemes one must consider more than just unipotent supergroup schemes.
\end{example}

\subsection{Future work} \label{SS:futurework}

The results of this paper raise a number of interesting questions. Given that $\Phi (k)$ is a homeomorphism in the classical case, it is natural to expect that the same is true in the super case as well. This is closely related to the question of whether the multiparameter supergroups introduced in this paper detect nilpotence of cohomology classes for arbitrary infinitesimal supergroups. The results of BIKP for unipotent infinitesimal supergroups suggest that this may be the case.

It is also desirable to refine \eqref{E:StratificationTheorem} along the lines of Quillen's original stratification theorem. Specifically, given a finite group $G$, one can form the colimit $\varinjlim \abs{E}$ over the ``Quillen category,'' whose objects consist of the elementary abelian $p$-subgroups of $G$, and whose morphisms consist of the inclusions and conjugation homomorphisms between them. Then Quillen's theorem states that the natural map $\varinjlim \abs{E} \rightarrow \abs{G}$ is an inseparable isogeny. In our context there are various inclusion maps among the $\Vrfeta(\GLmn)$, and one can form a corresponding colimit $\varinjlim \Vrfeta(\GLmn)$. It would be interesting to give an analogue of the Quillen category in this setting in order to give a more precise description of $\abs{\GLmnr}$. 

More generally, one can ask for an analogue of the Avrunin-Scott theorem in order to describe the support varieties of finite-dimensional supermodules in terms of multiparameter supergroups. Such a result would be valuable for both theory and calculations. It should generalize the non-cohomological description of support varieties given by Suslin, Friedlander, and Bendel \cite{Suslin:1997a}. We remark that this generalization is not expected to be routine, and will need to take into account the more interesting structure and representation theory of the multiparameter supergroups and of supergroup schemes in general. The examples of Section \ref{SS:Caution} show that the naive generalizations of the classical theory are inadequate even in the height-one infinitesimal case (i.e., for finite-dimensional restricted Lie superalgebras). 
Nevertheless, the announced results by BIKP that multiparameter supergroups detect projectivity in the unipotent case suggests that such a theory should be possible in general.\footnote{The authors have made some progress toward describing support varieties for a particular class of infinitesimal unipotent supergroup schemes; see \cite{Drupieski:2017b}.}

Finally, it of course would be extremely useful to generalize the results of this paper to other infinitesimal supergroup schemes, and to describe the spectrum of an arbitrary infinitesimal super\-group scheme $G$ in terms of the various homomorphisms $\rho: \Mrfeta \rightarrow G$. This would involve establishing suitable naturality properties, and could also involve understanding the supergroup analogue of ``embeddings of exponential type'' as in \cite{Suslin:1997}.

\subsection{Acknowledgements}\label{SS:Acknowledgements}  The authors are pleased to thank David Benson and Julia Pevtsova for enlightening conversations.

\section{Conventions} \label{conventions}

\subsection{General conventions}\label{SS:GeneralConventions}

Throughout the paper $k$ will denote a field of characteristic $p \geq 3$. Except when indicated otherwise, all vector spaces will be $k$-vector spaces and all unadorned tensor products will denote tensor products over $k$. If $V$ is a $k$-vector space, then $V^\# = \Hom_k(V,k)$ will denote its $k$-linear dual, and $V^{(r)} = V \otimes_{\varphi^r} k$ will denote the $r$-th Frobenius twist of $V$, i.e., the $k$-vector space obtained via base change along the $r$-th iterate of the Frobenius endomorphism $\varphi: \lambda \mapsto \lambda^p$. Given $v \in V$, we will write $v^{(r)}$ to denote the element $v \otimes_{\varphi^r} 1 \in V^{(r)}$.

We generally follow the notation, terminology, and conventions of \cite{Drupieski:2013b,Drupieski:2016,Drupieski:2016a}. In particular, we assume that the reader is familiar with the sign conventions of `super' linear algebra. Set $\Z_2 = \Z/2\Z = \{ \zero,\one \}$, and write $V = \Vzero \oplus \Vone$ for the decomposition of a superspace $V$ into its even and odd subspaces. Given a homogeneous element $v \in V$, write $\ol{v} \in \Z_2$ for the $\Z_2$-degree of $v$. Whenever we state a formula involving homogeneous degrees of elements, we mean that the formula is true as written for homogeneous elements and that it extends linearly to non-homo\-geneous elements. The category $\svec_k$, whose objects are the $k$-superspaces and whose morphisms are the arbitrary $k$-linear maps between them, is enriched over itself. The category $\svec_k$ is not an abelian category, though the underlying even subcategory $\fsvec_k = (\svec_k)_\ev$, having the same objects but only the even linear maps as morphisms, is an abelian category. Isomorphisms arising from even linear maps will be denoted by the symbol ``$\cong$'' while isomorphisms arising from odd linear maps will be denoted by the symbol ``$\simeq$''. If $A$ and $B$ are graded superalgebras (in the sense of \cite[\S2.2]{Drupieski:2016a}), then we write $A \gotimes B$ for the graded tensor product of $A$ and $B$ \cite[Definition 2.2.2]{Drupieski:2016a}.

Let $\sgrp_k$ denote the category of affine $k$-supergroup schemes. Then $\sgrp_k$ is anti-equivalent to the category of super-commutative Hopf $k$-superalgebras. A $k$-supergroup scheme $G$ is algebraic if its coordinate superalgebra $k[G]$ is finitely generated over $k$.

Write $\N$ for the set $\set{0,1,2,\ldots}$ of non-negative integers.

\subsection{Commutative superalgebras} \label{subsection:commsuperrings}

Let $A$ be a commutative $k$-superalgebra. Then $A$ is a $\Z_2$-graded $k$-algebra such that $ba = (-1)^{\ol{a} \cdot \ol{b}} ab$ for all $a,b \in A$. In this paper the term \emph{commutative} will always be used in the sense given here, and the usual notion of commutativity for abstract rings will be referred to as `ordinary' commutativity when it is necessary to distinguish the two notions. From now on we will write $\salg_k \supset \csalg_k \supset \calg_k$ to denote the categories whose objects are the $k$-superalgebras, the commutative $k$-superalgebras, and the purely even commutative $k$-superalgebras (equivalently, the ordinary commutative $k$-algebras), respectively, and whose morphisms are the \emph{even} $k$-superalgebra homomorphisms between them. Note that there is no such thing as an odd super\-algebra homomorphism, because a superalgebra homomorphism $\phi: A \rightarrow B$ must by definition satisfy $\phi(1_A) = 1_B$, and the identity element of a superalgebra is always even \cite[Lemma 3.1.2]{Westra:2009}.

Since $A$ is commutative, each left $A$-supermodule is canonically a right $A$-super\-module (and vice versa) via the action $m.a = (-1)^{\ol{a} \cdot \ol{m}} a.m$, so we will consider each $A$-supermodule as an $A$-bimodule, and we will only distinguish between left and right actions when it is convenient to do so (e.g., when the $\pm$ signs are simpler). In particular, if $V$ and $W$ are $A$-super\-modules, then the tensor product $V \otimes_A W$ is again an $A$-supermodule, and $V \otimes_A W \cong W \otimes_A V$ via the supertwist map $T : V \otimes_A W \rightarrow W \otimes_A V$ defined by $T(v \otimes w) = (-1)^{\ol{v} \cdot \ol{w}} w \otimes v$.

Write $\Hom_A(M,N)$ for the set of all right $A$-super\-module homomorphisms $\phi: M \rightarrow N$, i.e., the set of all additive maps such that $\phi(m.a) = \phi(m).a$ for all $m \in M$ and $a \in A$.\footnote{In the literature, the notation $\Hom_A(M,N)$ may be used to denote just the even (i.e., parity-preserving) $A$-super\-module homomorphisms, and then the set of all not necessarily even $A$-linear maps is denoted $\underline{\Hom}_A(M,N)$.} In terms of left actions, $\phi \in \Hom_A(M,N)$ if and only if $\phi(a.m) = (-1)^{\ol{a} \cdot \ol{\phi}}a.\phi(m)$ for all $a \in A$ and $m \in M$. The set $\Hom_A(M,N)$ is naturally $\Z_2$-graded and is naturally an $A$-supermodule via the left action $(a \cdot \phi)(m) := a.(\phi(m))$; the right action is then given by $(\phi \cdot a)(m) = (-1)^{\ol{a} \cdot \ol{m}}\phi(m) \cdot a$. Composition of $A$-supermodule homomorphisms is $A$-bilinear, so the category $\smod_A$ of $A$-super\-modules is enriched over itself. The category $\smod_A$ is not an abelian category, though the underlying even subcategory $\fsmod_A := (\smod_A)_\ev$, having all of the same objects but only the even $A$-linear maps as morphisms, is an abelian category.

\subsection{Supermatrices} \label{subsection:supermatrices}

In this paragraph, let $A \in \salg_k$ be an arbitrary $k$-superalgebra. Given $m,n \in \N$, let $\Amn = \kmn \otimes A$ denote the free $A$-supermodule with homogeneous basis $e_1,\ldots,e_{m+n}$ such that $\ol{e_i} = \zero$ if $1 \leq i \leq m$ and $\ol{e_i} = \one$ if $m+1 \leq i \leq m+n$. Then $\Hom_A(\Amn,\Amn)$ identifies with the ring $\Matmn(A)$ of all $(m+n) \times (m+n)$ matrices with coefficients in $A$. Specifically, given $T \in \Hom_A(\Amn,\Amn)$, write $T(e_j) = \sum_{i=1}^{m+n} e_i t_{ij}$ for some $t_{ij} \in A$. Then $T$ corresponds to the matrix whose $ij$-entry is $t_{ij}$, and (because the scalars were written to the right of the basis elements) composition of homomorphisms corresponds to matrix multiplication. Concretely, $\Matmn(A)_{\zero}$ identifies with the set of all block matrices
\begin{equation} \label{eq:blockform}
T = \begin{pmatrix} T_1 & T_2 \\ T_3 & T_4 \end{pmatrix}
\end{equation}
such that $T_1 \in M_{m \times m}(\Azero)$, $T_2 \in M_{m \times n}(\Aone)$, $T_3 \in M_{n \times m}(\Aone)$, and $T_4 \in M_{n \times n}(\Azero)$, while $\Matmn(A)_{\one}$ identifies with the set of all block matrices of the same shape but with the parities of the entries reversed. The ring $\Matmn(A)$ inherits a right $A$-supermodule structure via its identification with $\Hom_A(\Amn,\Amn)$. Specifically, if $T = (t_{ij}) \in \Matmn(A)$ and $a \in A$, then the $ij$-entry of $T \cdot a$ is equal to $(-1)^{\ol{a} \cdot \ol{e_j}} t_{ij} \cdot a$. In this way, the map
\begin{equation} \label{eq:Matmnbasechange}
\Matmn(k) \otimes A \rightarrow \Matmn(A): T \otimes a \mapsto T \cdot a
\end{equation}
is a ring isomorphism. We more frequently denote the ring $\Matmn(k)$ by $\glmn$.

\begin{remark} \label{remark:matrixtohom}
Given $A \in \salg_k$, there are canonical identifications (the first two as rings)
\begin{align*}
\Matmn(A) &\cong \Matmn(k) \otimes A \\
&\cong A \otimes \Matmn(k) \\
&\cong A \otimes (\Matmn(k)^\#)^\# \\
&\cong \Hom_k(\Matmn(k)^\#,A).
\end{align*}
Let $\set{e_{ij}: 1 \leq i,j \leq m+n}$ be the standard basis for $\Matmn(k)$ such that $e_{ij}$ is the matrix having a $1$ in the $ij$-position and $0$s elsewhere, and let $\set{E_{ij}: 1 \leq i,j \leq m+n}$ be the corresponding dual basis for $\Matmn(k)^\#$. Then $\ol{e_{ij}} = \ol{e_i} + \ol{e_j} = \ol{E_{ij}}$. Now if $T \in \Matmn(A)$ is the matrix whose $ij$ entry is $a$ and whose other entries are $0$, then the image of $T$ through the preceding series of identifications is the function $T: \Matmn(k)^\# \rightarrow A$ that maps $E_{ij}$ to $(-1)^{\ol{a}\cdot \ol{e_i} + \ol{e_i} + \ol{e_j}} a$, and that maps the other dual basis vectors for $\Matmn(k)^\#$ to $0$.
\end{remark}

\subsection{Superschemes}

Recall that an affine $k$-superscheme is a representable functor from the category $\csalg_k$ of commutative $k$-superalgebras to the category $\mathfrak{sets}$ of sets. Denote the coordinate (i.e., representing) superalgebra of an affine $k$-superscheme $X$ by $k[X]$. Then for each $A \in \csalg_k$, $X(A) = \Hom_{\salg}(k[X],A)$. Given an affine $k$-superscheme $X$, define the underlying purely even subscheme of $X$, denoted $X_{\ev}$, to be the closed subsuperscheme of $X$ such that
\[
k[X_{\ev}] = k[X]/ \subgrp{k[X]_{\one}},
\]
the largest purely even quotient of $k[X]$. Then $X_{\ev}(A) = X(\Azero)$ for each $A \in \csalg_k$.

The functor $\Matmn: A \mapsto \Matmn(A)$ for $A \in \csalg_k$ is naturally an affine $k$-superscheme. The coordinate super\-algebra $k[\Matmn]$ is the free commutative $k$-superalgebra generated for $1 \leq i,j \leq m+n$ by the even variables $X_{ij}$ and the odd variables $Y_{ij}$. Given $T \in \Hom_{\salg}(k[\Mmn],A)$, we identify $T$ with the matrix whose $ij$-entry is $T(X_{ij}) + T(Y_{ij})$. With this identification, the functors $A \mapsto \Matmn(A)_{\zero}$ and $A \mapsto \Matmn(A)_{\one}$ are closed affine subsuperschemes of $\Matmn$. Specifically,
\begin{align}
k[\Matmn(-)_{\zero}] &= k[\Matmn] / \subgrp{ \{X_{ij}: \ol{e_i}+\ol{e_j} = \one \} \cup \{ Y_{ij} : \ol{e_i} + \ol{e_j} = \zero \} }, & \text{and} \label{eq:Matmn0relations} \\
k[\Matmn(-)_{\one}] &= k[\Matmn] / \subgrp{ \{X_{ij}: \ol{e_i}+\ol{e_j} = \zero \} \cup \{ Y_{ij} : \ol{e_i} + \ol{e_j} = \one \} }. \label{eq:Matmn1relations}
\end{align}
Equivalently, $k[\Matmn(-)_{\zero}]$ is the free commutative $k$-superalgebra generated by the even variables $X_{ij}$ such that $\ol{e_i}+\ol{e_j} = \zero$ and the odd variables $Y_{ij}$ such that $\ol{e_i} + \ol{e_j} = \one$.

\section{Some multiparameter supergroups}\label{S:multiparametersupergroups}

In this section we introduce certain affine $k$-supergroup schemes that we refer to collectively as `multiparameter supergroups.' In contrast to the `one-parameter subgroups' defined in \cite{Suslin:1997}, whose elements are specified by only a single parameter in a commutative $k$-algebra $A$, these affine $k$-super\-group schemes typically require multiple parameters in a commutative $k$-superalgebra $A$ to completely specify their elements. Throughout this section, fix a positive integer $r \geq 1$.

\subsection{Definitions}

Define $k[\Mr]$ to be the commutative $k$-super\-algebra generated by the odd element $\tau$ and the even elements $\theta$ and $\sigma_i$ for $i \in \N$, such that $\tau^2 = 0$ (this is automatic because $k[\Mr]$ is commutative and $p = \chr(k) \neq 2$), $\sigma_0 = 1$, $\theta^{p^{r-1}} = \sigma_1$, and $\sigma_i \sigma_j = \binom{i+j}{i} \sigma_{i+j}$ for $i,j \in \N$:
\[ \textstyle
k[\Mr] = k[\tau,\theta,\sigma_1,\sigma_2,\ldots]/\subgrp{ \tau^2=0, \theta^{p^{r-1}} = \sigma_1, \text{ and } \sigma_i\sigma_j = \binom{i+j}{i}\sigma_{i+j} \text{ for $i,j \in \N$}}.
\]
If $j = \sum_{i=0}^\ell j_i p^i$ is the $p$-adic decomposition of $j$ (so $0 \leq j_i < p$), then it immediately follows that
\begin{equation} \label{eq:sigmafactorization}
\sigma_j = \frac{(\sigma_1)^{j_0} (\sigma_p)^{j_1} \cdots (\sigma_{p^\ell})^{j_\ell}}{(j_0!) (j_1)! \cdots (j_\ell)!},
\end{equation}
and $(\sigma_i)^p = 0$ for all $i \geq 1$.

\begin{lemma} \label{lemma:MrHopf}
The superalgebra $k[\M_r]$ admits a Hopf superalgebra structure such that its coproduct $\Delta$ and antipode $S$ satisfy
	\begin{align*}
	\Delta(\tau) &= \tau \otimes 1 + 1 \otimes \tau, & S(\tau) &= -\tau, \\
	\Delta(\theta) &= \theta \otimes 1 + 1 \otimes \theta, & S(\theta) &= -\theta, \\
	\Delta(\sigma_i) &= \textstyle \sum_{u+v=i} \sigma_u \otimes \sigma_v + \sum_{u+v+p=i} \sigma_u \tau \otimes \sigma_v \tau, & S(\sigma_i) &= (-1)^i \sigma_i.
	\end{align*}
\end{lemma}

\begin{proof}
It suffices to show that the stated formulas for the coproduct and the antipode are compatible with the defining relations in $k[\Mr]$. We will show for $i,j \in \N$ that $\Delta(\sigma_i\sigma_j) = \Delta(\sigma_i) \Delta(\sigma_j)$, but will leave the other (more straightforward) verifications to the reader.

Set $\ol{\Delta}(\sigma_i) = \sum_{u+v=i} \sigma_u \otimes \sigma_v$. Then $\Delta(\sigma_i) = \ol{\Delta}(\sigma_i) + \ol{\Delta}(\sigma_{i-p}) \cdot (\tau \otimes \tau)$, where $\sigma_{i-p}$ is interpreted to be $0$ if $i-p < 0$. Moreover, it follows from Vandermonde's convolution identity for binomial coefficients that $\ol{\Delta}(\sigma_i\sigma_j) = \ol{\Delta}(\sigma_i) \ol{\Delta}(\sigma_j)$. Now write $i = i_0 +i_1 p$ and $j = j_0 + j_1p$ for some $i_1,j_1 \in \N$ and $0 \leq i_0,j_0 < p$. Then
\[
\Delta(\sigma_i\sigma_j) = \textstyle \binom{i+j}{i} \Delta(\sigma_{i+j}) = \textstyle \binom{i+j}{i} \left[ \ol{\Delta}(\sigma_{i+j}) + \ol{\Delta}(\sigma_{i+j-p}) \cdot (\tau \otimes \tau) \right],
\]
and
\begin{align*}
\Delta(\sigma_i) \Delta(\sigma_j) &= \left[ \ol{\Delta}(\sigma_i) + \ol{\Delta}(\sigma_{i-p}) \cdot (\tau \otimes \tau) \right] \cdot \left[ \ol{\Delta}(\sigma_j) + \ol{\Delta}(\sigma_{j-p}) \cdot (\tau \otimes \tau) \right] \\
&= \ol{\Delta}(\sigma_i \sigma_j) + \ol{\Delta}(\sigma_i \sigma_{j-p}) \cdot (\tau \otimes \tau) + \ol{\Delta}(\sigma_{i-p} \sigma_j) \cdot (\tau \otimes \tau) \\
&= \textstyle \binom{i+j}{i} \ol{\Delta}(\sigma_{i+j}) + \left[ \binom{i+j-p}{i} + \binom{i+j-p}{i-p} \right] \ol{\Delta}(\sigma_{i+j-p}) \cdot (\tau \otimes \tau).
\end{align*}
It follows from Pascal's Identity and from Lucas's Theorem for binomial coefficients modulo $p$ that $\binom{i+j}{i} \equiv \binom{i+j-p}{i} + \binom{i+j-p}{i-p} \mod p$, so this implies that $\Delta(\sigma_i\sigma_j) = \Delta(\sigma_i)\Delta(\sigma_j)$.
\end{proof}

\begin{definition}
For $r,s \geq 1$, define $k[\Mrs]$ to be the Hopf subsuperalgebra of $k[\Mr]$ generated by $\tau$, $\theta$, and $\sigma_i$ for $1 \leq i < p^s$, and define $\Mr$ and $\Mrs$ to be the affine $k$-supergroup schemes whose coordinate Hopf superalgebras are $k[\Mr]$ and $k[\Mrs]$, respectively.
\end{definition}

Fix $A \in \csalg_k$. Specifying an element $g \in \Hom_{\salg}(k[\Mr],A)$ is equivalent to specifying the images in $A$ of the defining generators of $k[\Mr]$. Then by abuse of notation we will identify $\Mr(A)$ with the set of all infinite sequences $(\tau,\theta,\sigma_1,\sigma_2,\ldots)$ such that $\tau \in \Aone$, $\theta \in \Azero$, and $\sigma_i \in \Azero$ for $i \geq 1$, and such that the elements of the sequence satisfy the defining relations of $k[\Mr]$. Similarly, the elements of $\Mrs(A)$ identify with sequences of the form $(\tau,\theta,\sigma_1,\ldots,\sigma_{p^s-1})$. The group operation on $\Mr(A)$ is then given by
\begin{equation} \label{eq:Mrproduct}
(\tau,\theta,\sigma_1,\sigma_2,\ldots) * (\tau',\theta',\sigma_1',\sigma_2',\ldots) = (\tau+\tau',\theta+\theta',\sigma_1'',\sigma_2'',\ldots),
\end{equation}
where $\sigma_i '' := \sum_{u+v = i} \sigma_u \sigma_v' + \sum_{u+v+p=i} \sigma_u \tau \sigma_v' \tau'$, and the quotient homomorphism $\Mr(A) \twoheadrightarrow \Mrs(A)$ corresponding to the inclusion $k[\Mrs] \hookrightarrow k[\Mr]$ is given by $(\tau,\theta,\sigma_1,\sigma_2,\ldots) \mapsto (\tau,\theta,\sigma_1,\ldots,\sigma_{p^s-1})$. Since $\tau \tau' = - \tau'\tau$, $\Mr$ is not abelian, nor is $\Mrs$ except when $s=1$.

Recall that the (purely even) additive supergroup scheme $\Ga = \Ga^+$ is the affine supergroup scheme whose coordinate Hopf superalgebra $k[\Ga]$ is the polynomial algebra $k[\theta]$ generated by an even primitive element $\theta$. The odd additive supergroup scheme $\Gam$ is the affine supergroup scheme whose coordinate Hopf superalgebra $k[\Gam]$ is the exterior algebra $\Lambda(\tau)$ generated by an odd primitive element $\tau$. Thus for $A \in \csalg_k$, one has $\Ga(A) = (A_{\zero},+)$ and $\Gam(A) = (A_{\one},+)$.

\begin{remark} \label{remark:immediateMrs}
The following observations concerning $\Mr$ and $\Mrs$ are immediate.
\begin{enumerate}
\item The set $\set{\theta^i \sigma_j, \tau \theta^i \sigma_j : 0 \leq i < p^{r-1}, j \in \N}$ is a homogeneous basis for $k[\Mr]$.

\item The set $\set{\theta^i \sigma_j, \tau \theta^i \sigma_j : 0 \leq i < p^{r-1}, 0 \leq j < p^s}$ is a homogeneous basis for $k[\Mrs]$.

\item \label{item:Zgrading} The $\Z_2$-gradings on the Hopf superalgebras $k[\Mr]$ and $k[\Mrs]$ lift to $\Z$-gradings such that $\deg(\theta) = 2$, $\deg(\sigma_i) = 2ip^{r-1}$, and $\deg(\tau) = p^r$.

\item \label{item:MrfactorsMrs} Let $G$ be an algebraic $k$-supergroup scheme i.e., $G$ is an affine $k$-supergroup scheme such that $k[G]$ is finitely generated over $k$. Then any homomorphism of $k$-super\-group schemes $\rho: \Mr \rightarrow G$ factors through $\Mrs$ for $s \gg 0$. Indeed, by the finite-generation of $k[G]$, the image of the comorphism $\rho^*: k[G] \rightarrow k[\Mr]$ is contained in $k[\Mrs]$ for $s \gg 0$.

\item \label{item:superF} The function $\bsF^*: k[\Mr] \rightarrow k[\M_{r+1}]$ defined on generators by $\bsF^*(\theta) = \theta^p$, $\bsF^*(\tau) = \tau$, and $\bsF^*(\sigma_i) = \sigma_i$ for $i \in \N$ is a Hopf superalgebra homomorphism. (The assignment $\theta \mapsto \theta^p$ appears to be nonlinear, but this is only an optical illusion because the two $\theta$s are in different rings.) The corresponding supergroup scheme homomorphism $\bsF: \M_{r+1} \rightarrow \Mr$ satisfies
\[
\bsF(\tau,\theta,\sigma_1,\sigma_2,\ldots) = (\tau,\theta^p,\sigma_1,\sigma_2,\ldots)
\]
for $(\tau,\theta,\sigma_1,\sigma_2,\ldots) \in \M_{r+1}(A)$ and $A \in \csalg_k$. Thus, $\ker(\bsF)$ identifies with $\G_{a(1)}$, the first Frobenius kernel of $\Ga$. We call $\bsF$ the \emph{super Frobenius morphism} on $\Mr$. By abuse of notation, we will write $\bsF^\ell : \M_{r+\ell} \rightarrow \Mr$ for the function whose comorphism $k[\Mr] \rightarrow k[\M_{r+\ell}]$ is defined on generators by $\theta \mapsto \theta^{p^\ell}$, $\tau \mapsto \tau$, and $\sigma_i \mapsto \sigma_i$. Then $\ker(\bsF^\ell) \cong \G_{a(\ell)}$.

The super Frobenius morphism $\bsF: \M_{r+1} \rightarrow \Mr$ induces for each $s \geq 1$ a homomorphism $\M_{r+1;s} \rightarrow \Mrs$, which by abuse of notation we also denote $\bsF$. Then $\bsF$ is compatible in the obvious sense with the quotient homomorphisms $\Mrt \twoheadrightarrow \Mrs$ for $1 \leq s < t$.

\item \label{item:Garquotient} The natural inclusion $k[\G_{a(r)}] \cong k[\theta]/\subgrp{\theta^{p^r}} \hookrightarrow k[\Mr]$ is a Hopf superalgebra homomorphism. Write $q: \Mr \rightarrow \G_{a(r)}$, $(\tau,\theta,\sigma_1,\sigma_2,\ldots) \mapsto \theta$, for the corresponding quotient homomorphism. Then $\bsF: \M_{r+1} \rightarrow \Mr$ is related to the ordinary Frobenius morphism $F: \G_{a(r+1)} \rightarrow \G_{a(r)}$ by $F \circ q = q \circ \bsF$. The quotient map $q: \Mr \rightarrow \G_{a(r)}$ factors for each $s \geq 1$ through $\Mrs$. By abuse of notation we will also denote the induced quotient $\Mrs \twoheadrightarrow \G_{a(r)}$ by $q$.

\item \label{item:Ga-quotient} Similarly, the inclusion $k[\Gam] = k[\tau]/\subgrp{\tau^2} \hookrightarrow k[\Mr]$ is a Hopf superalgebra homomorphism, and hence gives rise to quotient homomorphisms $q^-: \Mr \twoheadrightarrow \Gam$ and $q^-: \Mrs \twoheadrightarrow \Gam$.

\item $k[\Mrone] \cong k[\theta]/\subgrp{\theta^{p^r}} \otimes \Lambda(\tau)$ as Hopf superalgebras, so $\Mrone \cong \Gar \times \Gam$.

\item \label{item:augmentationnilpotent} The augmentation ideal $I_\ve$ of $k[\Mr]$ is locally nilpotent, since $x^{p^r} = 0$ for all $x \in I_\ve$, but it is not globally nilpotent, since for example $\sigma_1 \sigma_p \cdots \sigma_{p^n} \neq 0$ for each $n \in \N$. However, using the explicit basis for $k[\Mr]$ one can check that $\bigcap_{n \geq 1} (I_\ve)^n = \set{0}$. The augmentation ideal of $k[\Mrs]$ is nilpotent, so $\Mrs$ is an infinitesimal supergroup scheme of height $r$.

\end{enumerate}
\end{remark}

Given an affine $k$-supergroup scheme $G$, set
	\[
	kG = k[G]^\# = \Hom_k(k[G],k).
	\]
The bialgebra structure of $k[G]$ induces via duality a $k$-superalgebra structure on $kG$. With this superalgebra structure, we call $kG$ the \emph{group algebra} of $G$. If $\phi: G \rightarrow H$ is a homomorphism of affine $k$-super\-group schemes, then composition with the comorphism $\phi^*: k[H] \rightarrow k[G]$ induces an algebra map $kG \rightarrow kH$, which by abuse of notation we also call $\phi$. If $G$ is finite, then by duality $kG$ inherits a Hopf superalgebra structure from $k[G]$, but in general the multiplication map on $k[G]$ does not induce a coproduct on $kG$.\footnote{If $G$ is not algebraic, i.e., if $k[G]$ is not a finitely-generated $k$-superalgebra, then the algebra of distributions $\Dist(G)$ need not be a Hopf superalgebra either, the first author's unqualified assertion in \cite[\S4.2]{Drupieski:2013b} notwithstanding.}

\begin{proposition} \label{prop:groupalgebras}
Let $r,s \geq 1$.
\begin{enumerate}
\item \label{item:groupalgebraiso} As $k$-superalgebras,
	\begin{align*}
	k\Mr &\cong k[[u_0,\ldots,u_{r-1},v]]/\subgrp{u_0^p,\ldots,u_{r-2}^p,u_{r-1}^p + v^2}, \quad \text{and} \\
	k\Mrs &\cong k[u_0,\ldots,u_{r-1},v]/\subgrp{u_0^p,\ldots,u_{r-2}^p,u_{r-1}^p + v^2,u_{r-1}^{p^s}},
	\end{align*}
where $\ol{u_i} = \zero$ for each $i$ and $\ol{v} = \one$. Furthermore, the inclusion $k[\Mrs] \hookrightarrow k[\Mr]$ induces the evident quotient map $k\Mr \twoheadrightarrow k\Mrs$ with kernel generated by $u_{r-1}^{p^s}$.

\item \label{item:gammai} Set $\gamma_{p^r} = u_{r-1}^p$, and given $0 \leq i < p^r$ with $p$-adic decomposition $i = \sum_{\ell=0}^{r-1} i_\ell p^\ell$, set
\[
\gamma_i = \frac{(u_0)^{i_0} (u_1)^{i_1} \cdots (u_{r-1})^{i_{r-1}}}{(i_0!)(i_1!) \cdots (i_{r-1}!)}.
\]
Given $j \in \N$, write $j = ap^r + b$ for some $a,b \in \N$ with $0 \leq b < p^r$, and set $\gamma_j = \gamma_b \cdot (\gamma_{p^r})^a$. Then $\set{\gamma_j,v \cdot \gamma_j: 0 \leq j < p^{r+s-1}}$ is a homogeneous basis for $k\Mrs$.

\item \label{item:coproduct} The coproduct $\Delta$ on $k\Mrs$ satisfies
\[
\Delta(\gamma_i) = \sum_{u+v = i} \gamma_u \otimes \gamma_v \quad \text{for} \quad 0 \leq i < p^r,
\]
$\Delta(u_{r-1}^p) = u_{r-1}^p \otimes 1 + 1 \otimes u_{r-1}^p$, and $\Delta(v) = v \otimes 1 + 1 \otimes v$.

\item \label{item:qGar} The homomorphism $q: k\Mr \rightarrow k\Gar$ induced by the quotient $q: \Mr \twoheadrightarrow \Gar$ identifies with the canonical quotient map $k\Mr \twoheadrightarrow k\Mr / \subgrp{v}$. Similarly, the map $q^-: k\Mr \rightarrow k\Gam$ induced by $q^-: \Mr \twoheadrightarrow \Gam$ identifies with the canonical map $k\Mr \twoheadrightarrow k\Mr/\subgrp{u_0,\ldots,u_{r-1}}$.

\item \label{item:bsF} The homomorphism $\bsF: k\M_{r+1} \rightarrow k\Mr$ induced by the super Frobenius morphism satisfies $\bsF(v) = v$, $\bsF(u_i) = u_{i-1}$ for $1 \leq i \leq r$, and $\bsF(u_0) = 0$.

\item The `polynomial subalgebra' $\Pr$ of $k\Mr$,
	\[
	\Pr = k[u_0,\ldots,u_{r-1},v]/\subgrp{u_0^p,\ldots,u_{r-2}^p,u_{r-1}^p + v^2},
	\]
	is a Hopf superalgebra, with coproducts given by the same formulas as in part \eqref{item:coproduct}, and with homogeneous basis $\set{\gamma_j,v \cdot \gamma_j : j \in \N}$. The maps $q$, $q^-$, and $\bsF$ of parts \eqref{item:qGar} and \eqref{item:bsF} restrict to Hopf superalgebra homomorphisms on $\Pr$ and $\P_{r+1}$, respectively.
\end{enumerate}
\end{proposition}

\begin{proof}
Recall the distinguished homogeneous basis $\set{\theta^i \sigma_j, \tau \theta^i \sigma_j : 0 \leq i < p^{r-1}, j \in \N}$ for $k[\Mr]$. Restricting $j$ to the range $0 \leq j < p^s$, we get the distinguished homogeneous basis for $k[\Mrs]$. For $0 \leq i \leq r-1$, define $u_i \in k[\Mr]^\#$ to be the functional that is linearly dual to the basis element $\theta^{p^i}$ (so in particular, $u_{r-1}$ is dual to $\theta^{p^{r-1}} = \sigma_1$), and define $v \in k[\Mr]^\#$ to be the functional that is linearly dual to $\tau$. By restriction, we also consider each $u_i$ and $v$ as an element of $k[\Mrs]^\#$. Now given an integer $0 \leq i < p^{r-1}$ with $p$-adic decomposition $i = \sum_{\ell = 0}^{r-2} i_\ell p^\ell$, and given $j \geq 0$, set
\[
\wt{\gamma}_{i+jp^{r-1}} = \frac{(u_0)^{i_0} (u_1)^{i_1} \cdots (u_{r-2})^{i_{r-2}}}{(i_0!)(i_1!) \cdots (i_{r-2}!)} \cdot (u_{r-1})^j \in k\Mr.
\]
Then one can check that the monomials $\wt{\gamma}_{i+jp^{r-1}}$ and $v \cdot \wt{\gamma}_{i+jp^{r-1}}$ in $k\Mr$ are linearly dual to $\theta^i \sigma_j$ and $\tau \theta^i \sigma_j$, respectively, and hence each $\phi \in k[\Mr]^\#$ can be uniquely expressed as a formal infinite linear combination of the monomials $\wt{\gamma}_{i+jp^{r-1}}$ and $v \cdot \wt{\gamma}_{i+jp^{r-1}}$ for $0 \leq i < p^{r-1}$ and $j \geq 0$. Similarly, it follows that the monomials $\gamma_j$ and $v \cdot \gamma_j$ for $0 \leq j < p^{r+s-1}$ form a basis for $k\Mrs$. Thus, the isomorphisms in part \eqref{item:groupalgebraiso} are true at the level of $k$-superspaces. The algebras $k\Mr$ and $k\Mrs$ are commutative in the ordinary sense because $k[\Mr]$ and $k[\Mrs]$ are cocommutative in the ordinary sense, and the other indicated algebra relations in $k\Mr$ and $k\Mrs$ are straightforward to verify. For example,
	\begin{align*} \textstyle
	(v^2)(\sigma_p) &= (v \otimes v) \circ \Delta(\sigma_p) \\
	&= \textstyle (v \otimes v)( \tau \otimes \tau + \sum_{i+j = p} \sigma_i 	\otimes \sigma_j ) \\
	&= (-1)^{\ol{v} \cdot \ol{\tau}} v(\tau) \otimes v(\tau) = -1,
	\end{align*}
and one can check that $v^2$ vanishes on all other distinguished basis elements for $k[\Mr]$, so that $v^2 = -u_{r-1}^p$. The statements regarding $q$ and $q^-$ are then evident from the identifications of $k[\Gar]$ and $k[\Gam]$ with the subalgebras of $k[\Mr]$ generated by $\theta$ and $\tau$, respectively, and the description of the map $\bsF: k\M_{r+1} \rightarrow k\Mr$ is similarly evident from the definition of $\bsF^*$ in Remark \ref{remark:immediateMrs}\eqref{item:superF}. The assertions concerning coproducts are straightforward to verify from the multiplicative structure in $k[\Mr]$. In particular the formula for $\Delta(u_{r-1}^p)$ follows from the coproduct formula for $\gamma_{p^{r-1}} = u_{r-1}$ and from the fact that $u_0,\ldots,u_{r-2}$ are each $p$-nilpotent.
\end{proof}

\begin{remark} \ 
\begin{enumerate}
\item By Remark \ref{remark:immediateMrs}\eqref{item:Zgrading}, the $\Z_2$-grading on $k[\Mr]$ lifts to a $\Z$-grading, which makes $k[\Mr]$ into a graded Hopf algebra of finite type in the sense of Milnor and Moore \cite{Milnor:1965}. Then $\Pr$ is the graded dual of $k[\Mr]$. In particular, $\Pr$ inherits by duality the structure of a Hopf superalgebra by \cite[Proposition 4.8]{Milnor:1965}.

\item The algebra of distributions $\Dist(\Mr)$ is defined by
	\[
	\Dist(\Mr) = \{ \phi \in k[\Mr]^\# : \phi(I_\ve^{n+1}) \text{ for some $n \geq 0$} \}.
	\]
Given an integer $j \geq 0$, let $s(j)$ denote the sum of the digits in the $p$-adic decomposition of $j$. By the relation $u_{r-1}^p + v^2 = 0$ in $k\Mr$, we may assume that each power series in $k\Mr$ is written so that $v$ never appears with an exponent greater than $1$. Then a power series $f \in k\Mr$ belongs to $\Dist(\M_r)$ if and only if there exists an integer $n(f) \geq 1$ such that no monomial involving $u_{r-1}^j$ with $s(j) \geq n(f)$ appears in $f$.
\end{enumerate}
\end{remark}

Recall that a $p$-polynomial $f = \sum_{i=0}^t a_i T^{p^i} \in k[T]$ is inseparable if and only if $a_0 = 0$. Then for each inseparable $p$-polynomial $f \in k[T]$ and each $\eta \in k$, $f(u_{r-1}) + \eta \cdot u_0$ is a primitive element of the Hopf superalgebra $\Pr$, and hence generates a Hopf ideal in $\Pr$.

\begin{definition}\label{Def:anotherdef}
Given an inseparable $p$-polynomial $0 \neq f \in k[T]$, $\eta \in k$, and $r \geq 1$, set
\[
k\Mrfeta = \Pr/\subgrp{f(u_{r-1}) + \eta u_0}.
\]
Then $k\Mrfeta$ inherits from $\Pr$ the structure of a (finite-dimensional) Hopf $k$-superalgebra. Define $\Mrfeta$ to be the finite $k$-supergroup scheme such that $k[\Mrfeta]^\# = k\Mrfeta$. For $\eta = 0$, set $k\Mrf = k\M_{r;f,0}$ and $k[\Mrf] = k[\M_{r;f,0}]$, so that
\[
k\Mrf = \Pr / \subgrp{f(u_{r-1})}.
\]
If $s \geq 1$ and $f = T^{p^s}$, then $k\Mrf = k\Mrs$, and hence $\Mrf = \Mrs$.
\end{definition}

\begin{lemma} \label{lemma:k[Mrf]algebra}
Let $r \geq 1$, let $0 \neq f = \sum_{i=1}^t a_i T^{p^i} \in k[T]$ be an inseparable $p$-polynomial of degree $p^t$, and let $\eta \in k$. Then $k\Mrfeta \cong k\Mrt$ as $k$-supercoalgebras, and hence $k[\Mrfeta] \cong k[\Mrt]$ as $k$-superalgebras. In particular, $k[\Mrfeta]$ is infinitesimal of height $r$.
\end{lemma}

\begin{proof}
Since $u_{r-1}^{p^t} = -(a_t^{-1}) \cdot (\eta u_0 + \sum_{i=1}^{t-1} a_i u_{r-1}^{p^i})$ in $k\Mrfeta$, it follows that the set
	\[
	\set{u_0^{i_0} \cdots u_{r-2}^{i_{r-2}} u_{r-1}^{i_{r-1}} v^j : 0 \leq i_\ell < p \text{ for $0\leq \ell \leq r-2$}; 0 \leq i_{r-1} < p^t; 0 \leq j \leq 1}
	\]
is a homogeneous basis for $k\Mrfeta$. Since the set of monomials of this form is also a homogeneous basis for $k\Mrt$, and since the coalgebra structures on $k\Mrfeta$ and $k\Mrt$ are both inherited from the coalgebra structure on $\Pr$, it follows that $k\Mrfeta \cong k\Mrt$ as supercoalgebras, and hence that $k[\Mrfeta] \cong k[\Mrt]$ as $k$-superalgebras.
\end{proof}

\begin{remark} \label{remark:immediateMrf}
Retain the assumptions of Lemma \ref{lemma:k[Mrf]algebra}. Assume that $a_i = 0$ for $i < s$ and $a_s \neq 0$.
\begin{enumerate}
\item \label{item:FMrfeta} The super Frobenius morphism induces a Hopf superalgebra map $\bsF : k\M_{r+1;f,\eta} \rightarrow k\Mrf$, and hence a supergroup homomorphism $\bsF: \M_{r+1;f,\eta} \rightarrow \Mrf$. Taking $\eta = 0$, we get maps $\bsF: k\M_{r+1;f} \rightarrow k\Mrf$ and $\bsF: \M_{r+1;f} \rightarrow \Mrf$. Similarly, the map $q^-: \Pr \twoheadrightarrow k\Gam$ factors through $k\Mrfeta$.

\item \label{item:MrftoMrs} By the assumption that $a_i = 0$ for $i < s$, one has $\subgrp{f(u_{r-1})} \subseteq \subgrp{u_{r-1}^{p^s}}$ as ideals in $\Pr$, so there exists a canonical quotient homomorphism $\pi: k\Mrf \twoheadrightarrow k\Mrs$. This quotient is compatible with the super Frobenius morphism in the sense that $\pi \circ \bsF = \bsF \circ \pi$. Composing $\pi$ with the quotient map $q: k\Mrs \twoheadrightarrow \Gar$, one gets a quotient $q: k\Mrf \twoheadrightarrow k\Gar$, and hence a group homomorphism $q: \Mrf \twoheadrightarrow \Gar$. Then as in Remark \ref{remark:immediateMrs}\eqref{item:Garquotient}, one gets $q \circ \bsF = F \circ q : \M_{r+1;f} \rightarrow \Gar$. For $\eta \in k$, one gets a quotient $q: \M_{r+1;f,\eta} \twoheadrightarrow \Gar$ by composing the super Frobenius morphism $\bsF : \M_{r+1;f,\eta} \rightarrow \Mrf$ with $q: \Mrf \twoheadrightarrow \Gar$.

\item The quotient $\pi: k\Mrf \twoheadrightarrow k\Mrs$ corresponds to an injective Hopf super\-algebra map $\pi^*: k[\Mrs] \hookrightarrow k[\Mrf]$; via the identification $k[\Mrf] = k[\Mrt]$ of Lemma \ref{lemma:k[Mrf]algebra}, this is just the inclusion of algebras $k[\Mrs] \subseteq k[\Mrt]$. Thus, the coproducts in $k[\Mrf]$ of $\tau$, $\theta$, and $\sigma_i$ for $0 \leq i < p^s$ are given by the same formulas as in Lemma \ref{lemma:MrHopf}. For $p^s \leq i < p^t$, the coproduct in $k[\Mrf]$ of $\sigma_i$ can be more complicated; cf.\ Lemma \ref{lemma:Mrfetacoproduct} below.

\item \label{item:Mrfetaiso} Suppose $\eta \neq 0$. Then $k\M_{r+1;f,\eta} \cong k\M_{r;f^p} = \Pr/\subgrp{[f(u_{r-1})]^p}$ as superalgebras. Explicitly, the isomorphism $\phi: k\M_{r+1;f,\eta} \simrightarrow k\M_{r;f^p}$ is given on generators by $v \mapsto v$, $u_i \mapsto u_{i-1}$ for $1 \leq i \leq r$, and $u_0 \mapsto (-\eta^{-1}) \cdot f(u_{r-1})$. 
Then $\bsF \circ \phi^{-1}: k\M_{r;f^p} \rightarrow k\Mrf$ identifies with the canonical quotient map $\Pr/\subgrp{[f(u_{r-1})]^p} \twoheadrightarrow \Pr/\subgrp{f(u_{r-1})}$. Making the superalgebra identifications $k[\Mrf] \cong k[\Mrt]$ and $k[\M_{r;f^p}] \cong k[\M_{r;t+1}]$ as in Lemma \ref{lemma:k[Mrf]algebra}, it follows that $(\bsF \circ \phi^{-1})^*: k[\Mrf] \hookrightarrow k[\M_{r;f^p}]$ identifies with the subalgebra inclusion $k[\Mrt] \hookrightarrow k[\M_{r;t+1}]$. The map $\bsF \circ \phi^{-1}$ is also compatible with the quotient $q: k\Mrf \twoheadrightarrow k\Gar$ in the sense that $q \circ \bsF \circ \phi^{-1} = q : k\M_{r;f^p} \twoheadrightarrow k\Gar$.

\item Let $A \in \csalg_k$. Since $k[\Mrf] \cong k[\Mrt]$, elements of $\Mrf(A)$ identify with sequences $(\tau,\theta,\sigma_1,\ldots,\sigma_{p^t-1})$ such that $\tau \in \Aone$, $\theta \in \Azero$, and $\sigma_i \in \Azero$ for $1 \leq i < p^t$, and such that the elements of the sequence satisfy the defining relations of $k[\Mrt]$. The group multiplication map in $\Mrf(A)$ is in general no longer given by \eqref{eq:Mrproduct} (namely, the formula for $\sigma_i''$ can be more complicated if $i \geq p^s$), though the quotient $\pi: \Mrf(A) \twoheadrightarrow \Mrs(A)$ is still given by truncation, the map $q: \Mrf(A) \rightarrow \Gar(A)$ is still given by projection onto $\theta$, and the Frobenius morphism $\bsF: \M_{r+1;f}(A) \rightarrow \Mrf(A)$ is still given as in Remark \ref{remark:immediateMrs}\eqref{item:superF}.
\end{enumerate}
\end{remark}

\begin{lemma} \label{lemma:Mrfetacoproduct}
Let $0 \neq f = T^{p^t} + \sum_{i=1}^{t-1} a_i T^{p^i} \in k[T]$ be an inseparable $p$-polynomial, and let $\eta \in k$. Set $a_0 = \eta$ and $a_t = 1$, so that $k\Monefeta \cong k[u,v]/\subgrp{u^p+v^2, \sum_{i=0}^t a_i u^{p^i}}$. Then under the identification $k[\Monefeta] \cong k[\Monet]$ of Lemma \ref{lemma:k[Mrf]algebra}, the coproduct on $k[\Monefeta]$ satisfies $\Delta(\tau) = \tau \otimes 1 + 1 \otimes \tau$ and
	\begin{align}
	\Delta(\sigma_\ell) &= \sum_{i+j = \ell} \sigma_i \otimes \sigma_j + \sum_{i+j+p = \ell} \sigma_i \tau \otimes \sigma_j \tau \label{eq:coproduct} \\
	&\relphantom{=} + \sum_{\substack{0 \leq c < t \\ i+j \geq p^t \\ i+j+p^c = \ell + p^t}} (-a_c) \cdot \sigma_i \otimes \sigma_j + \sum_{\substack{0 \leq c < t \\ i+j \geq p^t \\ i+j+p+p^c = \ell + p^t}} (-a_c) \cdot \sigma_i \tau \otimes \sigma_j \tau \nonumber \\
	&\relphantom{=} + \sum_{\substack{0 \leq c,d < t \\ i+j \geq p^t \\ i+j+p^c \geq 2p^t \\ i+j+p^c + p^d = \ell + 2p^t}} (a_c a_d) \cdot \sigma_i \otimes \sigma_j + \sum_{\substack{0 \leq c,d < t \\ i+j \geq p^t \\ i+j+p^c \geq 2p^t \\ i+j+p^c+p^d+p=\ell+2p^t}} (a_ca_d) \cdot \sigma_i \tau \otimes \sigma_j \tau \nonumber
	\end{align}
for $0 \leq \ell < p^t$ and $t \geq 2$. If $t = 1$, then $\Delta(\tau)$ and $\Delta(\sigma_\ell)$ are given by the same formulas as above, except that $\Delta(\sigma_1)$ includes the additional term $(-a_0^3) \cdot \sigma_{p-1} \tau \otimes \sigma_{p-1}\tau$.
\end{lemma}

\begin{proof}
The sets $\set{u^\ell, u^\ell v : 0 \leq \ell < p^t}$ and $\set{\sigma_\ell, \sigma_\ell \tau: 0 \leq \ell < p^t}$ are homogeneous bases for $k\Monefeta$ and $k[\Monet]$, respectively; we call these the standard homogeneous bases for $k\Monefeta$ and $k[\Monet]$. By definition, $k[\Monefeta]$ is the unique (up to isomorphism) $k$-Hopf super\-algebra such that $k[\Monefeta]^\# = k\Monefeta$. Then $k[\Monefeta] \cong k[\Monefeta]^{\#\#} = (k\Monefeta)^\#$. Now applying the identification $k[\Monefeta] \cong k[\Monet]$ of Lemma \ref{lemma:k[Mrf]algebra}, and viewing the elements of $k[\Monet]$ as elements of $(k\Monefeta)^\#$, one has $\sigma_i(u^j) = \delta_{ij}$, $\sigma_i(u^j v) = 0$, $(\sigma_i \tau)(u^j) = 0$, and $(\sigma_i \tau)(u^j v) = -\delta_{ij}$ for $0 \leq i,j < p^t$, where $\delta_{ij}$ is the usual Kronecker symbol. Furthermore, making the identifications
\[
k[\Monefeta] \otimes k[\Monefeta] \cong (k\Monefeta)^\# \otimes (k\Monefeta)^\# \cong (k\Monefeta \otimes k\Monefeta)^\#,
\]
the coproduct on $k[\Monefeta]$ identifies with the map $\Delta: (k\Monefeta)^\# \rightarrow (k\Monefeta \otimes k\Monefeta)^\#$ such that $\Delta(\phi)(x \otimes y) = \phi(xy)$. Then to prove the formula for $\Delta(\sigma_\ell)$, it suffices by duality to determine, for $0 \leq i,j < p^t$, the coefficient of $u^\ell$ in $u^i \cdot u^j = u^{i+j}$ and in $(u^iv) \cdot (u^jv) = -u^{i+j+p}$ when these products are rewritten in terms of the standard homogeneous basis for $k\Monefeta$.

Fix $0 \leq i,j < p^t$. We first consider the standard basis monomials $u^\ell$ that can occur in $u^i \cdot u^j$.
	\begin{enumerate}
	\item $i+j < p^t$. Then $u^i \cdot u^j = u^{i+j}$ is a standard basis monomial, so $u^\ell$ occurs in $u^{i+j}$ if and only if $\ell = i+j$, and if $u^\ell$ does occur, it does so with coefficient $1$.
	\item $i+j \geq p^t$. Set $m = (i+j)-p^t$. Then $0 \leq m \leq p^t -2$, and
		\[
		u^i \cdot u^j = u^{p^t} \cdot u^m = \left(-\sum_{c=0}^{t-1} a_c u^{p^c} \right) \cdot u^m = - \sum_{c=0}^{t-1} a_c u^{p^c+ m}.
		\]
		\begin{enumerate}
		\item If $0 \leq c < t$ and $\ell = p^c + m < p^t$, then we get a contribution of $-a_c \cdot u^\ell$ in $u^i \cdot u^j$.
		\item Suppose $0 \leq c < t$ and $p^c + m \geq p^t$. Set $n_c = (p^c + m) - p^t$. Then
			\[
			u^{p^c+m} = u^{p^t} \cdot u^{n_c} = \left(-\sum_{d=0}^{t-1} a_d u^{p^d} \right) \cdot u^{n_c} = -\sum_{d=0}^{t-1} a_d u^{p^d+n_c}.
			\]
		Now $0 \leq n_c \leq p^c-2$, so for $0 \leq d < t$ one has $p^d + n_c \leq p^d + p^c - 2 \leq 2p^{t-1} - 2 < p^t$. Then it follows for $0 \leq d < t$ that we get a contribution in $u^i \cdot u^j$ of
		\[
		(-a_c) \cdot ((-a_d) \cdot u^{p^d+n_c}) = (a_c a_d) \cdot u^{p^d+p^c + (i+j)-2p^t}.
		\]
		\end{enumerate}
	\end{enumerate}
Now we consider the standard basis monomials that appear when $(u^i v) \cdot (u^j v) = -u^{i+j+p}$ is written as a linear combination of standard basis elements.
	\begin{enumerate}
	\setcounter{enumi}{2}
	\item $i+j+p < p^t$. Then $(u^iv) \cdot (u^jv) = -u^{i+j+p}$ is a standard basis monomial.

	\item $i+j+p \geq p^t$. Set $m = i+j+p - p^t$. Then $0 \leq m \leq p^t+p-2$, and
		\[
		(u^iv) \cdot (u^jv) = (- u^{p^t}) \cdot u^m = \left( \sum_{c=0}^{t-1} a_c u^{p^c} \right) \cdot u^m = \sum_{c=0}^{t-1} a_c u^{p^c+m}.
		\]
		\begin{enumerate}
		\item If $0 \leq c < t$ and $\ell = p^c+m < p^t$, then we get a contribution of $a_c \cdot u^\ell$ in $(u^i v) \cdot (u^j v)$.
		
		\item Suppose $0 \leq c < t$ and $p^c + m \geq p^t$. Set $n_c = p^c + m - p^t$. Then as above,
			\[
			u^{p^c+m} = -\sum_{d=0}^{t-1} a_d u^{p^d+n_c}.
			\]
		Now $0 \leq n_c \leq p^c + p - 2$, so for $0 \leq d < t$ one has
			\[
			p^d + n_c \leq p^d + p^c + p-2 \leq 2p^{t-1} + (p-2).
			\]
				\begin{enumerate}
				\item $t \geq 2$. Then $2p^{t-1} + (p-2) < p^t$, so it follows for $0 \leq d < t$ that we get a contribution in $(u^i v) \cdot (u^jv)$ of
					\[
					a_c \cdot ( (-a_d) \cdot u^{p^d+n_c}) = -(a_ca_d) \cdot u^{p^d+p^c+p+(i+j)-2p^t}.
					\]
				\item $t = 1$. Then $m = i+j$, $n_c = n_0 = m-(p-1)$, and $p^d + n_c = p^0+n_0 =m-(p-2)$. Then $p^d+n_c < p = p^t$ except when $i = j = p-1$, in which case $p^d+n_c = p$. So for $i = j = p-1$ we get a contribution in $(u^{p-1}v) \cdot (u^{p-1}v)$ of
					\[
					(-a_0^2) \cdot u^{p^d+n_c} = (-a_0^2) \cdot u^p = (a_0^3) \cdot u.
					\]
				\end{enumerate}
		\end{enumerate}
	\end{enumerate}
The previous calculations imply by duality the formula for $\Delta(\sigma_\ell)$. (Note that $\sigma_i \tau \otimes \sigma_j \tau$, viewed as an element of $(k\Monefeta \otimes \Monefeta)^\#$, evaluates to $-1$ on $u^i v \otimes u^j v$ due to the sign conventions of super linear algebra.) Similarly, $\Delta(\tau) = \tau \otimes 1 + 1 \otimes \tau$ because the only products of standard basis monomials in $k\Monefeta$ that include $v$ as a summand when rewritten in terms of the standard basis are $v \cdot 1$ and $1 \cdot v$.
\end{proof}

\subsection{Cohomology} \label{subsection:cohomology}

Recall from \cite[\S5.1]{Drupieski:2013b} that if $G$ is an affine $k$-supergroup scheme, and if $M$ is a rational $G$-supermodule, then the cohomology group $\Hbul(G,M)$ can be computed as the cohomology of the Hochschild complex $C^\bullet(G,M) := M \otimes k[G]^{\otimes \bullet}$. Given a cocycle $z \in C^n(G,M)$, we will write $[z]$ to denote the cohomology class of $z$ in $\opH^n(G,M)$. Recall also that the cohomology ring $\Hbul(G,k)$ is a graded superalgebra in the sense of \cite[\S2.2]{Drupieski:2016a}, i.e., it admits both a $\Z$-grading (the cohomological degree) and a $\Z_2$-grading (the internal superdegree, induced by the $\Z_2$-grading on $k[G]$). Given a homogeneous element $a \in \Hbul(G,k)$, we write $\deg(a)$ for the $\Z$-degree of $a$, and write $\ol{a}$ for the $\Z_2$-degree of $a$. If $G$ is finite, then it follows from \cite[Corollary 2.3.6]{Drupieski:2016a} that $\Hbul(G,k)$ is graded-commutative superalgebra in the sense that if $a,b \in \Hbul(G,k)$ are both homogeneous, then
\[
a \cdot b = (-1)^{\deg(a) \cdot \deg(b) + \ol{a} \cdot \ol{b}} b \cdot a.
\]

\begin{proposition} \label{prop:Mrcohomology}
Let $r \geq 1$, let $0 \neq f \in k[T]$ be an inseparable $p$-polynomial, and let $\eta \in k$.
\begin{enumerate}
\item \label{item:Mr1} As a $k$-superalgebra, $\Hbul(\Mrone,k)$ is generated for $1 \leq i \leq r$ by the cohomology classes
\[
x_i = \left[ \textstyle \sum_{j=1}^{p-1} \frac{(p-1)!}{j!(p-j)!} (\theta^{p^{i-1}})^j \otimes (\theta^{p^{i-1}})^{p-j} \right], \quad
\lambda_i = [\theta^{p^{i-1}}], \quad \text{and} \quad
y = [\tau],
\]
where the bracketed expressions are interpreted as cocycles in the Hochschild complex for $\Mrone$, subject only to the condition that $\Hbul(\Mrone,k)$ is graded-commutative in the sense of \cite[Definition 2.2.1]{Drupieski:2016a}. In other words,
\[
\Hbul(\Mrone,k) \cong k[x_1,\ldots,x_r,y] \gotimes \Lambda(\lambda_1,\ldots,\lambda_r),
\]
where $x_i \in \opH^2(\Mrone,k)_{\zero}$, $\lambda_i \in \opH^1(\Mrone,k)_{\zero}$, $y \in \opH^1(\Mrone,k)_{\one}$.

\item \label{item:Mrquotient} The quotient map $\Mr \twoheadrightarrow \Mrone$ induces a surjective algebra homomorphism $\Hbul(\Mrone,k) \twoheadrightarrow \Hbul(\Mr,k)$ with kernel generated by $x_r-y^2$. In other words,
\[
\Hbul(\Mr,k) \cong k[x_1,\ldots,x_r,y]/\subgrp{x_r-y^2} \gotimes \Lambda(\lambda_1,\ldots,\lambda_r),
\]
where $x_i$, $\lambda_i$, and $y$ are defined by the same formulas as in \eqref{item:Mr1}, interpreting the bracketed expressions as cocycles in the Hochschild complex for $\Mr$.

\item \label{item:Mrs} Suppose $s \geq 2$. Then the quotient map $\Mr \twoheadrightarrow \Mrs$ induces a surjective algebra homomorphism $\Hbul(\Mrs,k) \twoheadrightarrow \Hbul(\Mr,k)$ with kernel generated by the cohomology class
\[ \textstyle
w_s := -\left[ \sum_{j=1}^{p^s-1} \sigma_j \otimes \sigma_{p^s - j} + \sum_{u+v+p = p^s} \sigma_u \tau \otimes \sigma_v \tau \right].
\]
Specifically,
\[
\Hbul(\Mrs,k) \cong k[x_1,\ldots,x_r,y,w_s]/\subgrp{x_r-y^2} \gotimes \Lambda(\lambda_1,\ldots,\lambda_r),
\]
where $w_s \in \opH^2(\Mrs,k)_{\zero}$, and $x_i$, $\lambda_i$, and $y$ are defined by the same formulas as in (1), interpreting the bracketed expressions as cocycles in the Hochschild complex for $\Mrs$.

\item \label{item:Mrfcohomology} Suppose $f = \sum_{i=s}^t a_i T^{p^i} \in k[T]$ with $1 \leq s \leq t$ and $a_s,a_t \neq 0$. Then the inclusion $\pi^*: k[\Mrs] \hookrightarrow k[\Mrf]$ induces an isomorphism in cohomology $\Hbul(\Mrs,k) \cong \Hbul(\Mrf,k)$.

\item \label{item:Mrfetacohomology} Suppose $\eta \neq 0$. Then the algebra isomorphism $k\M_{r+1;f,\eta} \cong k\M_{r;f^p}$ of Remark \ref{remark:immediateMrf}(\ref{item:Mrfetaiso}) induces an isomorphism in cohomology $\Hbul(\M_{r+1;f,\eta},k) \cong \Hbul(\M_{r;f^p},k)$.

\item \label{item:M1fetacohomology} Suppose $\eta \neq 0$. Then the inclusion $(q^-)^*: k[\Gam] = k[\tau]/\subgrp{\tau^2} \hookrightarrow k[\Monefeta]$ induces an isomorphism in cohomology $k[y] \cong \Hbul(\Gam,k) \cong \Hbul(\Monefeta,k)$.
\end{enumerate}
\end{proposition}

\begin{proof}
Fix $r,s \geq 1$, and define $\Nrs$ to be the closed subsupergroup scheme of $\Mr$ with
\[
k[\Nrs] = k[\Mr]/\subgrp{\tau,\theta,\sigma_1,\ldots,\sigma_{p^s-1}}.
\]
Then $\Mr/\Nrs = \Mrs$. For $i \in \N$, let $\wtsigma_i$ denote the image in $k[\Nrs]$ of the generator $\sigma_{ip^s} \in k[\Mr]$. Then $k[\Nrs]$ is generated as a commutative algebra by the elements $\wtsigma_i$ for $i \in \N$ subject only to the relation $\wtsigma_i \wtsigma_j = \binom{i+j}{i} \wtsigma_{i+j}$ for $i,j \in \N$. In particular, the set $\set{\wtsigma_i: i \in \N}$ is a basis for $k[\Nrs]$. Since $\tau$ is an element in the defining ideal of $\Nrs$, it follows that $\Nrs$ is central, hence normal, in $\Mr$. Then as in \cite[\S5.1]{Drupieski:2013b}, we can consider the Lyndon--Hochschild--Serre spectral sequence
\begin{equation} \label{eq:LHSMrs}
E_2^{i,j} = \opH^i(\Mr/\Nrs, \opH^j(\Nrs,k)) \Rightarrow \opH^{i+j}(\Mr,k).
\end{equation}
Since $\Nrs$ is central in $\Mr$, the action of $\Mr/\Nrs$ on the cohomology ring $\Hbul(\Nrs,k)$ is trivial, and hence $E_2^{i,j} \cong \opH^i(\Mr/\Nrs,k) \otimes \opH^j(\Nrs,k) = \opH^i(\Mrs,k) \otimes \opH^j(\Nrs,k)$.

Observe that $k[\Mrone]$ identifies with the tensor product of superalgebras $k[\theta]/\subgrp{\theta^{p^r}} \otimes \Lambda(\tau)$. Then $\Mrone \cong \G_{a(r)} \times \Gam$, and it follows from the K\"{u}nneth isomorphism that
\[
\Hbul(\Mrone,k) \cong \Hbul(\G_{a(r)},k) \gotimes \Hbul(\Gam,k).
\]
Applying the explicit calculation of $\Hbul(\G_{a(r)},k)$ given in \cite[I.4.20--I.4.27]{Jantzen:2003}, and the explicit (classical) calculation of $\Hbul(\Gam,k) = \Hbul(\Lambda(\tau),k)$ in \cite{Priddy:1970}, it follows that $\Hbul(\Mrone,k)$ admits the explicit description given in part \eqref{item:Mr1} of the proposition.

Next, using the fact that $\set{\wtsigma_i: i \in \N}$ is a basis for $k[\Nrs]$, one can check that the Hochschild complex $C^\bullet(\Nrs,k)$ identifies with the dual of the bar construction of a one-variable polynomial algebra, as described in \cite{Priddy:1970}. Then by \cite[Theorem 2.5]{Priddy:1970}, the cohomology ring $\Hbul(\Nrs,k)$ is an exterior algebra generated by the cohomology class of the cocycle $\wtsigma_1 \in C^1(\Nrs,k)$. In particular, $\opH^j(\Nrs,k) = 0$ for $j \geq 2$, so $E_2^{i,j} = 0$ for $j \geq 2$, and the only nonzero differential of the spectral sequence \eqref{eq:LHSMrs} is the map $d_2: E_2^{\bullet,1} \rightarrow E_2^{\bullet+2,0}$. Fixing basis elements $1 \in \opH^0(\Nrs,k)$ and $[\wtsigma_1] \in \opH^1(\Nrs,k)$, we identify $\opH^0(\Nrs,k)$ and $\opH^1(\Nrs,k)$ with the field $k$, and hence identify the differential with a linear map $\Hbul(\Mrs,k) \cong E_2^{\bullet,1} \rightarrow E_2^{\bullet+2,0} \cong \opH^{\bullet+2}(\Mrs,k)$. Then by the derivation property of the differential, $d_2$ acts on $\Hbul(\Mrs,k)$ as multiplication by $d_2([\wtsigma_1])$, and $\Hbul(\Mr,k) = \ker(d_2)/\im(d_2)$.

The differential $d_2: E_2^{0,1} \rightarrow E_2^{2,0}$ is induced by the Hochschild differential $\partial$ on $C^\bullet(\Mr,k)$, and 
\[
\partial(\sigma_{p^s}) = - \left( \sum_{j=1}^{p^s-1} \sigma_j \otimes \sigma_{p^s-j} + \sum_{u+v+p = p^s} \sigma_u \tau \otimes \sigma_v \tau \right).
\]
Since $\sigma_j = \frac{1}{j!}(\sigma_1)^j = \frac{1}{j!}(\theta^{p^{r-1}})^j$ for $1 \leq j \leq p-1$, and since $(p-1)! \equiv -1 \mod p$, we get in the case $s=1$ that $d_2([\wtsigma_1]) = x_r - y^2$. But $x_r - y^2$ is not a zero divisor in $\Hbul(\Mrone,k)$ by part \eqref{item:Mr1} of the proposition, so $\Hbul(\Mr,k) \cong \Hbul(\Mrone,k)/\subgrp{x_r - y^2}$, proving \eqref{item:Mrquotient}. Now suppose $s \geq 2$. By inspection, the cocycle representatives described in part \eqref{item:Mrquotient} for the generators of $\Hbul(\Mr,k)$ all pull back to cocycles in the Hochschild complex for $\Mrs$. Then the map in cohomology induced by the quotient $\Mr \twoheadrightarrow \Mrs$ is a surjection. Equivalently, the horizontal edge map of the spectral sequence \eqref{eq:LHSMrs} is a surjection. This implies that $E_\infty^{\bullet,1} = 0$, and hence that the differential $d_2$ is an injection. In particular, $d_2([\wtsigma_1])$ is not a zero divisor in $\Hbul(\Mrs,k)$. Setting $w_s = d_2([\wtsigma_1])$, \eqref{item:Mrs} then follows.

Recall that, via the equivalence between rational $G$-supermodules and $kG$-supermodules, the cohomology ring of a finite supergroup scheme $G$ identifies with the cohomology ring of its group ring $kG$ \cite[Remark 5.1.1]{Drupieski:2013b}. Then \eqref{item:Mrfetacohomology} is immediate, and we can prove \eqref{item:Mrfcohomology} by proving the equivalent dual statement that the quotient map $k\Mrf \twoheadrightarrow k\Mrs$ induces an isomorphism in cohomology. Indeed, let $A$ be the Hopf subalgebra of $k\Mrf$ generated by the (primitive) element $u_{r-1}^{p^s}$. Then the Hopf superalgebra quotient $k\Mrf // A$ identifies with $k\Mrs$, and one has a Lyndon--Hochschild--Serre spectral sequence
\[
E_2^{i,j} = \opH^i(k\Mrs,\opH^j(A,k)) \Rightarrow \opH^{i+j}(k\Mrf,k).
\]
Set $g = a_s T + a_{s+1} T^p + \cdots + a_t T^{p^{t-s}}$. Then $g$ is a separable polynomial, i.e., the roots of $g$ are distinct in any algebraic closure of $k$, and $A \cong k[x]/\subgrp{g(x)}$, so after scalar extension to a large enough field, the algebra $A$ splits as a direct product of fields. This implies that $\opH^j(A,k) = 0$ for $j > 0$, and hence that the spectral sequence collapses at the $E_2$-page, yielding for each $i \geq 0$ that the quotient map $k\Mrf \twoheadrightarrow k\Mrs$ induces an isomorphism $\opH^i(k\Mrs,k) \simrightarrow \opH^i(k\Mrf,k)$.

The proof of \eqref{item:M1fetacohomology} is essentially the same as that of \eqref{item:Mrfcohomology}, letting $A$ instead be the Hopf subalgebra of $k\Monefeta$ generated by $u_0$, so that $k\Monefeta//A \cong k\Gam$, and setting $g = f(T) - \eta T$.
\end{proof}

\begin{remark} \label{remark:w1}
Interpreting the expression $-[(\sum_{j=1}^{p-1} \sigma_j \otimes \sigma_{p - j}) + \tau \otimes \tau]$ as defining an element $w_1 \in \Hbul(\Mrone,k)$, one gets $w_1 = x_1-y^2$.
\end{remark}

Retain the notations and conventions of Lemma \ref{lemma:Mrfetacoproduct}. Let $I_\ve$ denote the augmentation ideal of $k[\Monefeta]$. Then the powers of $I_\ve$ define a decreasing filtration $F^\bullet$ on $k[\Monefeta]$:
\[
k[\Monefeta] = (I_\ve)^0 \supseteq I_\ve \supseteq (I_\ve)^2 \supseteq (I_\ve)^3 \supseteq \cdots.
\]
In particular, if $j = \sum_{\ell \geq 0} j_\ell p^\ell$ is the $p$-adic decomposition of $j$, then $\sigma_j \in (I_\ve)^i$ if and only if $\sum_{\ell \geq 0} j_\ell \geq i$; cf.\ \eqref{eq:sigmafactorization}. The decreasing filtration on $k[\Monefeta]$ induces a decreasing filtration $F^\bullet$ on $k[\Monefeta]^{\otimes 2}$ with $F^i(k[\Monefeta]^{\otimes 2}) = \sum_{i_1+i_2 \geq i} (I_\ve)^{i_1} \otimes (I_\ve)^{i_2}$. The following lemma will be applied in the proof of Proposition \ref{proposition:erMrs}.

\begin{lemma} \label{lemma:boundarymodfiltration}
Retain the notations and conventions of the preceding paragraph. Let $\partial$ denote the differential of the Hochschild complex $C^\bullet(\Monefeta,k)$, so that $-\partial(\sigma_p) = \Delta(\sigma_p) - (\sigma_p \otimes 1 + 1 \otimes \sigma_p)$. Then
\[
-\partial(\sigma_p) \equiv \left(\sum_{j=1}^{p-1} \sigma_j \otimes \sigma_{p-j}\right) + \tau \otimes \tau + \left( \sum_{j=1}^{p-1} (-a_1) \cdot \sigma_{jp^{t-1}} \otimes \sigma_{(p-j)p^{t-1}} \right) \mod F^{p+1} (k[\Monefeta]^{\otimes 2}).
\]
\end{lemma}

\begin{proof}
The coproduct $\Delta(\sigma_p)$ is described explicitly by the formula in Lemma \ref{lemma:Mrfetacoproduct}. Then it suffices to check that the only summands of $\Delta(\sigma_p)$ that are not elements of $F^{p+1} (k[\Monefeta]^{\otimes 2})$ are $\sigma_p \otimes 1$, $1 \otimes \sigma_p$, and the summands included in the statement of this lemma.

For example, suppose $0 \leq i,j < p^t$ are such that $i+j \geq p^t$ and $i+j+p^c = p + p^t$ for some $0 \leq c < t$. Let $i = \sum_{\ell = 0}^{t-1} i_\ell p^\ell$ and $j = \sum_{\ell = 0}^{t-1} j_\ell p^\ell$ be the $p$-adic decompositions of $i$ and $j$. If $\sum_{\ell = 0}^{t-1} (i_\ell+j_\ell) < p$, then $i_\ell+j_\ell < p$ for each $0 \leq \ell < t$, and hence $\sum_{\ell=0}^{t-1} (i_\ell+j_\ell)p^\ell$ is the $p$-adic decomposition of $i+j$. In particular, if $\sum_{\ell=0}^{t-1} (i_\ell+j_\ell) < p$, then $i+j < p^t$. But $i+j \geq p^t$ by assumption, so we must have $\sum_{\ell=0}^{t-1} (i_\ell+j_\ell) \geq p$. Now we consider several cases:
	\begin{enumerate}
	\item $t = 1$. Then $c = 0$ and $i+j+1 = 2p$. This is impossible, because $0 \leq i,j < p$.

	\item $t \geq 2$. Then $p+p^t$ is the $p$-adic decomposition of $i+j+p^c$. Using this observation, one can check that the only way to get $i+j+p^c = p+p^t$ with $\sum_{\ell = 0}^{t-1} (i_\ell + j_\ell) = p$, i.e., with $\sigma_i \otimes \sigma_j \notin F^{p+1}(k[\Monefeta]^{\otimes 2})$, is to have $c = 1$ and $i_1 + j_1 = p$.
	\end{enumerate}
The preceding calculations imply that the only terms in the first summation in the second row of \eqref{eq:coproduct} that are not elements of $F^{p+1}(k[\Monefeta]^{\otimes 2})$ are the terms included in the statement of this lemma. We leave the remaining details of this proof to the reader.
\end{proof}

Recall the (now classical) calculation of the cohomology ring for $\Gar$:
	\[
	\Hbul(\Gar,k) \cong k[x_1,\ldots,x_r] \gotimes \Lambda(\lambda_1,\ldots,\Lambda_r),
	\]
with $\deg(x_i) = 2$ and $\deg(\lambda_i) = 1$ for each $i$ (for details, see \cite[I.4]{Jantzen:2003}).

\begin{lemma} \label{lemma:Fandqoncohomology}
Let $r \geq 1$, and let $f = \sum_{i=s}^t a_i T^{p^i} \in k[T]$ with $1 \leq s \leq t$ and $a_s,a_t \neq 0$. Then in terms of the generators described in Proposition \ref{prop:Mrcohomology}, and identifying $\Hbul(\Mrf,k)$ with $\Hbul(\Mrs,k)$ via the isomorphism of Proposition \ref{prop:Mrcohomology}(\ref{item:Mrfcohomology}):
\begin{enumerate}
\item \label{item:Foncohomology} The maps in cohomology induced by the super Frobenius morphism,
\[
\bsF^* : \Hbul(\Mr,k) \rightarrow \Hbul(\M_{r+1},k) \quad \text{and} \quad 
\bsF^* : \Hbul(\Mrf,k) \rightarrow \Hbul(\M_{r+1;f},k),
\]
satisfy $\bsF^*(x_i) = x_{i+1}$, $\bsF^*(\lambda_i) = \lambda_{i+1}$, $\bsF^*(y) = y$, and (in the second case) $\bsF^*(w_s) = w_s$.

\item \label{item:qoncohomology} The maps in cohomology induced by the quotients $q: \Mr \twoheadrightarrow \G_{a(r)}$ and $q: \Mrf \twoheadrightarrow \G_{a(r)}$,
\[
q^*: \Hbul(\G_{a(r)},k) \rightarrow \Hbul(\Mr,k) \quad \text{and} \quad q^*: \Hbul(\G_{a(r)},k) \rightarrow \Hbul(\Mrf,k),
\]
map the generators for $\Hbul(\G_{a(r)},k)$ to the generators of the same names in $\Hbul(\Mr,k)$ and $\Hbul(\Mrf,k)$, respectively. In particular, $q^*$ maps $\Hbul(\Gar,k)$ isomorphically onto the subalgebra of $\Hbul(\Mr,k)$ generated by $x_1,\ldots,x_r,\lambda_1,\ldots,\lambda_r$.

\item \label{item:squotient} Let $s \geq 1$. The map in cohomology $\pi^*: \Hbul(\Mrs,k) \rightarrow \Hbul(\M_{r;s+1},k)$ induced by the quotient $\pi: \M_{r;s+1} \twoheadrightarrow \Mrs$ satisfies $y \mapsto y$, $x_i \mapsto x_i$, and $\lambda_i \mapsto \lambda_i$. For $s = 1$ one has $\pi^*(w_1) = x_1 - y^2$, while for $s \geq 2$ one has $\pi^*(w_s) = 0$.
\end{enumerate}
\end{lemma}

\begin{proof}
For $f = T^{p^s}$, \eqref{item:Foncohomology} and \eqref{item:qoncohomology} are immediate from the explicit descriptions of the cohomology ring generators in Proposition \ref{prop:Mrcohomology}, and then for general $f$ they then follow from the compatibility conditions described in Remark \ref{remark:immediateMrf}\eqref{item:MrftoMrs}. Next, the effect of $\pi^*$ on $y$, $x_i$, and $\lambda_i$ is also evident from the explicit descriptions of the generators in terms of cochain representatives, so we will focus our remaining attention on computing $\pi^*(w_s) = 0$. The calculation $\pi^*(w_1) = x_1 - y^2$ follows from Remark \ref{remark:w1}, so suppose $s \geq 2$.

Recall from Remark \ref{remark:immediateMrs}\eqref{item:Zgrading} that $k[\Mrs]$ is a $\Z$-graded Hopf superalgebra. Then the cohomology ring $\Hbul(\Mrs,k)$ inherits  an additional internal $\Z$-grading, which we denote by $\deg$. Then $\deg(y) = p^r$, $\deg(\lambda_i) = 2p^{i-1}$, $\deg(x_i) = 2p^i$, and $\deg(w_s) = 2p^{r+s-1}$. Now for $s \geq 2$, one can check that the subspace of $\opH^2(\Mrs,k)$ of elements of internal degree $2p^{r+s-1}$ is one-dimensional spanned by $w_s$, while the subspace of $\opH^2(\M_{r;s+1},k)$ of elements of internal degree $2p^{r+s-1}$ is $0$. The homomorphism $\pi^*$ evidently preserves internal degrees, so $\pi^*(w_s) = 0$ for $s \geq 2$.
\end{proof}

\subsection{Superschemes of multiparameter supergroups} \label{SS:superschemeofMS}

Given an affine $k$-supergroup scheme $G \in \sgrp_k$ and a commutative $k$-superalgebra $A \in \csalg_k$, let $G \otimes_k A \in \sgrp_A$ denote the affine $A$-supergroup scheme obtained from $G$ via base change to $A$. That is, $G \otimes_k A$ is the $A$-supergroup scheme with coordinate Hopf $A$-superalgebra $A[G] := k[G] \otimes_k A$. Base change defines a functor $- \otimes_k A : \sgrp_k \rightarrow \sgrp_A$ from the category of affine $k$-supergroup schemes to the category of affine $A$-supergroup schemes. If $\phi: G \rightarrow G'$ is a homomorphism of affine $k$-supergroup schemes, then $\phi \otimes_k A: G \otimes_k A \rightarrow G' \otimes_k A$ is the homomorphism of $A$-supergroup schemes with comorphism defined by $(\phi \otimes_k A)^* := \phi^* \otimes_k A: k[G'] \otimes_k A \rightarrow k[G] \otimes_k A$.

\begin{definition}
Given affine $k$-supergroup schemes $G,G'$, define the $k$-superfunctor
\[
\bfHom(G,G') : \csalg_k \rightarrow \mathfrak{sets}
\]
by
\[
\bfHom(G,G')(A) = \Hom_{Grp/A}(G \otimes_k A, G' \otimes_k A),
\]
the set of $A$-supergroup scheme homomorphisms $\rho: G \otimes_k A \rightarrow G' \otimes_k A$.
\end{definition}

By the anti-equivalence between $A$-supergroup schemes and commutative Hopf $A$-superalgebras, an $A$-supergroup scheme homomorphism $\rho: G \otimes_k A \rightarrow G' \otimes_k A$ is equivalent to the data of a Hopf $A$-superalgebra homomorphism $\rho^*: A[G'] \rightarrow A[G]$.

\begin{lemma} \label{lemma:HomMrsMrs}
Let $r \geq 1$, and let $A = \Azero \in \calg_k$ be a purely even commutative $k$-superalgebra. Then there exists a natural identification
\[
\bfHom(\Mrone,\Mrone)(A) = \set{(\mu,a_0,\ldots,a_{r-1}) \in A^{r+1}}.
\]
More generally, for $s \geq 2$ there exists a natural inclusion
\[
\set{(\mu,a_0,\ldots,a_{r-1},b) \in A^{r+2}: \mu^2 = a_0^{p^r}} \subseteq \bfHom(\Mrs,\Mrs)(A),
\]
i.e., the elements of the former set naturally identify with elements of the latter set, and if $A$ is reduced (i.e., if $A$ has no nonzero nilpotent elements), then this inclusion is an equality.
\end{lemma}

\begin{proof}
Let $s \geq 1$, and let $\phi: A[\Mrs] \rightarrow A[\Mrs]$ be a Hopf $A$-superalgebra homomorphism. Since $\phi$ is by definition an even map, and since $A = \Azero$, we can write
\[
\phi(\tau) = \sum_{i=0}^{p^{r-1}-1} \sum_{j=0}^{p^s-1} \tau \theta^i \sigma_j \cdot \mu_{ij} \in A[\Mrs] = k[\Mrs] \otimes_k A
\]
for some $\mu_{ij} \in A$. Since $\phi$ is a Hopf superalgebra homomorphism, $\phi(\tau)$ is also primitive, which implies that $\mu_{ij} = 0$ unless $i = j = 0$. So $\phi(\tau) = \tau \cdot \mu$ for some $\mu \in A$. Similarly, $\phi$ must map $\theta$ to an even primitive element in $A[\Mrs]$, so $\phi(\theta) = \sum_{i = 0}^{r-1} \theta^{p^i} \cdot a_i$ for some $a_0,\ldots,a_{r-1} \in A$. Since $A[\Mrone]$ is generated as an $A$-algebra by $\tau$ and $\theta$, the scalars $\mu,a_0,\ldots,a_{r-1}$ completely determine $\phi$ if $s=1$, and conversely one can check that any such choice of scalars defines a Hopf superalgebra homomorphism $\phi: A[\Mrone] \rightarrow A[\Mrone]$. This completes the proof in the case $s=1$.

Next suppose that $s \geq 2$ and that $r=1$. The Hopf superalgebra homomorphism $\phi: A[\Mones] \rightarrow A[\Mones]$ defines by duality a Hopf superalgebra homomorphism $\phi^*: A\Mones \rightarrow A\Mones$ between the group algebras. Write $A\Mones = k\Mones \otimes_k A = A[u,v]/\subgrp{u^p+v^2,u^{p^s}}$ with $u$ (resp.\ $v$) an even (resp.\ odd) primitive element. As in the previous paragraph, $\phi^*$ must map $u$ and $v$ to primitive elements of the same parity, so $\phi^*(v) = v \cdot \mu$ and $\phi^*(u) = \sum_{i=0}^{s-1} u^{p^i} \cdot c_i$ for some scalars $\mu,c_0,\ldots,c_{s-1} \in A$. The relation $u^p+v^2 = 0$ implies that
	\[
	0 = \phi(u)^p + \phi(v)^2 = \left(\sum_{i=0}^{s-1} u^{p^{i+1}} \cdot c_i^p\right) + v^2  \cdot \mu^2 = u^p \cdot (c_0^p - \mu^2)  + \left( \sum_{i=1}^{s-2} u^{p^{i+1}} \cdot c_i^p \right),
	\]
and hence that $\mu^2 = c_0^p$ and $c_i^p = 0$ for $1 \leq i \leq s-2$. Conversely, if $\mu,c_0,\ldots,c_{s-1}\in A$ are any scalars such that $\mu^2 = c_0^p$ and such that $c_i^p = 0$ for $1 \leq i \leq s-2$, then one can check that the formulas $\phi^*(v) = v \cdot \mu$ and $\phi^*(u) = \sum_{i=0}^{s-1} u^{p^i} \cdot c_i$ define a Hopf superalgebra homomorphism $\phi^*: A\Mones \rightarrow A\Mones$. In the special case that $c_i = 0$ for $1 \leq i \leq s-2$ (e.g., when $A$ is reduced), set $a_0 = c_0$ and $b = c_{s-1}$. Then $\phi^*(v) = v \cdot \mu$ and $\phi^*(u) = u \cdot a_0 + u^{p^{s-1}} \cdot b$, and the map $\phi: A[\Mones] \rightarrow A[\Mones]$ corresponding via duality to $\phi^*$ satisfies the formulas $\phi(\tau) = \tau \cdot \mu$, $\phi(\sigma_{p^i}) = \sigma_{p^i} \cdot a_0^{p^i}$ for $0 \leq i \leq s-2$, and $\phi(\sigma_{p^{s-1}}) = \sigma_{p^{s-1}} \cdot a_0^{p^{s-1}} + \sigma_1 \cdot b$. This completes the proof in the case $r=1$.

Now let $r \geq 2$ and $s \geq 2$ be arbitrary, and let $\phi: A[\Mrs] \rightarrow A[\Mrs]$ be a Hopf superalgebra homomorphism. As in the first paragraph of the proof, we get $\phi(\tau) = \tau \cdot \mu$ and $\phi(\theta) = \sum_{i=0}^{r-1} \theta^{p^i} \cdot a_i$ for some scalars $\mu,a_0,\ldots,a_{r-1} \in A$. Next, identify $A[\Mones]$ with the Hopf subalgebra of $A[\Mrs]$ generated by $\tau$ and by $\sigma_i$ for $1 \leq i < p^s$. Under the hypothesis that $A$ is reduced, we will show that $\phi$ restricts to a Hopf superalgebra homomorphism $\wt{\phi}: A[\Mones] \rightarrow A[\Mones]$. This will help us, via the results of the second paragraph, to further describe $\phi$. Since $\phi(\tau) = \tau \cdot \mu \in A[\Mones]$, it suffices to show for $1 \leq i < p^s$ that $\phi(\sigma_i) \in A[\Mones]$. We argue by induction on $i$. First, since $\theta^{p^{r-1}} = \sigma_1$ and $(\sigma_1)^p = 0$, the relation $\phi(\theta) = \sum_{i=0}^{r-1} \theta^{p^i} \cdot a_i$ implies that $\phi(\sigma_1) = \sigma_1 \cdot a_0^{p^{r-1}}$. Then the algebra relations in $A[\Mrs]$ imply for $1 \leq i < p$ that $\phi(\sigma_i) \in A[\Mones]$. Now let $1 \leq n < s$, and suppose by induction that $\phi(\sigma_i) \in A[\Mones]$ for all $1 \leq i < p^n$. Since $\phi$ is a Hopf algebra homomorphism, then
\begin{align*}
\Delta \circ \phi(\sigma_{p^n}) &= (\phi \otimes \phi) \circ \Delta(\sigma_{p^n}) \\
&= \phi(\sigma_{p^n}) \otimes 1 + 1 \otimes \phi(\sigma_{p^n}) + \sum_{i=1}^{p^n-1} \phi(\sigma_i) \otimes \phi(\sigma_{p^n-i}) + \sum_{i+j+p = p^n} \phi(\sigma_i \tau) \otimes \phi(\sigma_j \tau).
\end{align*}
Recall from Remark \ref{remark:immediateMrs}\eqref{item:Zgrading} that the algebra $A[\Mrs]$ is $\Z$-graded, and the generators of $A[\Mones]$ are concentrated in $\Z$-degrees at least $2p^{r-1}$. Since the coproduct on $A[\Mrs]$ is injective (and preserves the $\Z$-grading), this observation combined with the above formula for $\Delta \circ \phi(\sigma_{p^n})$ and the induction hypothesis implies that any homogeneous summands of $\phi(\sigma_{p^n})$ of $\Z$-degree less than $2p^{r-1}$ must be primitive. In other words, $\phi(\sigma_{p^n}) = (\sum_{i=0}^{r-2} \theta^{p^i} \cdot c_i ) + z$ for some $c_0,\ldots,c_{r-2} \in A$ and some element $z$ in the augmentation ideal of $A[\Mones]$. But then $z^p = 0$ (cf.\ Remark \ref{remark:immediateMrs}\eqref{item:augmentationnilpotent}), and $(\sigma_{p^n})^p = 0$, so $0 = \phi(\sigma_{p^n})^p = \sum_{i=0}^{r-2} \theta^{p^{i+1}} \cdot c_i^p$. Thus, if $A$ is reduced, it must be the case that $c_0 = \cdots = c_{r-2} = 0$, and hence that $\phi(\sigma_{p^n}) \in A[\Mones]$. Then the induction hypothesis and the algebra relations in $A[\Mrs]$ imply that $\phi(\sigma_i) \in A[\Mones]$ for all $1 \leq i < p^{n+1}$. So by induction on $n$, $\phi(\sigma_i) \in A[\Mones]$ for all $1 \leq i < p^s$.

We have shown that if $A$ is reduced, then $\phi$ restricts to a Hopf superalgebra homomorphism $\wt{\phi}: A[\Mones] \rightarrow A[\Mones]$. Then by the case $r=1$, there exist $a,b \in A$ such that $\phi(\sigma_{p^i}) = \sigma_{p^i} \cdot a^{p^i}$ for $0 \leq i \leq s-2$, and $\phi(\sigma_{p^{s-1}}) = \sigma_{p^{s-1}} \cdot a^{p^{s-1}} + \sigma_1 \cdot b$. Since $\phi(\tau) = \tau \cdot \mu$, the case $r=1$ also implies that $\mu^2 = a^p$, while the relations $\theta^{p^{r-1}} = \sigma_1$ and $(\sigma_1)^p = 0$ imply that $a = a_0^{p^{r-1}}$. Thus, if $A$ is reduced and if $\phi: A[\Mrs] \rightarrow A[\Mrs]$ is a Hopf superalgebra homomorphism, then there exist scalars $\mu,a_0,\ldots,a_{r-1},b \in A$ such that $\mu^2 = a_0^{p^r}$, $\phi(\tau) = \tau \cdot \mu$, $\phi(\theta) = \sum_{i=0}^{r-1} \theta^{p^i} \cdot a_i$, $\phi(\sigma_{p^i}) = \sigma_{p^i} \cdot a_0^{p^{r+i-1}}$ for $0 \leq i \leq s-2$, and $\phi(\sigma_{p^{s-1}}) = \sigma_{p^{s-1}} \cdot a_0^{p^{r+s-2}} + \sigma_1 \cdot b$. Conversely, if $A$ is an arbitrary purely even commutative superalgebra, and if $\mu,a_0,\ldots,a_{r-1},b \in A$ are any scalars such that $\mu^2 = a_0^{p^r}$, then one can check that the preceding formulas define a Hopf superalgebra homomorphism $\phi: A[\Mrs] \rightarrow A[\Mrs]$.
\end{proof}

\begin{remark} \label{remark:toobig}
Let $A = \Azero \in \calg_k$ be an arbitrary (not necessarily reduced) purely even commutative superalgebra, and let $s \geq 2$. Implicit in the proof is an identification of sets
	\[
	\bfHom(\Mones,\Mones)(A) = \set{ (\mu,c_0,\ldots,c_{s-1}) \in A^{s+1}: \mu^2 = c_0^p \text{ and } c_i^p = 0 \text{ for } 1 \leq i \leq s-2},
	\]
but for $r \geq 2$ the set $\bfHom(\Mrs,\Mrs)(A)$ appears to be strictly larger than the subset identified in Lemma \ref{lemma:HomMrsMrs}. Specifically, when constructing for $r,s \geq 2$ a Hopf superalgebra homomorphism $\phi: A[\Mrs] \rightarrow A[\Mrs]$, it appears that $\phi(\sigma_p)$ can equal $\sigma_p \cdot (a_0)^{p^r} + \sum_{i=0}^{r-1} \theta^{p^i} \cdot b_i$ for any choice of scalars $b_0,\ldots,b_{r-1} \in A$, though the combinatorics of checking that this assignment works seem to be sufficiently complicated that we have been deterred from completing the verification. On the other hand, if $k$ is an algebraically closed field, then $\bfHom(\Mrs,\Mrs)(k)$ identifies as a set with the cohomology variety $\abs{\Mrs}$, i.e., the maximal ideal spectrum of the cohomology ring $\Hbul(\Mrs,k)$; cf.\ \cite[Lemma 1.10]{Suslin:1997} and \cite[Theorem 5.2]{Suslin:1997a}. The calculations of Lemma \ref{lemma:HomMrsMrs} are significantly extended in our subsequent work \cite{Drupieski:2017b}.
\end{remark}

\begin{definition} \label{definition:VrGLmn}
For $m,n \in \N$, $r \geq 1$, define the $k$-superfunctor $\bsV_r(\GLmn) : \csalg_k \rightarrow \mathfrak{sets}$ by
\begin{align*}
\bsV_r(\GLmn)(A) = \Big \{ (\alpha_0,\ldots,\alpha_{r-1},\beta) &\in (\Matmn(A)_{\zero})^{\times r} \times \Matmn(A)_{\one} : \\
&[\alpha_i,\alpha_j] = [\alpha_i,\beta] = 0 \text{ for all $0 \leq i,j \leq r-1$}, \\
&\alpha_i^p = 0 \text{ for all $0 \leq i \leq r-2$}, \text{ and } \alpha_{r-1}^p + \beta^2 = 0 \Big \}.
\end{align*}
Given an inseparable $p$-polynomial $0 \neq f(T) \in k[T]$ and given a scalar $\eta \in k$, let $\bsVrfeta(\GLmn)$ be the subsuperfunctor of $\bsV_r(\GLmn)$ defined by
\[
\bsVrfeta(\GLmn)(A) = \set{ (\alpha_0,\ldots,\alpha_{r-1},\beta) \in \bsV_r(\GLmn)(A) : f(\alpha_{r-1}) + \eta \alpha_0=0}.
\]
Finally, set $\bsVrf(\GLmn) = \bsV_{r;f,0}(\GLmn)$, and for $s \geq 1$ set $\bsVrs(\GLmn) = \bsV_{r;T^{p^s}}(\GLmn)$.
\end{definition}

\begin{proposition} \label{proposition:HomMrGLmn}
Let $m,n \in \N$, let $r \geq 1$, let $0 \neq f \in k[T]$ be an inseparable $p$-polynomial, and let $\eta \in k$. Then for each $A \in \csalg_k$, there exists a natural identification
\begin{gather*}
\bfHom(\Mrfeta,\GLmn)(A) = \bsVrfeta(\GLmn)(A), \quad \text{and} \\
\bfHom(\Mr,\GLmn)(A) = \set{ (\alpha_0,\ldots,\alpha_{r-1},\beta) \in \bsV_r(\GLmn) : \alpha_{r-1}^{p^s} = 0 \text{ for some $s \geq 1$} }.
\end{gather*}
Specifically, a tuple $(\ualpha|\beta) := (\alpha_0,\ldots,\alpha_{r-1},\beta) \in \bsVrfeta(\GLmn)(A)$ corresponds to the unique homo\-morphism $\rho_{(\ualpha|\beta)}: \Mrfeta \otimes_k A \rightarrow \GLmn \otimes_k A$ such that, under the induced action of the group algebra $A \Mrfeta := k[\Mrfeta] \otimes_k A$ on $\Amn$, $u_i \in k[\Mrfeta]$ acts via $\alpha_i$ and $v \in k[\Mrfeta]$ acts via $\beta$.
\end{proposition}

\begin{proof}
Let $G$ be a finite $k$-supergroup scheme, and let $AG := kG \otimes_k A$ be the group algebra of $G$ over $A$. Then to give a homomorphism of $A$-supergroup schemes $\rho: G \otimes_k A \rightarrow \GLmn \otimes_k A$ is equivalent to specifying a rational $G \otimes_k A$-supermodule structure on $\Amn$, which is in turn equivalent to making $\Amn$ into a left $AG$-supermodule; cf.\ \cite[I.8.6]{Jantzen:2003}. Now the first identification of the proposition is immediate from the structure of the group algebra $k\Mrfeta$, with an element $(\alpha_0,\ldots,\alpha_{r-1},\beta) \in \bsVrfeta(\GLmn)(A)$ corresponding to the $A\Mrfeta$-supermodule structure on $\Amn$ in which the algebra generators $u_i$ and $v$ act via the matrices $\alpha_i$ and $\beta$, respectively. Finally, since $\GLmn$ is an algebraic supergroup scheme, it follows as in Remark \ref{remark:immediateMrs}\eqref{item:MrfactorsMrs} that any $A$-supergroup scheme homomorphism $\rho: \Mr \otimes_k A \rightarrow \GLmn \otimes_k A$ factors through $\Mrs \otimes_k A$ for some $s \geq 1$. Then the second identification of the proposition follows from the first.
\end{proof}

\begin{theorem} \label{theorem:Homsuperscheme}
Let $r \geq 1$, let $0 \neq f \in k[T]$ be an inseparable $p$-polynomial, and let $\eta \in k$.
\begin{enumerate}
\item Let $m,n \in \N$. Then the $k$-superfunctor $\bsV_r(\GLmn)$ admits the structure of an affine superscheme of finite type over $k$.

\item For each algebraic $k$-supergroup scheme $G$, the functor $\bfHom(\Mrfeta,G)$ admits the structure of an affine superscheme of finite type over $k$. Then the assignment $G \mapsto \bfHom(\Mrfeta,G)$ is a covariant functor from the category of algebraic $k$-supergroup schemes to the category of affine superschemes of finite type over $k$ that takes closed embeddings to closed embeddings.
\end{enumerate}
\end{theorem}

\begin{proof}
Recall from Section \ref{subsection:commsuperrings} that the functors $\Matmn(-)_{\zero}$ and $\Matmn(-)_{\one}$ are closed subsuper\-schemes of $\Matmn$. The defining equations of $\bsV_r(\GLmn)$ are homo\-geneous polynomial equations on the entries of the component matrices, so $\bsV_r(\GLmn)$ is then a closed subsuperscheme of $(\Matmn)^{\times (r+1)}$, and $\bfHom(\Mrfeta,\GLmn)$ identifies by Proposition \ref{proposition:HomMrGLmn} with the closed sub\-super\-scheme defined by the additional homogeneous polynomial condition $f(\alpha_{r-1})+\eta \alpha_0 = 0$.

Let $A = k[\bfHom(\Mrfeta,\GLmn)]$ be the coordinate superalgebra of $\bfHom(\Mrfeta,\GLmn)$, considered as a closed subsuper\-scheme of $(\Matmn)^{\times (r+1)}$. Then $A$ is a quotient of $k[\Matmn]^{\otimes(r+1)}$. For $0 \leq \ell \leq r$, let $X_{ij}(\ell)$ and $Y_{ij}(\ell)$ denote the copies of the coordinate functions $X_{ij},Y_{ij} \in k[\Matmn]$ that live in the $(\ell+1)$-th tensor factor of $k[\Matmn]^{\otimes (r+1)}$. Now define $(\alpha_0,\ldots,\alpha_{r-1},\beta) \in \Matmn(A)^{\times (r+1)}$ such that the $ij$-entry of $\alpha_\ell$ is the image in $A$ of $X_{ij}(\ell) + Y_{ij}(\ell)$, and the $ij$-entry of $\beta$ is the image in $A$ of $X_{ij}(r) + Y_{ij}(r)$.\footnote{Given $0 \leq \ell \leq r$ and given $i$ and $j$, it is always the case that precisely one of $X_{ij}(\ell)$ or $Y_{ij}(\ell)$ has image equal to $0$ in $A$, by the definition of $\bfHom(\Mrfeta,\GLmn)$ as a sub-superscheme of $(\Matmn(-)_{\zero})^{\times r} \times \Matmn(-)_{\one}$.} Then via the identification of Proposition \ref{proposition:HomMrGLmn},
\[
(\ualpha|\beta) := (\alpha_0,\ldots,\alpha_{r-1},\beta) \in \bfHom(\Mrfeta,\GLmn)(A).
\]
If $B \in \csalg_k$ and if $\phi \in \Hom_{\salg}(A,B)$, then $\phi(\ualpha|\beta) \in \bfHom(\Mrfeta,\GLmn)(B)$. Conversely, each $\rho \in \bfHom(\Mrfeta,\GLmn)(B)$ is of the form $\phi(\ualpha|\beta)$ for some $\phi \in \Hom_{\salg}(A,B)$, namely, if $\rho$ corresponds to the tuple $(\ualpha'|\beta') \in \bsVrfeta(\GLmn)(B)$, then $\phi$ is the unique $k$-algebra homomorphism that maps $(\ualpha|\beta)$ to $(\ualpha'|\beta')$. Let $\rho_{(\ualpha|\beta)}: \Mrfeta \otimes_k A \rightarrow \GLmn \otimes_k A$ be the homomorphism of $A$-super\-group schemes corresponding to the tuple $(\ualpha|\beta)$ defined above; we call $\rho_{(\ualpha|\beta)}$ the \emph{universal supergroup homomorphism from $\Mrfeta$ to $\GLmn$}.

Let $G$ be an algebraic $k$-supergroup scheme. Since $k[G]$ is finitely generated over $k$, there exists a closed embedding $G \hookrightarrow \GLmn$ for some $m,n \in \N$ \cite[Theorem 9.3.2]{Westra:2009}. Choose homogeneous elements $F_1,\ldots,F_M \in k[\GLmn]$ such that $k[G] = k[\GLmn]/\subgrp{F_1,\ldots,F_M}$. Each $F_i$ defines a function $F_i : \GLmn(B) \rightarrow B$ for each $B \in \csalg_k$. More precisely, if $g \in \GLmn(B) = \Hom_{\salg}(k[\GLmn],B)$, then $F_i(g) := g(F_i) \in B$ is homogeneous of the same parity as $F_i$. Now consider the identity map $\id_{A[\Mrfeta]}$ as an element of $(\Mrfeta \otimes_k A)(A[\Mrfeta])$. Then
\[
\rho_{(\ualpha|\beta)}(\id_{A[\Mrfeta]}) \in \GLmn(A[\Mrfeta]) \subset \Matmn(A[\Mrfeta])_{\zero},
\]
and $\rho_{(\ualpha|\beta)}(\id_{A[\Mrfeta]}) \in G(A[\Mrfeta])$ if and only if $F_\ell ( \rho_{(\ualpha|\beta)}(\id_{A[\Mrfeta]}) ) = 0$ for each $1 \leq \ell \leq M$. Assume as in Remark \ref{remark:immediateMrf}\eqref{item:MrftoMrs} that $f = \sum_{i=s}^t a_i T^{p^i}$ with $a_s,a_t \neq 0$ and $1 \leq s \leq t$. Then
\[
F_\ell \left( \rho_{(\ualpha|\beta)}(\id_{A[\Mrfeta]}) \right) = \sum_{\substack{0 \leq i < p^{r-1} \\ 0 \leq j < p^t}} \left[ \theta^i \sigma_j \otimes F_{\ell,i,j} + \tau \theta^i \sigma_j \otimes F_{\ell,i,j}' \right] \in k[\Mrfeta] \otimes_k A = A[\Mrfeta]
\]
for some homogeneous elements $F_{\ell,i,j},F_{\ell,i,j}' \in A$. Denote by $I_{r;f,\eta}(G)$ the homogeneous ideal in $A$ generated by the $F_{\ell,i,j},F_{\ell,i,j}'$. Then it follows via naturality that $I_{r;f,\eta}(G)$ defines $\bfHom(\Mrfeta,G)$ as a closed subsuperscheme of $\bfHom(\Mrfeta,\GLmn)$.

Next, let $\phi: G' \rightarrow G$ be a homomorphism of algebraic $k$-supergroup schemes. Composition with $\phi$ defines a natural transformation of $k$-superfunctors
\[
\bfHom(\Mrfeta,\phi): \bfHom(\Mrfeta,G') \rightarrow \bfHom(\Mrfeta,G).
\]
More precisely, base change induces for each $B \in \csalg_k$ a homomorphism $\phi_B: G' \otimes_k B \rightarrow G \otimes_k B$, and $\bfHom(\Mrfeta,\phi)(B)$ is then defined by $\rho \mapsto \phi_B \circ \rho$. Now by Yoneda's Lemma, $\phi$ corresponds to a morphism of affine $k$-superschemes $\bfHom(\Mrfeta,G') \rightarrow \bfHom(\Mrfeta,G)$. Finally, suppose $\phi$ is a closed embedding. Then for each $B \in \csalg_k$, the induced set map $\bfHom(\Mrfeta,\phi)(B)$ is an injection. As in the previous paragraph, assume that $G$ is a closed subsupergroup scheme of $\GLmn$. Then $G'$ is also closed in $\GLmn$, and $I_{r;f,\eta}(G')$ is defined in the same fashion as $I_{r;f,\eta}(G)$. Then $I_{r;f,\eta}(G') \supseteq I_{r;f,\eta}(G)$, and it follows that $I_{r;f,\eta}(G')$ defines the image of $\bfHom(\Mrfeta,\phi)$ in $\bfHom(\Mrfeta,G)$, and hence that $\bfHom(\Mrfeta,\phi)$ is a closed embedding.
\end{proof}

\begin{definition} \label{definition:evensubschemes}
Given an algebraic $k$-supergroup scheme $G$, let
\[
\Vr(\GLmn) = [\bsV_r(\GLmn)]_\ev, \quad \text{and} \quad \Vrfeta(G) = [\bfHom(\Mrfeta,G)]_\ev
\]
be the underlying purely even subschemes of $\bsV_r(\GLmn)$ and $\bfHom(\Mrfeta,G)$, respectively. Set $\Vrf(G) = V_{r;f,0}(G)$, and for $s \geq 1$ set $\Vrs(G) = V_{r;T^{p^s}}(G)$.
\end{definition}

By definition, $\Vr(\GLmn)$ and $\Vrfeta(\GLmn)$ are closed subschemes of
\begin{equation} \label{eq:Vrfetaambientscheme}
(\Matmn(-)_{\zero})^{\times r} \times \Matmn(-)_{\one};
\end{equation}
cf.\ \eqref{eq:Matmn0relations} and \eqref{eq:Matmn1relations}.

\begin{definition}[Universal homomorphisms] \label{def:universalhom}
Let $r \geq 1$, let $0 \neq f \in k[T]$ be an inseparable $p$-polynomial, and let $\eta \in k$. Let $G$ be an algebraic $k$-supergroup scheme, and let $G \hookrightarrow \GLmn$ be a closed embedding. Set $A' = k[\bfHom(\Mrfeta,\GLmn)]$, and let
\[
\rho_{(\ualpha|\beta)}: \Mrfeta \otimes_k k[\bfHom(\Mrfeta,\GLmn)] \rightarrow \GLmn \otimes_k k[\bfHom(\Mrfeta,\GLmn)]
\]
be the universal homomorphism from $\Mrfeta$ to $\GLmn$ as defined in the proof of Theorem \ref{theorem:Homsuperscheme}.
	\begin{enumerate}
	\item Set $A_G' = k[\bfHom(\Mrfeta,G)]$, and let $u': \Mrfeta \otimes_k A_G' \rightarrow \GLmn \otimes_k A_G'$ be the homomorphism of $A'$-supergroup schemes obtained from $\rho_{(\ualpha|\beta)}$ via base-change along the quotient map $A' \twoheadrightarrow A_G'$. Then the image of $u'$ is contained in $G \otimes_k A_G'$. We call the induced map
	\[
	u_G': \Mrfeta \otimes_k A_G' \rightarrow G \otimes_k A_G'
	\]
	the \emph{universal supergroup homomorphism from $\Mrfeta$ to $G$}. It is universal in the sense that if $B \in \csalg_k$ and $\rho \in \bfHom(\Mrfeta,G)(B)$, then $\rho = u_G' \otimes_\phi B$ for a unique $\phi \in \Hom_{\salg}(A_G',B)$, and conversely $u_G' \otimes_\phi B \in \bfHom(\Mrfeta,G)(B)$ for each $\phi \in \Hom_{\salg}(A_G',B)$.
	
	\item \label{item:uGpurelyeven} Set $A_G = k[\Vrfeta(G)]$, and let $\pi: A_G' \rightarrow A_G$ be the canonical quotient map whose kernel is generated by the odd elements in $A_G'$. Let
	\[
	u_G = (u_G' \otimes_\pi A_G): \Mrfeta \otimes_k A_G \rightarrow G \otimes_k A_G
	\]
	be the homomorphism obtained from $u_G'$ via base change along $\pi$. We call $u$ the \emph{universal purely even supergroup homomorphism from $\Mrfeta$ to $G$}. It is universal in the sense that if $B \in \calg_k$ is a purely even commutative and if $\rho \in \bfHom(\Mrfeta,G)(B)$, then $\rho = u_G \otimes_\phi B$ for a unique $\phi \in \Hom_{\salg}(A_G,B)$, and conversely $u_G \otimes_\phi B \in \bfHom(\Mrfeta,G)(B)$ for each $\phi \in \Hom_{\salg}(A_G,B)$.
	\end{enumerate}
\end{definition}

\begin{remark}
Let $G$ be an algebraic $k$-supergroup scheme. One can check via naturality that the homomorphisms $u_G'$ and $u_G$ defined above correspond respectively to the identity elements
\begin{gather*}
\id \in \Hom_{\salg}(k[\bfHom(\Mrfeta,G)],k[\bfHom(\Mrfeta,G)]) = \bfHom(\Mrfeta,G)(k[\bfHom(\Mrfeta,G)]), \\
\id \in \Hom_{\alg}(k[\Vrfeta(G)],k[\Vrfeta(G)]) = \Vrfeta(G)(k[\Vrfeta(G)]) = \bfHom(\Mrfeta,G)(k[\Vrfeta(G)]).
\end{gather*}
\end{remark}

\subsection{Gradings}\label{SS:gradings}

In this section set $B = k[x,y]/\subgrp{x^{p^r}-y^2}$. We consider $B$ as a purely even augmented $k$-superalgebra, with augmentation $\ve: B \rightarrow k$ defined by $\ve(x) = \ve(y) = 1$. The assignments $x \mapsto x \otimes x$ and $y \mapsto y \otimes y$ uniquely extend to a $k$-algebra homomorphism $\Delta_B: B \rightarrow B \otimes B$, and hence $\Spec(B)$ admits the structure of a (unital, associative) monoid $k$-scheme. The following lemma is analogous to the classical result that rational representations of the multiplicative group $\G_m$ correspond to $\Z$-graded vector spaces.

\begin{lemma} \label{lemma:Zpr2grading}
The category of rational representations of the monoid $k$-scheme $\Spec(B)$ is naturally equivalent to the category non-negatively $\Z[\frac{p^r}{2}]$-graded $k$-vector spaces. More generally, let $X$ be an affine $k$-superscheme admitting a right monoid action of the purely even $k$-scheme $\Spec(B)$. Then $k[X]$ admits a non-negative $\Z[\frac{p^r}{2}]$-grading, which induces a $\Z[\frac{p^r}{2}]$-grading on the purely even quotient $k[X_{\ev}] = k[X]/\subgrp{k[X]_{\one}}$.
\end{lemma}

\begin{proof}
Let $f: X \times \Spec(B) \rightarrow X$ be the morphism defining the right monoid action of $\Spec(B)$ on $X$. Considering $B$ as a purely even commutative superalgebra, the comorphism $f^*$ is a $k$-superalgebra homomorphism
\[
f^*: k[X] \rightarrow k[X \times \Spec(B)] = k[X] \otimes B.
\]
The set $\set{x^i, x^i y : i \in \N}$ is a $k$-basis for $B$, so for each $a \in k[X]$ one can write
\[
f^*(a) = \sum_{i \geq 0} [ \phi_i(a) \otimes x^i + \psi_i(a) \otimes x^i y]
\]
for some elements $\phi_i(a),\psi_i(a) \in k[X]$ of the same parity as $a$. Then the fact that $f$ defines a right monoid action of $\Spec(B)$ on $X$ implies that $(1_{k[X]} \otimes \Delta_B) \otimes f^* = (f^* \otimes 1_B) \circ f^*$. Using this identity and the fact that $f^*$ is an $k$-superalgebra homomorphism, it follows that the functions $\phi_i: k[X] \rightarrow k[X]$ and $\psi_i: k[X] \rightarrow k[X]$ define projection maps onto the $\Z[\frac{p^r}{2}]$-graded components of $k[X]$ of degrees $i$ and $i+\frac{p^r}{2}$, respectively. Now to see that the $\Z[\frac{p^r}{2}]$-grading on $k[X]$ induces a $\Z[\frac{p^r}{2}]$-grading on $k[X_{\ev}]$, one can use the fact that the functions $\phi_i,\psi_i: k[X] \rightarrow k[X]$ are even linear maps, and hence $k[X]_{\one}$ is a $\Z[\frac{p^r}{2}]$-homogeneous subspace of $k[X]$.
\end{proof}

Now fix $r,s \geq 1$, and let $G$ be an algebraic $k$-supergroup scheme. Composition of homomorphisms defines a natural morphism of affine $k$-superschemes
\begin{equation} \label{eq:HomMonoidAction}
\bfHom(\Mrs,G) \times \bfHom(\Mrs,\Mrs) \rightarrow \bfHom(\Mrs,G).
\end{equation}
Consider the purely even subscheme of $\bfHom(\Mrs,\Mrs)$ consisting for each $A \in \csalg_k$ of those morphisms $\rho_{(\mu,a)} \in \bfHom(\Mrs,\Mrs)(A)$ such that the comorphism $\rho_{(\mu,a)}^*: A[\Mrs] \rightarrow A[\Mrs]$ satisfies $\rho_{(\mu,a)}^*(\tau) = \tau \cdot \mu$, $\rho_{(\mu,a)}^*(\theta) = \theta \cdot a$, and $\rho_{(\mu,a)}^*(\sigma_i) = \sigma_i \cdot a^{ip^{r-1}}$ for some $\mu,a \in \Azero$ with $a^{p^r} = \mu^2$; cf.\ Lemma \ref{lemma:HomMrsMrs}. One has $\rho_{(\mu,a)} \circ \rho_{(\mu',a')} = \rho_{(\mu\mu',aa')}$, so this purely even subscheme identifies with $\Spec(B)$. Then the right action of $\bfHom(\Mrs,\Mrs)$ restricts to a right monoid action of $\Spec(B)$ on $\bfHom(\Mrs,G)$.

\begin{lemma} \label{lemma:Hom(Mrs,G)graded}
Let $r,s \geq 1$, and let $G$ be an algebraic $k$-supergroup scheme. Then the coordinate superalgebra $k[\bfHom(\Mrs,G)]$ is a $\Z[\frac{p^r}{2}]$-graded connected $k$-algebra. Moreover, if $G \rightarrow G'$ is a homomorphism of algebraic $k$-super\-group schemes, then the induced superalgebra map
\[
k[\bfHom(\Mrs,G)] \rightarrow k[\bfHom(\Mrs,G')]
\]
respects gradings.
\end{lemma}

\begin{proof}
The coordinate algebra $k[\bfHom(\Mrs,G)]$ is $\Z[\frac{p^r}{2}]$-graded by Lemma \ref{lemma:Zpr2grading}, and the algebra map $k[\bfHom(\Mrs,G)] \rightarrow k[\bfHom(\Mrs,G')]$ then respects gradings by the naturality of \eqref{eq:HomMonoidAction}. So it remains only to show that $k[\bfHom(\Mrs,G)]$ is connected, i.e., that the degree-$0$ component of $k[\bfHom(\Mrs,G)]$ is just the field $k$.

Let $m,n \in \N$. As a superscheme, $\bfHom(\Mrs,\GLmn)$ is by construction a closed subsuperscheme of $\Matmn(-)_{\zero}^{\times r} \times \Matmn(-)_{\one}$, so $k[\bfHom(\Mrs,\GLmn)]$ is a quotient of the $k$-superalgebra
\[
A := k[\Matmn(-)_{\zero}]^{\otimes r} \otimes k[\Matmn(-)_{\one}].
\]
The $k$-superalgebra $A$ admits a $\Z[\frac{p^r}{2}]$-grading in which the coordinate functions in the $\ell$-th tensor factor of $k[\Matmn(-)_{\zero}]^{\otimes r}$ are of graded degree $p^\ell$ and the coordinate functions in $k[\Matmn(-)_{\one}]$ are of graded degree $\frac{p^r}{2}$. Identifying $\bfHom(\Mrs,\GLmn)$ as in Proposition \ref{proposition:HomMrGLmn}, the right monoid action of $\Spec(B)$ on $\bfHom(\Mrs,\GLmn)$ corresponds to the map
\[
(\alpha_0,\alpha_1,\ldots,\alpha_{r-1},\beta) \times (\mu,a) \mapsto (\alpha_0 \cdot a,\alpha_1 \cdot a^p,\ldots,\alpha_{r-1} \cdot a^{p^{r-1}},\beta \cdot \mu).
\]
From this it follows that the $\Z[\frac{p^r}{2}]$-grading on $k[\bfHom(\Mrs,\GLmn)]$ is precisely the grading induced by the aforementioned $\Z[\frac{p^r}{2}]$-grading on $A$. In particular, $k[\bfHom(\Mrs,\GLmn)]$ is connected. Now choosing a closed embedding $G \hookrightarrow \GLmn$, we get a surjective superalgebra homomorphism $k[\bfHom(\Mrs,\GLmn)] \twoheadrightarrow k[\bfHom(\Mrs,G)]$ that respects gradings, and hence $k[\bfHom(\Mrs,G)]$ is connected as well.
\end{proof}

\begin{corollary} \label{cor:VrsGgraded}
Let $G$ be an algebraic $k$-supergroup scheme. Then $k[\Vrs(G)]$ is a $\Z[\frac{p^r}{2}]$-graded connected $k$-algebra. If $G \rightarrow G'$ is a homomorphism of algebraic $k$-super\-group schemes, then the corresponding algebra map $k[\Vrs(G)] \rightarrow k[\Vrs(G')]$ respects gradings.
\end{corollary}

\section{The Yoneda algebra \texorpdfstring{$\Ext_{\bsp}^\bullet(\bsir,\bsir)$}{ExtP(Ir,Ir)}}\label{S:Ext}

\subsection{Recollections on the cohomology of strict polynomial superfunctors}

Let $\bsp = \bsp_k$ be the category of strict polynomial superfunctors over $k$ as defined in \cite[\S2.1]{Drupieski:2016}. In \cite{Drupieski:2016}, the first author studied for each $r \geq 1$ the structure of $\Ext_{\bsp}^\bullet(\bsir,\bsir)$, the extension algebra in $\bsp$ of the $r$-th Frobenius twist of the identity functor $\bsi$. The functor $\bsir$ admits a direct sum decomposition, $\bsir = \bsirzero \oplus \bsirone$, which gives rise to the matrix ring decomposition
\begin{equation} \label{eq:matrixring}
\renewcommand*{\arraystretch}{1.5}
\Ext_{\bsp}^\bullet(\bsir,\bsir) =
\begin{pmatrix}
\Ext_{\bsp}^\bullet(\bsi_0^{(r)},\bsi_0^{(r)}) & \Ext_{\bsp}^\bullet(\bsi_1^{(r)},\bsi_0^{(r)}) \\
\Ext_{\bsp}^\bullet(\bsi_0^{(r)},\bsi_1^{(r)}) & \Ext_{\bsp}^\bullet(\bsi_1^{(r)},\bsi_1^{(r)})
\end{pmatrix}.
\end{equation}
Define $\bse_0 \in \Hom_{\bsp}(\bsirzero,\bsirzero)$ and $\bse_0^\Pi \in \Hom_{\bsp}(\bsirone,\bsirone)$ to be the respective identity elements. Then $\bse_0$ and $\bse_0^\Pi$ are commuting orthogonal idempotents that sum to the identity in $\Ext_{\bsp}^\bullet(\bsir,\bsir)$, and $\Ext_{\bsp}^\bullet(\bsir,\bsir)$ is generated as a $k$-algebra by $\bse_0$ and $\bse_0^\Pi$ together with certain distinguished (even superdegree) extension classes
\begin{equation} \label{eq:Extgenerators}
\left.
\begin{aligned}
\bse_i^{(r-i)} &\in \Ext_{\bsp}^{2p^{i-1}}(\bsi_0^{(r)},\bsi_0^{(r)}) \\
(\bse_i^{(r-i)})^\Pi &\in \Ext_{\bsp}^{2p^{i-1}}(\bsi_1^{(r)},\bsi_1^{(r)})
\end{aligned}
\right\} \text{ for $1 \leq i \leq r$, and }
\left\{
\begin{aligned}
\bsc_r &\in \Ext_{\bsp}^{p^r}(\bsi_1^{(r)},\bsi_0^{(r)}), \\
\bsc_r^\Pi &\in \Ext_{\bsp}^{p^r}(\bsi_0^{(r)},\bsi_1^{(r)}),
\end{aligned}
\right.
\end{equation}
which are unique up to nonzero scalar multiples. To completely specify the classes, i.e., to fix the scalar multiples, it suffices to fix $\bse_r$ and $\bsc_r$ for each $r \geq 1$, since then the remaining classes are defined in terms of Frobenius twists and the action of the parity change functor $\Pi$.

Recall from \cite[\S3.3]{Drupieski:2016} that post-composition with $\Pi$ induces for each $F,G \in \bsp$ a canonical identi\-fication $\Ext_{\bsp}^\bullet(F,G) = \Ext_{\bsp}^\bullet(\Pi \circ F,\Pi \circ G)$, while pre-composition induces an even isomorphism $\Ext_{\bsp}^\bullet(F,G) \cong \Ext_{\bsp}^\bullet(F \circ \Pi,G \circ \Pi)$. Combining these two actions we get the conjugation action $F \mapsto F^\Pi := \Pi \circ F \circ \Pi$ on objects in $\bsp$, and its induced action on extension groups. Since $\bsi^\Pi = \bsi$ and $(\bsirzero)^\Pi = \bsirone$, conjugation by $\Pi$ defines an automorphism of $\Ext_{\bsp}^\bullet(\bsir,\bsir)$ that exchanges the diagonal (resp., the anti-diagonal) components of the matrix ring \eqref{eq:matrixring}.

Fix positive integers $m$ and $n$. Then as discussed in \cite[Corollary 5.2.3]{Drupieski:2013b}, there exists a May spectral sequence converging to the cohomology of the Frobenius kernel $\GLmnone$ of $\GLmn$:
\begin{equation} \label{eq:Mayspecseq}
\begin{split}
E_0^{i,j} &= \Lambda_s^j(\glmn^\#) \otimes S^{i/2}(\glzero^\#[2])^{(1)} \Rightarrow \opH^{i+j}(\GLmnone,k), \\
E_1^{i,j} &= \opH^j(\glmn,k) \otimes S^{i/2}(\glzero^\#[2])^{(1)} \Rightarrow \opH^{i+j}(\GLmnone,k).
\end{split}
\end{equation}
Here the superscript $i/2$ means that the term is zero unless $i$ is even. In particular, $E_1^{i,j} = 0$ for all odd $i$, so $E_1^{i,j} = E_2^{i,j}$. The restrictions of $\bse_r + \bse_r^\Pi$ and $\bsc_r+\bsc_r^\Pi$ to $\GLmnone$ (cf.\ Definition \ref{def:charclasses}) naturally give rise to homomorphisms of graded superalgebras
\begin{subequations}
\begin{align}
(\bse_r+\bse_r^\Pi) &: S(\glzero^\#[2p^{r-1}])^{(r)} \rightarrow \Hbul(\GLmnone,k), \quad \text{and} \label{ertoG1} \\
(\bsc_r+\bsc_r^\Pi)|_G &: S(\glone^\#[p^r])^{(r)} \rightarrow \Hbul(\GLmnone,k). \label{crtoG1}
\end{align}
\end{subequations}

In the proof of \cite[Theorem 5.5.1]{Drupieski:2016} (as corrected in \cite{Drupieski:2017}), the first author verified the following:
	\begin{itemize}	
	\item Up to a nonzero scalar factor, \eqref{ertoG1} is equal to the composition of the $p^{r-1}$-power map
	\[
	S(\glzero^\#[2p^{r-1}])^{(r)} \rightarrow S(\glzero^\#[2])^{(1)}
	\]
	with the horizontal edge map $E_1^{\bullet,0} \twoheadrightarrow E_\infty^{\bullet,0} \hookrightarrow \Hbul(\GLmnone,k)$ of \eqref{eq:Mayspecseq}. 

	\item Up to a nonzero scalar factor, the composition of \eqref{crtoG1} with the vertical edge map
	\[
	\Hbul(\GLmnone,k) \twoheadrightarrow E_\infty^{0,\bullet} \hookrightarrow E_1^{0,\bullet} = \Hbul(\glmn,k)
	\]
	of \eqref{eq:Mayspecseq} is equal to the homomorphism of graded superalgebras induced by the composition
	\[
	S(\glone^\#[p^r])^{(r)} \rightarrow S(\glone^\#) \hookrightarrow \Lambda_s(\glmn^\#),
	\]
	where the first arrow is the $p^r$-power map and the second arrow is the natural inclusion.
	\end{itemize}
By the strict polynomial superfunctor analogue of \cite[Lemma 3.14]{Friedlander:1997}, the aforementioned scalar factors are independent of the particular values of $m$ and $n$.

\begin{convention} \label{convention}
From now on we will assume for each $r \geq 1$ that the extension classes $\bse_r$ and $\bsc_r$ are fixed so as to ensure that the aforementioned scalar factors are both equal to $1$. This convention may differ from the convention chosen in \cite[\S4.4]{Drupieski:2016}.
\end{convention}

The following proposition is a reformulation of \cite[Corollary 4.6.5]{Drupieski:2016}:

\begin{proposition} \label{prop:erj}
Given an integer $0 \leq j < p^r$ with $p$-adic decomposition $j = \sum_{i=0}^{r-1} j_i p^i$, set
\[
\bse_r(j) = \frac{(\bse_1^{(r-1)})^{j_0}(\bse_2^{(r-2)})^{j_1} \cdots (\bse_r)^{j_{r-1}}}{(j_0!)(j_1!) \cdots (j_{r-1}!)} \in \Ext_{\bsp}^{2j}(\bsi_0^{(r)},\bsi_0^{(r)}).
\]
For $j \geq p^r$, write $j = b + ap^r$ with $0 \leq b < p^r$. Set $\bse_r(j) = \bse_r(b) \cdot [(\bse_r)^p]^a$ and $\bse_r^\Pi(j) = [\bse_r(j)]^\Pi$. Then:
	\begin{enumerate}
	\item $\set{\bse_r(j) : j \in \N}$ is a (purely even) basis for $\Ext_{\bsp}^\bullet(\bsirzero,\bsirzero)$.
	\item $\set{\bse_r(j) \cdot \bsc_r : j \in \N}$ is a (purely even) basis for $\Ext_{\bsp}^\bullet(\bsirone,\bsirzero)$.
	\item $\set{\bse_r^\Pi(j) : j \in \N}$ is a (purely even) basis for $\Ext_{\bsp}^\bullet(\bsirone,\bsirone)$.
	\item $\set{\bse_r^\Pi(j) \cdot \bsc_r^\Pi : j \in \N}$ is a (purely even) basis for $\Ext_{\bsp}^\bullet(\bsirzero,\bsirone)$.
	\end{enumerate}
\end{proposition}

\begin{remark} \label{remark:erjtoerj}
It follows from the naturality of the May spectral sequence that, with the adoption of Convention \ref{convention}, the restriction homomorphism $\Ext_{\bsp}^\bullet(\bsir,\bsir) \rightarrow \Ext_{\cp}^\bullet(I^{(r)},I^{(r)})$ induced by restriction from $\bsp$ to $\cp$ (see the end of \cite[\S2.1]{Drupieski:2016}) maps $\bse_r$ to the class $e_r$ defined in \cite[Convention 3.4]{Suslin:1997}, and hence maps $\bse_r(j)$ for $0 \leq j < p^r$ to the class $e_r(j)$ defined in \cite[Theorem 3.1]{Suslin:1997}.
\end{remark}

\subsection{Hopf superalgebra structure of \texorpdfstring{$\Ext_{\bsp}^\bullet(\bsir,\bsir)$}{ExtP(Ir,Ir)}}

Recall the category $\bibsp$ of strict poly\-nomial bisuperfunctors defined in \cite[\S2.4]{Drupieski:2016}. Given $F,G \in \bsp$, the external tensor product $F \boxtimes G \in \bibsp$ is defined on objects by $(F \boxtimes G)(V,W) = F(V) \otimes G(W)$. As discussed in \cite[\S3.4]{Drupieski:2016}, the external tensor product operation gives rise for each $A,B,C,D \in \bsp$ to an even bilinear map
\[
\Ext_{\bsp}^m(A,B) \otimes \Ext_{\bsp}^n(C,D) \rightarrow \Ext_{\bibsp}^{m+n}(A \boxtimes C, B \boxtimes D) \label{eq:externalbibspproduct}.
\]
Then the same argument as for \cite[Proposition 3.6]{Suslin:1997} shows the following:

\begin{lemma}
Let $F,G \in \bsp$. Then the external tensor product operation induces an even isomorphism of graded superalgebras
\begin{equation} \label{eq:externalproductextiso}
\boxtimes^*: \Ext_{\bsp}^\bullet(F,F) \gotimes \Ext_{\bsp}^\bullet(G,G) \simrightarrow \Ext_{\bibsp}^\bullet(F \boxtimes G, F \boxtimes G),
\end{equation}
where $\Ext_{\bsp}^\bullet(F,F)$, $\Ext_{\bsp}^\bullet(G,G)$, and $\Ext_{\bibsp}^\bullet(F \boxtimes G, F \boxtimes G)$ are graded superalgebras via their respective Yoneda composition products.
\end{lemma}

Define $\nabla: \bsp \rightarrow \bibsp$ by $\nabla(F) = F \circ (\bsi \boxtimes \bsi)$. Then on objects, $\nabla(F)$ is defined by
\[
\nabla(F)(V,W) = F(V \otimes W).
\]
Since $\nabla$ is exact, we again get, as in \cite[\S3]{Suslin:1997}, an induced map on extension groups
\[
\nabla^*: \Ext_{\bsp}^\bullet(F,G) \rightarrow \Ext_{\bibsp}^\bullet(\nabla(F),\nabla(G)).
\]

\begin{lemma}
There exists a natural identification $\nabla(\bsir) = \bsir \boxtimes \bsir$.
\end{lemma}

\begin{proof}
Given $V,W \in \bsv$, there is an obvious natural identification of vector superspaces
\[
(V \otimes W)^{(r)} = V^{(r)} \otimes W^{(r)},
\]
so it suffices to check that this identification is compatible with the actions of the bisuperfunctors $\nabla(\bsir)$ and $\bsir \boxtimes \bsir$ on morphisms. The compatibility can be checked via duality from a commutative diagram of the type considered in \cite[(2.7.1)]{Drupieski:2016}.
\end{proof}

\begin{remark} \label{remark:nablabsir}
The identification $\nabla(\bsir) = \bsir \boxtimes \bsir$ restricts to identifications
\begin{align*}
\nabla(\bsi_0^{(r)}) &= (\bsi_0^{(r)} \boxtimes \bsi_0^{(r)}) \oplus (\bsi_1^{(r)} \boxtimes \bsi_1^{(r)}), \quad \text{and} \\
\nabla(\bsi_1^{(r)}) &= (\bsi_0^{(r)} \boxtimes \bsi_1^{(r)}) \oplus (\bsi_1^{(r)} \boxtimes \bsi_0^{(r)}).
\end{align*}
\end{remark}

\begin{lemma} \label{lemma:augmentation}
Define $\ve: \Ext_{\bsp}^\bullet(\bsir,\bsir) \rightarrow k$ to be the composite of projection onto cohomological degree zero and evaluation on the vector space $k$,
\[
\ve: \Ext_{\bsp}^\bullet(\bsir,\bsir) \rightarrow \Hom_{\bsp}(\bsir,\bsir) \rightarrow \Hom_k(k^{(r)},k^{(r)}) \cong k.
\]
Set $\ve^\Pi = \ve \circ (-)^\Pi$. Then $\ve$ and $\ve^\Pi$ both make $\Ext_{\bsp}^\bullet(\bsir,\bsir)$ into an augmented superalgebra.
\end{lemma}

\begin{proof}
It is straightforward to verify that $\ve$ is homomorphism of graded superalgebras, and hence that $\ve$ makes $\Ext_{\bsp}^\bullet(\bsir,\bsir)$ into an augmented superalgebra. The claim for $\ve^\Pi$ then follows immediately, since conjugation by $\Pi$ defines an even automorphism of $\Ext_{\bsp}^\bullet(\bsir,\bsir)$.
\end{proof}

\begin{remark}
Recall that $\bse_0$ and $\bse_0^\Pi$ are commuting orthogonal idempotents that sum to the identity in $\Ext_{\bsp}^\bullet(\bsir,\bsir)$. The augmentation map $\ve$ satisfies $\ve(\bse_0) = 1$ and $\ve(\bse_0^\Pi) = 0$, while $\ve^\Pi$ satisfies $\ve^\Pi(\bse_0) = 0$ and $\ve^\Pi(\bse_0^\Pi) = 1$. 
\end{remark}

\begin{proposition} \label{prop:coproduct}
Making the identification $\nabla(\bsir) = \bsir \boxtimes \bsir$, the map $\Delta$ defined by
\begin{multline*}
\Delta := (\boxtimes^*)^{-1} \circ \nabla^*: \Ext_{\bsp}^\bullet(\bsir,\bsir) \rightarrow \Ext_{\bibsp}^\bullet(\bsir \boxtimes \bsir,\bsir \boxtimes \bsir) \\
\simrightarrow \Ext_{\bsp}^\bullet(\bsir,\bsir) \gotimes \Ext_{\bsp}^\bullet(\bsir,\bsir)
\end{multline*}
is a homomorphism of graded superalgebras that, together with the augmentation map $\ve$ defined in Lemma \ref{lemma:augmentation}, provides $\Ext_{\bsp}^\bullet(\bsir,\bsir)$ with the structure of a graded super-bialgebra. The coproduct $\Delta$ is compatible with the conjugation action of $\Pi$ in the sense that
\begin{equation} \label{eq:coproductPi}
[(-)^\Pi \otimes 1] \circ \Delta = \Delta \circ (-)^\Pi = [1 \otimes (-)^\Pi] \circ \Delta.
\end{equation}
In particular,
\begin{equation} \label{eq:coproductPiPi}
\Delta = [(-)^\Pi \otimes (-)^\Pi] \circ \Delta.
\end{equation}
Finally, the (super)algebra homomorphism $\Ext_{\bsp}^\bullet(\bsir,\bsir) \rightarrow \Ext_{\cp}^\bullet(I^{(r)},I^{(r)})$ defined by restriction from $\bsp$ to $\cp$ is a homomorphism of graded (super)bialgebras.
\end{proposition}

\begin{proof}
The proof that $\Delta$ and $\ve$ make $\Ext_{\bsp}^\bullet(\bsir,\bsir)$ into a graded super-bialgebra proceeds in precisely the same fashion as the proof of the corresponding fact in \cite[Proposition 3.7]{Suslin:1997}. Next we verify the equalities in \eqref{eq:coproductPi}. Recall that the parity change functor $\Pi$ acts on objects in $\bsv$ by reversing the $\Z_2$-grading, and acts on morphisms by $\Pi(\phi) = \phi$, i.e., $\Pi(\phi)$ is equal to $\phi$ as a map between the underlying vector spaces. Then one has canonical identifications in $\bibsp$
\[
\Pi \boxtimes \bsi = \Pi \circ (\bsi \boxtimes \bsi) = \bsi \boxtimes \Pi.
\]
Using these identifications, one gets natural identifications
\begin{align*}
\Ext_{\bibsp}^\bullet(\Pi \circ \nabla(\bsir \circ \Pi),\Pi \circ \nabla(\bsir \circ \Pi)) &= \Ext_{\bibsp}^\bullet(\nabla(\bsir \circ \Pi), \nabla(\bsir \circ \Pi)) \\
&= \Ext_{\bibsp}^\bullet(\bsir \circ (\Pi \boxtimes \bsi),\bsir \circ (\Pi \boxtimes \bsi)) \\
&= \Ext_{\bibsp}^\bullet((\bsir \circ \Pi) \boxtimes \bsir,(\bsir \circ \Pi) \boxtimes \bsir).
\end{align*}
One also has canonical identifications
\[
\Ext_{\bsp}^\bullet(\bsir \circ \Pi,\bsir \circ \Pi) = \Ext_{\bsp}^\bullet(\Pi \circ \bsir \circ \Pi,\Pi \circ \bsir \circ \Pi) = \Ext_{\bsp}^\bullet((\bsir)^\Pi,(\bsir)^\Pi).
\]
Combining these identifications, the first equality in \eqref{eq:coproductPi} follows, and the reasoning for the second equality in \eqref{eq:coproductPi} is entirely analogous. Now applying both equalities in \eqref{eq:coproductPi}, one gets
\[
\Delta = \Delta \circ (-)^\Pi \circ (-)^\Pi = [(-)^\Pi \otimes 1] \circ \Delta \circ (-)^\Pi = [(-)^\Pi \otimes 1] \circ [1 \otimes (-)^\Pi] \circ \Delta,
\]
establishing \eqref{eq:coproductPiPi}. Finally, restriction from $\bsp$ to $\cp$ sends the operations $\boxtimes$, $\nabla$, and $\ve$ to the operations used in \cite[3.7]{Suslin:1997} to define the coproduct and augmentation map on $\Ext_{\cp}^\bullet(I^{(r)},I^{(r)})$. This implies the last assertion of the proposition.
\end{proof}

\begin{proposition} \label{proposition:coproduct}
The coproduct $\Delta$ on $\Ext_{\bsp}^\bullet(\bsir,\bsir)$ satisfies the following:
\begin{align*}
\Delta(\bse_r(\ell)) &= \textstyle \sum_{i+j = \ell} \left[ \bse_r(i) \otimes \bse_r(j) + \bse_r^\Pi(i) \otimes \bse_r^\Pi(j) \right], & \text{if $0 \leq \ell < p^r$,} \\
\Delta(\bse_r^\Pi(\ell)) &= \textstyle \sum_{i+j = \ell} \left[ \bse_r^\Pi(i) \otimes \bse_r(j) + \bse_r(i) \otimes \bse_r^\Pi(j) \right], & \text{if $0 \leq \ell < p^r$,} \\
\Delta(\bsc_r) &= \bsc_r \otimes \bse_0 + \bse_0 \otimes \bsc_r + \bsc_r^\Pi \otimes \bse_0^\Pi + \bse_0^\Pi \otimes \bsc_r^\Pi, & \text{and} \\
\Delta(\bsc_r^\Pi) &= \bsc_r^\Pi \otimes \bse_0 + \bse_0 \otimes \bsc_r^\Pi + \bsc_r \otimes \bse_0^\Pi + \bse_0^\Pi \otimes \bsc_r.
\end{align*}
\end{proposition}

\begin{proof}
It suffices by \eqref{eq:coproductPi} to verify the stated formulas for $\Delta(\bse_r(\ell))$ and $\Delta(\bsc_r)$. First, it follows from Remark \ref{remark:nablabsir} and the definition of the coproduct that $\Delta(\bse_r(\ell))$ is an element of
\[
\Ext_{\bsp}^\bullet(\bsi_0^{(r)},\bsi_0^{(r)})^{\otimes 2}
\oplus \Ext_{\bsp}^\bullet(\bsi_1^{(r)},\bsi_1^{(r)})^{\otimes 2}
\oplus \Ext_{\bsp}^\bullet(\bsi_0^{(r)},\bsi_1^{(r)})^{\otimes 2}
\oplus \Ext_{\bsp}^\bullet(\bsi_1^{(r)},\bsi_0^{(r)})^{\otimes 2}.
\]
Since the coproduct preserves the total cohomological degree, and since the third and fourth direct summands of the previous expression are both nonzero only in total cohomological degrees at least $2p^r$, it follows for $0 \leq \ell < p^r$ that $\Delta(\bse_r(\ell)) \in \Ext_{\bsp}^\bullet(\bsi_0^{(r)},\bsi_0^{(r)})^{\otimes 2}
\oplus \Ext_{\bsp}^\bullet(\bsi_1^{(r)},\bsi_1^{(r)})^{\otimes 2}$. Next consider the restriction map $\Ext_{\bsp}^\bullet(\bsir,\bsir) \rightarrow \Ext_{\cp}^\bullet(I^{(r)},I^{(r)})$, which is a homomorphism of graded superbialgebras by Proposition \ref{prop:coproduct}. As observed in the proof of \cite[Theorem 4.7.1]{Drupieski:2016}, the restriction map induces isomorphisms $\Ext_{\bsp}^j(\bsirzero,\bsirzero) \cong \Ext_{\cp}^j(I^{(r)},I^{(r)})$ for $0 \leq j < 2p^r$. Then the formula for $\Delta(\bse_r(\ell))$ follows from Remark \ref{remark:erjtoerj}, the formula stated in \cite[Theorem 4.6]{Suslin:1997} for the coproduct in $\Ext_{\cp}^\bullet(I^{(r)},I^{(r)})$, and the identity \eqref{eq:coproductPiPi}.

By reasoning parallel to that for $\Delta(\bse_r(\ell))$, $\Delta(\bsc_r)$ must be an element of the direct sum
\begin{multline*}
\left( \Ext_{\bsp}^\bullet(\bsi_1^{(r)},\bsi_0^{(r)}) \otimes \Ext_{\bsp}^\bullet(\bsi_0^{(r)},\bsi_0^{(r)}) \right)
\oplus \left( \Ext_{\bsp}^\bullet(\bsi_0^{(r)},\bsi_0^{(r)}) \otimes \Ext_{\bsp}^\bullet(\bsi_1^{(r)},\bsi_0^{(r)}) \right) \\
\oplus \left( \Ext_{\bsp}^\bullet(\bsi_0^{(r)},\bsi_1^{(r)}) \otimes \Ext_{\bsp}^\bullet(\bsi_1^{(r)},\bsi_1^{(r)}) \right)
\oplus \left( \Ext_{\bsp}^\bullet(\bsi_1^{(r)},\bsi_1^{(r)}) \otimes \Ext_{\bsp}^\bullet(\bsi_0^{(r)},\bsi_1^{(r)}) \right).
\end{multline*}
Since $\Delta$ preserves the total cohomological degree, $\Delta(\bsc_r)$ must then be a linear combination of the monomials $\bsc_r \otimes \bse_0$, $\bse_0 \otimes \bsc_r$, $\bsc_r^\Pi \otimes \bse_0^\Pi$, and $\bse_0^\Pi \otimes \bsc_r^\Pi$. Using \eqref{eq:coproductPiPi} and the bialgebra axiom $(\ve \otimes 1) \circ \Delta = 1 = (1 \otimes \ve) \circ \Delta$, it follows that each monomial occurs with coefficient $1$.
\end{proof}

\section{Characteristic classes arising from multiparameter supergroups} \label{section:charclasses}

\subsection{Definition of characteristic classes}

The following lemma strengthens Remark \ref{remark:GLVstructure}.

\begin{lemma} \label{lemma:evaluationV}
Let $A$ be a commutative $k$-superalgebra, and let $d,m,n \in \N$. Evaluation on $\Amn$ defines an exact functor from $\bsp_{d,A}$ to the category of rational $\GLmn \otimes_k A$-supermodules.
\end{lemma}

\begin{proof}
Set $V = \Amn$. Evaluation on $V$ defines an exact functor from $\bsp_{d,A}$ to the category of $A$-supermodules by the definition of exactness in $\bsp_{d,A}$, so it suffices to show that the image of the evaluation functor is contained in the subcategory of rational $\GLmn \otimes_k A$-supermodules.

Let $T \in \bsp_{d,A}$. Then $T$ defines an even $A$-linear map
\[
T_{V,V} : \Hom_{\bsg^d(\bsv_A)}(V,V) = \bsg^d \Hom_A(V,V) \rightarrow \Hom_A(T(V),T(V))
\]
that is compatible with the composition of homomorphisms. For $X,Y \in \bsv_A$, one has a natural even $A$-linear identification $\Hom_A(X,Y) \cong Y \otimes_A \Hom_A(X,A) = Y \otimes_A X^\#$. Then
\begin{align*}
\Hom_A(\bsg^d \Hom_A(V,V),\Hom_A(T(V),T(V)) &\cong (T(V) \otimes_A T(V)^\#) \otimes_A [\bsg^d \Hom_A(V,V)]^\# \\
&\cong (T(V) \otimes_A [\bsg^d \Hom_A(V,V)]^\#) \otimes_A T(V)^\# \\
&\cong \Hom_A(T(V),T(V) \otimes_A [\bsg^d \Hom_A(V,V)]^\#).
\end{align*}
Let $T_V$ denote the image of $T_{V,V}$ under this sequence of identifications. Then for each $v \in T(V)$, we can write (in Sweedler's notation) $T_V(v) = \sum_{(v)} v_{(0)} \otimes_A v_{(1)}$ for some $v_{(0)} \in T(V)$ and $v_{(1)} \in [\bsg^d \Hom_A(V,V)]^\#$ such that $T_{V,V}(\gamma)(v) = (-1)^{\ol{v} \cdot \ol{\gamma}} \sum_{(v)} v_{(0)} \cdot v_{(1)}(\gamma)$ for each $\gamma \in \bsg^d \Hom_A(V,V)$. Similarly, the composition map $[\bsg^d \Hom_A(V,V)]^{\otimes 2} \rightarrow \bsg^d \Hom_A(V,V)$, $\gamma \otimes \gamma' \rightarrow \gamma \circ \gamma'$, corresponds by duality to an even $A$-linear map $\Delta: [\bsg^d \Hom_A(V,V)]^\# \rightarrow ([\bsg^d \Hom_A(V,V)]^{\otimes 2})^\#$ such that $\Delta(f)(\gamma \otimes \gamma') = f(\gamma \circ \gamma')$. By \eqref{eq:boxtimes} we identify the image of $\Delta$ with $([\bsg^d \Hom_A(V,V)]^\#)^{\otimes 2}$, and write $\Delta(f) = \sum_{(f)} f_{(1)} \otimes_A f_{(2)}$ for some $f_{(1)},f_{(2)} \in [\bsg^d \Hom_A(V,V)]^\#$ such that
\[ \textstyle
f(\gamma \circ \gamma') = \Delta(f)(\gamma \otimes \gamma') = \sum_{(f)} (-1)^{\ol{\gamma} \cdot \ol{f_{(2)}}} f_{(1)}(\gamma) \cdot f_{(2)}(\gamma').
\]
Now for $\gamma,\gamma' \in \bsg^d \Hom_A(V,V) = \Hom_{\bsg^d(\bsv_A)}(V,V)$, we get
\begin{align*}
T_{V,V}(\gamma \circ \gamma')(v) &= T_{V,V}(\gamma) \circ T_{V,V}(\gamma')(v) \\
&= \textstyle T_{V,V}(\gamma)\left( (-1)^{\ol{v} \cdot \ol{\gamma'}} \sum_{(v)} v_{(0)} \cdot v_{(1)}(\gamma') \right) \\
&= \textstyle \left[\sum_{(v)} (-1)^{\ol{v} \cdot \ol{\gamma'}} T_{V,V}(\gamma)(v_{(0)}) \right] \cdot v_{(1)}(\gamma') \\
&= \textstyle \left[ \sum_{(v)} \sum_{(v_{(0)})} (-1)^{\ol{v} \cdot \ol{\gamma'} + \ol{v_{(0)}} \cdot \ol{\gamma}} (v_{(0)})_{(0)} \cdot (v_{(0)})_{(1)}(\gamma)\right] \cdot v_{(1)}(\gamma').
\end{align*}
Using the coproduct $\Delta$, we also get
\begin{align*}
T_{V,V}(\gamma \circ \gamma')(v) &= \textstyle \sum_{(v)} (-1)^{\ol{v} \cdot (\ol{\gamma} + \ol{\gamma'})} v_{(0)} \cdot v_{(1)}(\gamma \circ \gamma') \\
&= \textstyle \sum_{(v)} \sum_{(v_{(1)})} (-1)^{\ol{v} \cdot (\ol{\gamma} + \ol{\gamma'}) + \ol{\gamma} \cdot \ol{(v_{(1)})_{(2)}}} v_{(0)} \cdot \left[ (v_{(1)})_{(1)}(\gamma) \cdot (v_{(1)})_{(2)}(\gamma') \right].
\end{align*}
Note that $T(V) \otimes_A ([\bsg^d \Hom_A(V,V)]^\#)^{\otimes 2}$ identifies with $\Hom_A([\bsg^d \Hom_A(V,V)]^{\otimes 2},T(V))$ via
\[
v \otimes_A (\phi \otimes_A \psi): \gamma \otimes_A \gamma' \mapsto (-1)^{\ol{\psi} \cdot \ol{\gamma}} v \cdot \left( \phi(\gamma) \cdot \psi(\gamma') \right).
\]
Viewed in this way as functions, we see from the two expressions for $T_{V,V}(\gamma \circ \gamma')(v)$ that
\[
\sum_{(v)} \sum_{(v_{(0)})} \left[(v_{(0)})_{(0)} \otimes (v_{(0)})_{(1)} \right] \otimes v_{(1)} = \sum_{(v)} \sum_{(v_{(1)})} v_{(0)} \otimes \left[ (v_{(1)})_{(1)} \otimes (v_{(1)})_{(2)} \right],
\]
and thus conclude that $T_V$ and $\Delta$ make $T(V)$ into a right $[\bsg^d \Hom_A(V,V)]^\#$-super\-comodule. Now to finish the proof, observe that the duality isomorphism $\bsg^\# \cong \bss$ restricts to an isomorphism $[\bsg^d \Hom_A(V,V)]^\# \cong \bss^d(\Hom_A(V,V)^\#)$. The symmetric superalgebra $\bss(\Hom_A(V,V)^\#)$ identifies, via the choice of a homogeneous basis for $V$, with the coordinate superalgebra $A[\Matmn(-)_{\zero}]$ of the $A$-superfunctor $\Matmn(-)_{\zero}$. Moreover, under these identifications the coproduct $\Delta$ corresponds to the usual coproduct on $A[\Matmn(-)_{\zero}]$ that is induced by matrix multiplication in $\Matmn$. Finally, since $A[\GLmn]$ is by definition a localization of $A[\Matmn(-)_{\zero}]$, $\bss^d(\Hom_A(V,V)^\#)$ identifies with a subspace of $A[\GLmn]$. Then $T_V$ defines an even $A$-linear map $T_V: T(V) \rightarrow T(V) \otimes A[\GLmn]$.
\end{proof}

\begin{lemma} \label{lemma:restrictionGV}
Let $d,m,n \in \N$, and let $S,T \in \bsp_{d,k}$. Set $V = \Amn$, and denote $\GLmn \otimes_k A$ by $GL(V)$. Let $\rho: G \rightarrow GL(V)$ be a homomorphism of $A$-supergroup schemes. Then base-change to $A$, evaluation on $V$, and restriction to $G$ define an even $A$-linear map
\begin{multline*}
\res_{(G,V)}: \Ext_{\bsp}^\bullet(S,T) \stackrel{-\otimes_k A}{\longrightarrow} \Ext_{\bsp_A}^\bullet(S_A,T_A) \\
\longrightarrow \Ext_{GL(V)}^\bullet(S_A(V),T_A(V)) \longrightarrow \Ext_G^\bullet(S_A(V),T_A(V)),
\end{multline*}
which we denote $z \mapsto z(G,V)$. If $S = T$, then $\res_{(G,V)}$ is homomorphism of graded superalgebras.
\end{lemma}

\begin{proof}
This is an immediate consequence of Corollary \ref{cor:functorbasechange} and Lemma \ref{lemma:evaluationV}.
\end{proof}

\begin{definition}[Characteristic classes] \label{def:charclasses}
Let $\bse_r(j)$ and $\bse_r^\Pi(j)$ be as defined in Proposition \ref{prop:erj}. Let $A \in \csalg_k$. Now for $V = \Amn$ and $\rho: G \rightarrow GL(V)$ as in Lemma \ref{lemma:restrictionGV}, define
\begin{align*}
\bse_r(j)(G,V) &\in \Ext_G^{2j}((A^{m|0})^{(r)},(A^{m|0})^{(r)}), & \bsc_r(G,V) &\in \Ext_G^{p^r}((A^{0|n})^{(r)},(A^{m|0})^{(r)}), \\
\bse_r^\Pi(j)(G,V) &\in \Ext_G^{2j}((A^{0|n})^{(r)},(A^{0|n})^{(r)}), & \bsc_r^\Pi(G,V) &\in \Ext_G^{p^r}((A^{m|0})^{(r)},(A^{0|n})^{(r)}),
\end{align*}
to be the images of $\bse_r(j)$, $\bsc_r$, $\bse_r^\Pi(j)$, and $\bsc_r^\Pi$, respectively, under the restriction homomorphism $\res_{(G,V)}$ of Lemma \ref{lemma:restrictionGV} (cf.\ Remark \ref{remark:Frobeniusbasechange} for the calculation of $(\bsirzero)_A(V)$ and $(\bsirone)_A(V)$).
\end{definition}

We now get the following direct analogue of \cite[Proposition 3.2]{Suslin:1997}:

\begin{proposition} \label{prop:charclassproperties}
Let $V,W \in \bsv_A$ be free $A$-supermodules equipped with rational representations of an $A$-supergroup scheme $G$ as in Lemma \ref{lemma:restrictionGV}. Let $z \in \Ext_{\bsp}^\bullet(\bsir,\bsir)$. Then:
\begin{enumerate}
\item \label{item:evenmap} If $\phi: V \rightarrow W$ is an even homomorphism of rational $G$-supermodules, then
	\[
	\phi^{(r)} \cdot z(G,V) = z(G,W) \cdot \phi^{(r)} \in \Ext_G^\bullet(V^{(r)},W^{(r)}).
	\]

\item \label{item:directsum} Under the matrix ring decomposition
\[
\renewcommand*{\arraystretch}{1.5}
\Ext_G^\bullet((V \oplus W)^{(r)},(V \oplus W)^{(r)}) =
\begin{pmatrix}
\Ext_G^\bullet(V^{(r)},V^{(r)}) & \Ext_G^\bullet(W^{(r)},V^{(r)}) \\
\Ext_G^\bullet(V^{(r)},W^{(r)}) & \Ext_G^\bullet(W^{(r)},W^{(r)})
\end{pmatrix}
\]
arising from the identification $(V \oplus W)^{(r)} = V^{(r)} \oplus W^{(r)}$, $z(G,V \oplus W) = z(G,V) \oplus z(G,W)$.

\item \label{item:basechange} If $G$ is affine and if $A'$ is a commutative $A$-superalgebra, then $z(G \otimes_A A', V_{A'})$ is equal to the image of $z(G,V)$ under the base change homomorphism of Corollary \ref{cor:functorbasechange}.
\end{enumerate}
\end{proposition}

\begin{proof}
We may assume that $z$ is homogeneous of cohomological degree $n$. Since $\Ext_{\bsp}^\bullet(\bsir,\bsir)$ is a purely even algebra by the calculations of \cite{Drupieski:2016}, $z$ is automatically of even superdegree. Then as in \cite[\S3.5]{Drupieski:2016}, $z$ may be interpreted as the  equivalence class of an exact sequence in $(\bsp_{p^r})_\ev$:
\[
0 \rightarrow \bsir \rightarrow T_1 \rightarrow \cdots \rightarrow T_n \rightarrow \bsir \rightarrow 0.
\]
Given an even map $\phi: V \rightarrow W$, one then has $\phi^{\otimes p^r} \in \bsg^{p^r} \Hom_A(V,W)$. Now interpreting $(T_i)_A(\phi)$ to be $(T_i)_A(\phi^{\otimes p^r})$, the proof becomes a direct repetition of the proof of \cite[Proposition 3.2]{Suslin:1997}.
\end{proof}

\begin{remark}
Let $m,n \in \N$, set $V = \Amn$, and let $\rho: G \rightarrow GL(V)$ be a homomorphism of $A$-supergroup schemes as in Lemma \ref{lemma:restrictionGV}. Suppose that $G$ is infinitesimal of height $\leq r$, i.e., suppose the augmentation ideal $I_G$ of $A[G]$ is nilpotent of degree $\leq p^r$. Then the comorphism $\rho^*: A[GL(V)] \rightarrow A[G]$ factors through the quotient $A[GL(V)]/\subgrp{f^{p^r}: f \in I_{GL(V)}}$, and hence the image of $\rho$ is contained in $GL(V)_{(r)}$, the $r$-th Frobenius kernel of $GL(V)$. Next, since $V$ is free over $A$, so is $\Hom_A(V,V)$, and hence the structure morphism
\[
\bsir_{V,V} = (\varphi^r)^\#(\Hom_A(V,V)): \bsg^{p^r} \Hom_A(V,V) \rightarrow\Hom_A(V^{(r)},V^{(r)})
\]
admits an explicit description on generators as in \cite[(2.7.2)]{Drupieski:2016}. In particular, fixing a homogeneous basis for $\Hom_A(V,V)$, and then fixing the corresponding basis for $\bsg^{p^r} \Hom_A(V,V)$ as in \cite[(2.3.4)]{Drupieski:2016}, it follows that $\bsir_{V,V}$ vanishes on all basis monomials except those of the form $\gamma_{p^r}(x)$ for $x$ an even basis element of $\Hom_A(V,V)$. Using this observation, one can check that the evaluation functor of Lemma \ref{lemma:evaluationV} maps the Frobenius twist superfunctor $\bsir \in \bsp_{p^r,A}$ to a rational $GL(V)$-supermodule $V^{(r)}$ that restricts trivially to $GL(V)_{(r)}$, and hence also restricts trivially along $\rho$ to $G$.
\end{remark}

\begin{remark} \label{remark:Extmatrices}
Retain the assumptions and notations of the previous remark. Then as in \cite[Remark 3.3]{Suslin:1997}, we have algebra identifications
	\begin{equation} \label{eq:Extfloatout}
	\begin{split}
	\Ext_G^\bullet(V^{(r)},V^{(r)}) &= \Hom_A(V^{(r)},V^{(r)}) \otimes_A \Hbul(G,A) \\
	&= \Hom_A(V,V)^{(r)} \otimes_A \Hbul(G,A) \\
	&= \Matmn(A)^{(r)} \otimes_A \Hbul(G,A)
	\end{split}
	\end{equation}
Moreover, $\Matmn(A)^{(r)}$ further identifies as an algebra with $\Matmn(A)$: If $A \in \csalg_k$ is purely even, then the identification $\Matmn(A)^{(r)} = \Matmn(A)$ is defined for $\alpha \in \Matmn(A)$ by sending $\alpha^{(r)} = \alpha \otimes_{\varphi^r} 1 \in \Matmn(A) \otimes_{\varphi^r} A$ to the matrix in $\Matmn(A)$ obtained by raising the individual matrix entries of $\alpha$ to the $p^r$-th power; if $A$ is not purely even, then the identification is induced by a similar map, though with various $\pm$ signs involved based on the convention for the right action of $A$ on $\Matmn(A)$; cf.\ Section \ref{subsection:supermatrices}. Similarly, there exists an identification
\[
\Matmn(A) \otimes_A \Hbul(G,A) = \Matmn(\Hbul(G,A))
\]
involving various $\pm$ signs; cf.\ \eqref{eq:Matmnbasechange}. Under the composite identification
\[
\Ext_G^\bullet(V^{(r)},V^{(r)}) = \Matmn(\Hbul(G,A)),
\]
the Yoneda product in $\Ext_G^\bullet(V^{(r)},V^{(r)})$ corresponds to the matrix product in $\Matmn(\Hbul(G,A))$.
\end{remark}

\begin{proposition} \label{proposition:Kunneth}
Let $V,V' \in \bsv_A$ be free $A$-supermodules equipped with rational representations $\rho: G \rightarrow GL(V)$ and $\rho': G' \rightarrow GL(V')$ of $A$-supergroup schemes $G$ and $G'$ as in Lemma \ref{lemma:restrictionGV}. Let $z \in \Ext_{\bsp}^n(\bsir,\bsir)$, and let $\Delta(z) = \sum_{(z)} z_{(1)} \otimes z_{(2)}$ be the coproduct of $z$ as in Proposition \ref{prop:coproduct}. Then the K\"{u}nneth isomorphism gives rise to an identification
	\begin{align*}
	z(G \times G',V \otimes_A V') &= \sum_{(z)} z_{(1)}(G,V) \otimes z_{(2)}(G',V') \\
	&\in \bigoplus_{i+j=n} \opH^i(G,\End_A(V)^{(r)}) \otimes_A \opH^j(G',\End_A(V')^{(r)}) \\
	&=\opH^n(G \times G',\End_A(V \otimes_A V')^{(r)}).
	\end{align*}
\end{proposition}

\begin{proof}
The proof is a word-for-word repetition of the argument proving \cite[Proposition 3.8]{Suslin:1997}.
\end{proof}

\subsection{Reduction lemmas for characteristic classes} \label{subsection:reductionlemmas}

For the rest of this section, fix an integer $r \geq 1$, an inseparable $p$-polynomial $0 \neq f \in k[T]$, and a scalar $\eta \in k$. Assume that $f = \sum_{i=s}^t a_i T^{p^i}$ with $1 \leq s \leq t$, $a_s \neq 0$, and $a_t = 1$. 

\begin{notation}
Given $A \in \csalg_k$ and $(\ualpha|\beta) \in \bsVrfeta(\GLmn)(A)$, let $\Vuab$ denote the free $A$-super\-module $\Amn$, considered as a rational $\Mrfeta \otimes_k A$-supermodule via the correspondence of Proposition \ref{proposition:HomMrGLmn}.
\end{notation}

Set $\Grf = \Gar \times \cdots \times \G_{a(2)} \times \Monef$. If $(\ualpha|\beta) \in \bsVrf(\GLmn)(A)$, then $(\ualpha|\beta)$ also determines a group homomorphism $\wh{\rho}_{(\ualpha|\beta)}: \Grf \otimes_k A \rightarrow \GLmn \otimes_k A$, or equivalently, an $A\Grf$-supermodule structure on $\Amn$. The action of $A\Grf$ on $\Amn$ is described as follows: First, we make the identification of algebras
\[
A\Grf = (k\Gar \otimes k\G_{a(r-1)} \otimes \cdots \otimes k\G_{a(2)} \otimes k\Monef) \otimes_k A.
\]
By Proposition \ref{prop:groupalgebras}\eqref{item:qGar}, $k\Gal$ is generated by the commuting elements $u_0,\ldots,u_{\ell-1}$ subject to the relations $u_0^p = \cdots = u_{\ell-1}^p = 0$. Then for $2 \leq \ell \leq r$, the tensor factor $k\Gal$ acts on $\Amn$ with $u_0 \in k\Gal$ acting via the matrix $\alpha_{r-\ell}$ and with the remaining generators acting as zero. Finally, the tensor factor $k\Monef$ is generated by $u$ and $v$ subject to $u^p + v^2 = 0$ and $f(u) = 0$. Then $u$ acts via the matrix $\alpha_{r-1}$, and $v$ acts via the matrix $\beta$.

\begin{notation}
Given $A \in \csalg_k$ and $(\ualpha|\beta) \in \bsVrf(\GLmn)(A)$, let $\Wuab$ denote the free $A$-super\-module $\Amn$, considered as a rational $\Grf \otimes_k A$-supermodule as in the preceding paragraph.
\end{notation}

Now given $z \in \Ext_{\bsp}^\bullet(\bsir,\bsir)$ and $\Wuab$ as above, the characteristic class $z(\Grf \otimes_k A,\Wuab)$ identifies with an element of
\begin{align*}
\Matmn(A)^{(r)} \otimes_A \Hbul(\Grf \otimes_k A,A) &= \Matmn(A)^{(r)} \otimes_A \left( \Hbul(\Grf,k) \otimes_k A \right) \\
&= \Matmn(A)^{(r)} \otimes_k \Hbul(\Grf,k).
\end{align*}
As discussed in the first paragraph of Section \ref{subsection:phi}, the subspace
\[
H(\Grf,k) := \opH^{\ev}(\Grf,k)_{\zero} \oplus \opH^{\odd}(\Grf,k)_{\one}
\]
of $\Hbul(\Grf,k)$ is a commutative $k$-algebra in the ordinary sense.

The next lemma is a direct analogue of \cite[Lemma 4.3]{Suslin:1997}.


\begin{lemma} \label{lemma:freeimpliesall}
Let $P \in k[X_0,\ldots,X_{r-1}] \otimes H(\Grf,k)$ and $Q \in k[Y] \otimes H(\Grf,k)$ be polynomials with coefficients in the ring $H(\Grf,k)$. Given a purely even commutative algebra $A \in \calg_k$, and given $(\ualpha|\beta) = (\alpha_0,\ldots,\alpha_{r-1},\beta) \in \bsVrf(\GLmn)(A)$, let $\alpha_i = \sm{\dalpha_i & 0 \\ 0 & \ddalpha_i}$ and $\beta = \sm{0 & \dbeta \\ \ddbeta & 0}$ be the block decom\-positions of $\alpha_i$ and $\beta $ as in (\ref{eq:blockform}). By abuse of notation, we identify $\dalpha_i$ and $\ddalpha_i$ with $\sm{\dalpha_i & 0 \\ 0 & 0}$ and $\sm{0 & 0 \\ 0 & \ddalpha_i}$, and similarly for $\dbeta$ and $\ddbeta$. Now given $0 \leq j < p^r$, suppose that the formulas
\begin{align*}
\bse_r(j)(\Grf \otimes_k A,\Wuab) &= P(\dalpha_0^{(r)},\ldots,\dalpha_{r-1}^{(r)}), & \bsc_r(\Grf \otimes_k A,\Wuab) &= Q(\dbeta^{(r)}), \\
\bse_r^\Pi(j)(\Grf \otimes_k A,\Wuab) &= P(\ddalpha_0^{(r)},\ldots,\ddalpha_{r-1}^{(r)}), & \bsc_r^\Pi(\Grf \otimes_k A,\Wuab) &= Q(\ddbeta^{(r)}),
\end{align*}
in $\Matmn(A) \otimes \Hbul(\Grf,k)$ hold in the following special cases:
	\begin{enumerate}
	\item $A = k$, $\Wuab = k\Mrf$, and the left $k\Mrf$-super\-module structure on $\Wuab$ defined by $(\ualpha|\beta)$ is the left action of $k\Mrf$ on itself (i.e., $\alpha_i$ represents left multiplication by $u_i$, and $\beta$ represents left multiplication by $v$).
	
	\item $A = k$, $\Wuab = \Pi(k\Mrf)$, and the left $k\Mrf$-supermodule structure on $\Wuab$ defined by $(\ualpha|\beta)$ is the structure induced as in Remark \ref{remark:paritychange} by the left action of $k\Mrf$ on itself.
	\end{enumerate}
Then the formulas hold for any $A \in \calg_k$ and any $(\ualpha|\beta) \in \bsVrf(\GLmn)(A)$.
\end{lemma}

\begin{proof}
The argument is a direct generalization of the proof of \cite[Lemma 4.3]{Suslin:1997}. First, by hypothesis, the formulas hold if $A = k$ and if $\Wuab$ is one of the free rank-one $k\Mrf$-supermodules $k\Mrf$ or $\Pi(k\Mrf)$. Then it follows from Proposition \ref{prop:charclassproperties}\eqref{item:basechange} that the formulas hold if $A \in \calg_k$ is arbitrary and $\Wuab$ is one of the free rank-one $A\Mrf$-supermodules $A\Mrf$ or $\Pi(A\Mrf)$. Next, Proposition \ref{prop:charclassproperties}\eqref{item:directsum} implies that the formulas hold whenever $\Wuab$ is a finite direct sum of copies of $A\Mrf$ or $\Pi(A\Mrf)$. Finally, any finitely-generated $A\Mrf$-supermodule can be written as a quotient via an even homomorphism of a direct sum of copies of $A\Mrf$ and $\Pi(A\Mrf)$. If $(\ualpha'|\beta')$ is some other tuple defining a representation of $\Grf \otimes_k A$, and if $\phi: V \rightarrow W$ is a surjective even $A$-supermodule homomorphism between the underlying spaces for $(\ualpha|\beta)$ and $(\ualpha'|\beta')$ such that $\phi \circ \beta = \beta' \circ \phi$ and $\phi \circ \alpha_i = \alpha_i' \circ \phi$ for each $i$, then because $\phi$ is even and $A$ is purely even, one gets
\begin{align*}
\phi \circ \dbeta &= \dbeta' \circ \phi, & \phi \circ \dalpha_i &= \dalpha_i' \circ \phi, \\
\phi \circ \ddbeta &= \ddbeta' \circ \phi, & \phi \circ \ddalpha_i &= \ddalpha_i' \circ \phi.
\end{align*}
Then one can finish the proof by arguing exactly as in \cite[4.3.2]{Suslin:1997}, replacing \cite[Proposition 3.2(a)]{Suslin:1997} with Proposition \ref{prop:charclassproperties}\eqref{item:evenmap}.
\end{proof}

The case $r=1$ of Lemma \ref{lemma:freeimpliesall} concerns characteristic classes for $\Monef$.

\begin{lemma} \label{lemma:freeimpliesallMonefeta}
The assertion of Lemma \ref{lemma:freeimpliesall} also holds, via exactly the same line of reasoning, if $\Grf$ is replaced by $\Monefeta$ and $\Wuab$ is replaced by $\Vab$ with $(\alpha|\beta) \in \bsVonefeta(\GLmn)(A)$.
\end{lemma}

\begin{proposition} \label{proposition:erMrs}
Let $(\alpha|\beta) \in \bsV_{1;f}(\GLmn)(k)$, and let $\alpha = \sm{\dalpha & 0 \\ 0 & \ddalpha}$ be the block decomposition of $\alpha$ as in Lemma \ref{lemma:freeimpliesall}. Set $P_f(X) = \sum_{\ell=1}^{t-1} (a_{\ell+1}a_s^{-1})^{p^{r-1}} X^{p^\ell} \in k[X]$, where by convention $P_f(X) = 0$ if $t = 1$, and $a_j = 0$ for $1 \leq j < s$. Set $w_1 = x_1-y^2$ as in Remark \ref{remark:w1}. Then under the identification (\ref{eq:Extfloatout}), and identifying $\Hbul(\Monef,k)$ with $\Hbul(\Mones,k)$ as in Proposition \ref{prop:Mrcohomology}(\ref{item:Mrfcohomology}),
	\begin{subequations}
	\begin{align}
	\bse_r(\Monef,\Vab) &= \dalpha^{(r)} \otimes x_1^{p^{r-1}} + P_f(\dalpha^{(r)}) \otimes w_s^{p^{r-1}}, \label{eq:erM1f} \\
	\bse_r^\Pi(\Monef,\Vab) &= \ddalpha^{(r)} \otimes x_1^{p^{r-1}} + P_f(\ddalpha^{(r)}) \otimes w_s^{p^{r-1}}, \quad \text{and} \label{eq:erPiM1f} \\
	\bse_r(j)(\Monef,\Vab) &= \bse_r^\Pi(j)(\Monef,\Vab) = 0 \text{ for $0 < j < p^{r-1}$.}
	\end{align}
	\end{subequations}
\end{proposition}

\begin{proof}
Let $\g$ be the restricted Lie superalgebra generated by an even element $u$ and an odd element $v$ subject only to the relations $u^{[p]} + \frac{1}{2}[v,v] = 0$ and $\sum_{\ell = s}^t a_\ell u^{[p^\ell]} =0$. Then $V(\g) = k\Monef$ and $\Hbul(V(\g),k) = \Hbul(\Monef,k)$. Now as in \eqref{eq:Mayspecseq}, there exists a May spectral sequence
\begin{equation} \label{eq:M1Mayspecseq}
\begin{split}
E_0^{i,j} &= \Lambda_s^j(\g^\#) \otimes S^{i/2}(\gzero^\#[2])^{(1)} \Rightarrow \opH^{i+j}(\Monef,k), \\
E_1^{i,j} = E_2^{i,j} &= \opH^j(\g,k) \otimes S^{i/2}(\gzero^\#[2])^{(1)} \Rightarrow \opH^{i+j}(\Monef,k).
\end{split}
\end{equation}
In particular, $E_2^{0,1} = \opH^1(\g,k) \cong (\g/[\g,\g])^\#$, and the differential $d_2: E_2^{0,1} \rightarrow E_2^{2,0} \cong \gzero^{\#(1)}$ identifies with the transpose of the linear map ${\gzero}^{(1)} \rightarrow \g/[\g,\g]$ induced by the $p$-mapping $x \mapsto x^{[p]}$ on $\gzero$. This identification of the map $d_2: E_2^{0,1} \rightarrow E_2^{2,0}$ can be checked, for example, by using the explicit construction of the May spectral sequence described in \cite[\S\S3.4--3.5]{Drupieski:2013b}.

The naturality of the May spectral sequence implies the commutativity of the diagram below, in which the first pair of horizontal arrows are the $p^{r-1}$-power maps, and the second pair of horizontal arrows are the horizontal edge maps of the corresponding May spectral sequences:
\begin{equation} \label{eq:Mayspecseqcomdiag}
\vcenter{\xymatrix{
\glzero^{\#(r)} \ar@{->}[r] \ar@{->}[d]^{\rho_{(\alpha|\beta)}^{\#(r)}} & S^{p^{r-1}}(\glzero^\#)^{(1)} \ar@{->}[r] \ar@{->}[d]^{S^{p^{r-1}}(\rho_{(\alpha|\beta)}^\#)^{(1)}} & \opH^{2p^{r-1}}(\GLmnone,k) \ar@{->}[d]^{\rho_{(\alpha|\beta)}^*} \\
\gzero^{\#(r)} \ar@{->}[r] & S^{p^{r-1}}(\gzero^\#)^{(1)} \ar@{->}[r] & \opH^{2p^{r-1}}(\Monef,k),
}}
\end{equation}
By Convention \ref{convention}, the restriction of the top row to the subspaces $\gl_m^{\#(r)}$ and $\gl_n^{\#(r)}$ of $\glzero^{\#(r)}$ are the maps corresponding to $\bse_r(\GLmnone,\Vab)$ and $\bse_r^\Pi(\GLmnone,\Vab)$, respectively. Thus, to describe the characteristic classes $\bse_r(\Monef,\Vab)$ and $\bse_r^\Pi(\Monef,\Vab)$, it suffices to describe the images of these subspaces through the bottom row of the diagram.

Let $\{ X_{ij} \}$ and $\{ Y_{ij} \}$ be the bases of coordinate functions that are dual to the standard matrix unit bases of $\glzero$ and $\glone$, respectively. For $0 \leq i \leq t$, set $u_i = u^{[p^i]}$. Then the set $\{ v,u_0,\ldots,u_{t-1} \}$ is a homogeneous basis for $\g$; let $\{ v^*,u_0^*,\ldots,u_{t-1}^* \}$ be the dual basis. At the level of Lie superalgebras, $\rho_{(\alpha|\beta)}$ maps $u_i$ and $v$ to the matrices $\alpha^{p^i}$ and $\beta$, respectively. Then
\[
\rho_{(\alpha|\beta)}^{\#(r)}\left( X_{ij}^{(r)} \right) = \left( \alpha_{ij} \cdot u_0^* + (\alpha^p)_{ij} \cdot u_1^* + \cdots + (\alpha^{p^{t-1}})_{ij} \cdot u_{t-1}^* \right)^{(r)},
\]
where $(\alpha^{p^\ell})_{ij}$ denotes the $ij$-entry of the matrix $\alpha^{p^\ell}$. Then the image of $X_{ij}^{(r)}$ in $S^{p^{r-1}}(\gzero^\#)^{(1)}$ is
\[
[\alpha_{ij}]^{p^r} \cdot (u_0^{*(1)})^{p^{r-1}} + [(\alpha^p)_{ij}]^{p^r} \cdot (u_1^{*(1)})^{p^{r-1}} + \cdots + [(\alpha^{p^{t-1}})_{ij}]^{p^r} \cdot (u_{t-1}^{*(1)})^{p^{r-1}}.
\]

Next, since $[\g,\g]$ is spanned by $u_1$, the space $E_2^{0,1} \cong (\g/[\g,\g])^\#$ is spanned by $v^*,u_0^*$, and (if $t \geq 3$) $u_2^*,\ldots,u_{t-1}^*$. Then by the description of the differential $d_2: E_2^{0,1} \rightarrow E_2^{2,0}$ given above, and since $u_t = -\sum_{\ell=s}^{t-1} a_\ell u_\ell$ by the assumption that $a_t = 1$, one has $d_2(v^*) = d_2(u_0^*) = 0$, and
\[
d_2(u_\ell^*) = u_{\ell-1}^{*(1)} - a_\ell \cdot u_{t-1}^{*(1)} \quad \text{for $2 \leq \ell \leq t-1$.}
\]
So $u_\ell^{*(1)} \equiv a_{\ell+1} \cdot u_{t-1}^{*(1)} \mod \im(d_2)$ for $1 \leq \ell \leq t-2$. Since $a_t = 1$, the previous congruence also holds for $\ell = t-1$. Now recall that the horizontal edge map $E_2^{\bullet,0} \rightarrow \Hbul(\Monef,k)$ factors through the quotient $E_2^{\bullet,0} / \im(d_2)$. Then by the multiplicativity of the May spectral sequence, it follows that the image of $X_{ij}^{(r)}$ through the bottom row of \eqref{eq:Mayspecseqcomdiag} is the equal to the image under the horizontal edge map $E_2^{\bullet,0} \rightarrow \Hbul(\Monef,k)$ of
\begin{multline*} \label{eq:Xijimage}
[\alpha_{ij}]^{p^r} \cdot (u_0^{*(1)})^{p^{r-1}} +  \sum_{\ell=1}^{t-1} \left( [(\alpha^{p^\ell})_{ij}]^{p^r} (a_{\ell+1})^{p^{r-1}} \cdot (u_{t-1}^{*(1)})^{p^{r-1}} \right) \\
= [\alpha^{(r)}]_{ij} \cdot (u_0^{*(1)})^{p^{r-1}} + P_f(\alpha^{(r)})_{ij} \cdot (a_s \cdot u_{t-1}^{*(1)})^{p^{r-1}},
\end{multline*}
where $[\alpha^{(r)}]_{ij}$ and $P_f(\alpha^{(r)})_{ij}$ denote, respectively, the $ij$-entries of the matrices $\alpha^{(r)} \in \Matmn(k)^{(r)}$ and $P_f(\alpha^{(r)}) \in \Matmn(k)^{(r)}$. It follows from the construction of the May spectral sequence given in the proof of \cite[Proposition 3.4.2]{Drupieski:2013b}, and from the explicit descriptions of the generators in Proposition \ref{prop:Mrcohomology}, that the horizontal edge map $E_2^{\bullet,0} \rightarrow \Hbul(\Monef,k)$ sends $u_0^{*(1)}$ to $x_1 \in \opH^2(\Monef,k)$. Then to finish the proof of \eqref{eq:erM1f} and \eqref{eq:erPiM1f}, it suffices to show for $t \geq 2$ that the horizontal edge map sends $a_s \cdot u_{t-1}^{*(1)} \in E_2^{2,0}$ to $w_s \in \opH^2(\Monef,k)$.

The cohomology class $w_s$ is represented by the cochain
\[ \textstyle
z:= -(\sum_{j=1}^{p^s-1} \sigma_j \otimes \sigma_{p^s-j} + \sum_{u+v+p=p^s} \sigma_u \tau \otimes \sigma_v \tau) \in k[\Monef]^{\otimes 2}.
\]
Suppose for the moment that $s \geq 2$. Then
\[
z \equiv -\left(\sum_{j=1}^{p-1} \sigma_{j p^{s-1}} \otimes \sigma_{(p-j)p^{s-1}}\right) \mod F^{p+1} \left(k[\Monef]^{\otimes 2}\right),
\]
where $F^\bullet$ denotes the filtration of $k[\Monef]^{\otimes 2}$ discussed prior to Lemma \ref{lemma:boundarymodfiltration}. Next observe that, in the notation of \cite[(3.4.8)]{Drupieski:2013b},
\[
\beta(\sigma_{p^{s-1}}) := \sum_{j=1}^{p-1} \frac{(p-1)!}{j!(p-j)!} (\sigma_{p^{s-1}})^j \otimes (\sigma_{p^{s-1}})^{p-j} = -\sum_{j=1}^{p-1} \sigma_{j p^{s-1}} \otimes \sigma_{(p-j)p^{s-1}}.
\]
Since the image of $\sigma_{p^{s-1}}$ in $I_\ve/(I_\ve)^2 \cong \Lie(\Monef)^\#$ is the dual basis vector $u_{s-1}^*$, this implies by the construction of the May spectral sequence in \cite[Proposition 3.4.2]{Drupieski:2013b} (and the reindexing in \cite[\S3.6]{Drupieski:2013b}) that the cochain $z$ represents $u_{s-1}^{*(1)} \in E_1^{2,0}$, and hence the horizontal edge map sends $u_{s-1}^{*(1)}$ to $w_s$. Then applying the congruence $u_\ell^{*(1)} \equiv a_{\ell+1} \cdot u_{t-1}^{*(1)} \mod \im(d_2)$ in the case $\ell = s-1$, it follows that the horizontal edge map sends $a_s \cdot u_{t-1}^{*(1)}$ to $w_s$.

Now suppose that $t \geq 2$ and $s = 1$. Let $\partial$ denote the differential on the Hochschild complex $C^\bullet(\Monef,k)$. Then by Lemma \ref{lemma:boundarymodfiltration},
\[
-\partial(\sigma_p) \equiv \left(\sum_{j=1}^{p-1} \sigma_j \otimes \sigma_{p-j} + \tau \otimes \tau \right) - a_1 \cdot \left( \sum_{j=1}^{p-1} \sigma_{jp^{t-1}} \otimes \sigma_{(p-j)p^{t-1}} \right) \mod F^{p+1} \left(k[\Monef]^{\otimes 2}\right).
\]
Then in the notation of the previous paragraph, $-\partial(\sigma_p) \equiv -z + a_1 \cdot \beta(\sigma_{p^{t-1}}) \mod F^{p+1}\left(k[\Monef]^{\otimes 2}\right)$, or equivalently, $a_1 \cdot \beta(\sigma_{p^{t-1}}) \equiv z + \partial(\sigma_p) \mod F^{p+1}\left(k[\Monef]^{\otimes 2}\right)$. This implies by the construction of the May spectral sequence that the horizontal edge map sends the cohomology class of $a_1 \cdot \beta(\sigma_{p^{t-1}})$ to the cohomology class of $z$, i.e., sends $a_1 \cdot u_{t-1}^{*(1)}$ to $x_1 - y^2$.

Finally, for $0 < j < p^{r-1}$, the triviality of $\bse_r(j)(\GLmnone,\Vab)$ and $\bse_r^\Pi(j)(\GLmnone,\Vab)$, and hence the triviality of $\bse_r(j)(\Monef,\Vab)$ and $\bse_r^\Pi(j)(\Monef,\Vab)$, can be established by a direct adaptation of the argument proving \cite[Lemma 4.2(b)]{Suslin:1997}.
\end{proof}

\begin{proposition} \label{proposition:erMrfeta}
Suppose $0 \neq \eta \in k$. Let $A \in \calg_k$ be a purely even commutative $k$-algebra, let $(\alpha|\beta) \in \bsVonefeta(\GLmn)(A)$, and let $\alpha = \sm{\dalpha & 0 \\ 0 & \ddalpha}$ be the block decomposition of $\alpha$ as in Lemma \ref{lemma:freeimpliesall}. Then under the identification (\ref{eq:Extfloatout}), and identifying $\Hbul(\Monefeta,k)$ with $\Hbul(\Ga^-,k) \cong k[y]$ as in Proposition \ref{prop:Mrcohomology}(\ref{item:M1fetacohomology}),
	\begin{align*}
	\bse_r(\Monefeta,\Vab) &= \dalpha^{(r)} \otimes y^{2p^{r-1}} \\
	\bse_r^\Pi(\Monefeta,\Vab) &= \ddalpha^{(r)} \otimes y^{2p^{r-1}}, \quad \text{and} \\
	\bse_r(j)(\Monefeta,\Vab) &= \bse_r^\Pi(j)(\Monef,\Vab) = 0 \text{ for $0 < j < p^{r-1}$.}
	\end{align*}
\end{proposition}

\begin{proof}
By Lemma \ref{lemma:freeimpliesallMonefeta} it suffices to assume that $A = k$. Then the proof is a modification of the argument used to establish Proposition \ref{proposition:erMrs}. Specifically, set $a_0 = \eta$. Then $k\Monefeta \cong V(\g)$, where $\g$ is the restricted Lie superalgebra generated by an even element $u$ and an odd element $v$ subject only to the relations $u^{[p]} + \frac{1}{2}[v,v] = 0$ and $\sum_{\ell=0}^t a_\ell u^{[p^\ell]} = 0$. Now the argument proceeds in precisely the same manner as in the first three (and last) paragraphs of the proof of Proposition \ref{proposition:erMrs}. In the fourth paragraph one now has $u^t = -\sum_{\ell=0}^{t-1} a_\ell u_\ell$, so it follows that $d_2(u_0^*) = a_0 \cdot u_{t-1}^{*(1)}$. Since $a_0 \neq 0$ by assumption, this implies that the image of $u_{t-1}^{*(1)}$ under the horizontal edge map is $0$. Then reasoning as in the fourth paragraph of the proof of Proposition \ref{proposition:erMrs}, it suffices to show that the horizontal edge map sends $u_0^{*(1)} \in E_2^{2,0}$ to the cohomology class $y^2 \in \opH^2(\Monefeta,k)$. Explicitly, by the construction of the May spectral sequence, $u_0^{*(1)}$ is represented by the cocycle
\[ \textstyle
\beta(\sigma_1) := \sum_{j=1}^{p-1} \frac{1}{j! (p-j)!} (\sigma_1)^j \otimes (\sigma_1)^{p-j} = \sum_{j=1}^{p-1} \sigma_j \otimes \sigma_{p-j} \in k[\Monefeta]^{\otimes 2}.
\]
Then by the same reasoning as in the second-to-last paragraph of the proof of Proposition \ref{proposition:erMrs},
\[
a_1 \cdot \beta(\sigma_{p^{t-1}}) \equiv -\beta(\sigma_1) + \tau \otimes \tau + \partial(\sigma_p) \mod F^{p+1}(k[\Monefeta]^{\otimes 2}).
\]
Since $\beta(\sigma_{p^{t-1}})$ is a cocycle representative for $u_{t-1}^{*(1)}$, and since $u_{t-1}^{*(1)}$ maps to $0$ in $\Hbul(\Monefeta,k)$, this implies that $\beta(\sigma_1)$ and $\tau \otimes \tau$ represent the same cohomology class in $\opH^2(\Monefeta,k)$, and hence that the horizontal edge map sends $u_0^{*(1)} \in E_2^{2,0}$ to $y^2 \in \opH^2(\Monefeta,k)$.
\end{proof}

\begin{remark}
Let $\eta \in k$ be arbitrary. Let $0 \leq j < p^r$, and let $j = \sum_{i=0}^{r-1} j_i p^i$ be the $p$-adic decomposition of $j$. Since $\bse_r(\ell)(\Monefeta,\Vab) = 0$ for $0 < \ell < p^{r-1}$, it follows that $\bse_r(j)(\Monefeta,\Vab) = 0$ unless $j \equiv 0 \mod p^{r-1}$, and similarly for $\bse_r^\Pi(j)(\Monefeta,\Vab)$.
\end{remark}

\begin{lemma} \label{lemma:crMonefeta}
Let $\eta \in k$. Let $A \in \calg_k$ be a purely even commutative $k$-superalgebra, let $(\alpha|\beta) \in \bsVonefeta(\GLmn)(A)$, and let $\beta = \sm{0 & \dbeta \\ \ddbeta & 0}$ be the block decomposition of $\beta$ as in Lemma \ref{lemma:freeimpliesall}. Then under the identification (\ref{eq:Extfloatout}), and under the identifications of Proposition \ref{prop:Mrcohomology}(\ref{item:Mrfcohomology}) and (\ref{item:M1fetacohomology}),
\begin{align*}
\bsc_r(\Monefeta,\Vab) = \dbeta^{(r)} \otimes y^{p^r} \quad \text{and} \quad \bsc_r^\Pi(\Monefeta,\Vab) = \ddbeta^{(r)} \otimes y^{p^r}.
\end{align*}
\end{lemma}

\begin{proof}
Retain the notational conventions of the two preceding proofs, so that $k\Monefeta \cong V(\g)$. By Lemma \ref{lemma:freeimpliesall}, we may assume that $A = k$ and that $\g$ is a restricted subalgebra of $\glmn$. Now recall from \cite[\S5.2]{Drupieski:2016a} that there exists a $V(\g)$-free resolution $(X(\g),d_t)$ of the trivial $V(\g)$-module $k$ such that, as graded superspaces,
\begin{equation} \label{eq:Xgiso}
\Hom_{V(\g)}(X(\g),k) \cong \Lambda_s(\g^\#) \otimes S(\gzero^\#[2])^{(1)}
\end{equation}
Then by \cite[Lemma 5.2.4]{Drupieski:2016a} and Convention \ref{convention}, the homomorphism
\[
(\bsc_r + \bsc_r^\Pi): S(\glone^\#[p^r])^{(r)} \rightarrow \Hbul(V(\g),k)
\]
is equal to the composite homomorphism of graded superalgebras
\[
S(\glone^\#[p^r])^{(r)} \rightarrow S(\gone^\#[p^r])^{(r)} \rightarrow \Hbul(V(\g),k),
\]
where the first arrow is induced by restriction from $\glmn$ to $\g$, and the second arrow is induced via \eqref{eq:Xgiso} by the $p^r$-power map $S(\gone^\#[p^r])^{(r)} \rightarrow S(\gone^\#)$ and the inclusion $S(\gone^\#) \hookrightarrow \Lambda_s(\g^\#)$. Restriction from $\glmn$ to $\g$ sends the odd coordinate function $Y_{ij} \in \glone^\#$ to $\beta_{ij} \cdot v^*$, and one can check via the comparison results of \cite[\S3.5]{Drupieski:2013b} that the cohomology class in $\Hbul(V(\g),k)$ of $(v^*)^{p^r}$ identifies with the cohomology class in $\Hbul(\Monefeta,k)$ of $\tau^{\otimes p^r} \in k[\Monefeta]^{\otimes p^r}$, i.e., identifies with $y^{p^r}$. Then $(\bsc_r + \bsc_r^\Pi)(Y_{ij}^{(r)}) = (\beta_{ij})^{p^r} \cdot y^{p^r}$, which implies the result.
\end{proof}

\begin{lemma} \label{lemma:erGalcrGal}
Let $A \in \calg_k$ be a purely even commutative superalgebra. Then for any $0 < \ell \leq r$, any $0 \leq j < p^r$, and any $p$-nilpotent matrix $\alpha = \sm{\dalpha & 0 \\ 0 & \ddalpha} \in \Matmn(A)_{\zero}$, one has in the notation of \cite[Theorem 1.13(4)]{Suslin:1997},
	\begin{align*}
	\bse_r(j)(\Gal \otimes_k A,V_\alpha) &= (\dalpha^{(r)})^{s(j)} \otimes \frac{x_r^{j_0} x_{r-1}^{pj_1} \cdots x_1^{p^{r-1}j_{r-1}}}{(j_0!)(j_1!)\cdots (j_{r-1}!)}, \\
	\bse_r^\Pi(j)(\Gal \otimes_k A,V_\alpha) &= (\ddalpha^{(r)})^{s(j)} \otimes \frac{x_r^{j_0} x_{r-1}^{p j_1} \cdots x_1^{p^{r-1}j_{r-1}}}{(j_0!)(j_1!) \cdots (j_{r-1}!)}, \text{ and} \\
	\bsc_r(\Gal \otimes_k A,V_\alpha) &= \bsc_r^\Pi(\Gal \otimes_k A,V_\alpha) = 0,
	\end{align*}
where $j = \sum_{i=0}^{r-1} j_i p^i$ is the $p$-adic expansion of $j$, $s(j) = \sum_{i=0}^{r-1} j_i$, and $x_i$ is interpreted to be $0$ if $i > \ell$. In particular, $\bse_r(j)(\Gal \otimes_k A,V_\alpha) = \bse_r^\Pi(j)(\Gal \otimes_k A,V_\alpha) = 0$ unless $j \equiv 0 \mod p^{r-\ell}$.
\end{lemma}

\begin{proof}
First observe that the block decomposition of $\alpha$ gives rise to the direct sum decomposition $V_\alpha = V_{\dalpha} \oplus V_{\ddalpha}$, where $V_{\dalpha} = A^{m|0}$ considered as a rational $\Gal$-module via $\dalpha$, and $V_{\ddalpha} = A^{0|n}$ considered as a rational $\Gal \otimes_k A$-module via $\ddalpha$. Then Proposition \ref{prop:charclassproperties}\eqref{item:directsum} allows us to reduce to the case where $V_\alpha$ is a homogeneous superspace. Thus, for the rest of the proof let us assume that $V_\alpha$ is purely even and $\alpha = \dalpha$; the details for the purely odd case are entirely analogous. First, since $V_\alpha$ is purely even, one has $\bsirone(V_\alpha) = 0$, and hence $\bse_r^\Pi(j)(\Gal \otimes_k A,V_\alpha) = \bsc_r(\Gal \otimes_k A,V_\alpha) = \bsc_r(\Gal \otimes_k A,V_\alpha) = 0$. Next, there exists a commutative diagram as follows:
\[
\xymatrix{
\Ext_{\bsp}^\bullet(\bsir_{\zero},\bsir_{\zero}) \otimes_k A \ar@{->}[r]^{\sim} \ar@{->}[d]^{\res \otimes 1_A} & \Ext_{\bsp_A}^\bullet((\bsir_{\zero})_A,(\bsir_{\zero})_A) \ar@{->}[r] \ar@{->}[d]^{\res} & \Mat_{m|0}(A) \otimes_A \Hbul(\Gal \otimes_k A, A) \ar@{=}[d] \\
\Ext_{\cp}^\bullet(I^{(r)},I^{(r)}) \otimes_k A \ar@{->}[r]^{\sim} & \Ext_{\cp_A}^\bullet(I_A^{(r)},I_A^{(r)}) \ar@{->}[r] & M_m(A) \otimes_A \Hbul(\Gal \otimes_k A, A).
}
\]
The first vertical arrow is induced by restriction from the category $\bsp$ of strict polynomial superfunctors to the category $\cp$ of ordinary strict polynomial functors (cf.\ \cite[\S2.1]{Drupieski:2016}), and the second vertical arrows is the analogous restriction map from $\bsp_A$ to $\cp_A$. The first pair of horizontal arrows are the base-change isomorphisms of Corollary \ref{cor:functorbasechange} and \cite[Corollary 2.7]{Suslin:1997}, and the second pair of horizontal arrows are the restriction maps of Lemma \ref{lemma:restrictionGV} and \cite[Proposition 3.2]{Suslin:1997}. The leftmost vertical arrow is an isomorphism in cohomological degrees less than $2p^r$, so the calculation of $\bse_r(j)(\Gal \otimes_k A,V_\alpha)$ now follows from \cite[Theorem 4.7]{Suslin:1997}.
\end{proof}

\subsection{Calculation of characteristic classes}\label{SS:calculationofcharacteristicclasses}

In this section again fix $r \geq 1$, fix an inseparable $p$-polynomial $0 \neq f \in k[T]$, and let $\eta \in k$. Again assume that $f = \sum_{i=s}^t a_i T^{p^i}$ with $1 \leq s \leq t$, $a_s \neq 0$, and $a_t = 1$, and set $\Grf = \Gar \times \cdots \times \G_{a(2)} \times \Monef$. Let $A \in \calg_k$ be a purely even commutative $k$-superalgebra, and let $(\ualpha|\beta) \in \Vrfeta(\GLmn)(A)$. Our goal in this section is to finish describing the characteristic classes for $\Mrfeta$ corresponding to the representation $\Vuab$ introduced at the beginning of the preceding section. For $r = 1$ and $\eta \neq 0$, the relevant calculations are already given by Proposition \ref{proposition:erMrfeta} and Lemma \ref{lemma:crMonefeta}.

\begin{theorem} \label{theorem:charclassesWuab}
Let $A \in \calg_k$ be a purely even commutative $k$-superalgebra, and let $(\ualpha|\beta) \in \bsVrf(\GLmn)(A)$. Write the block decompositions of each $\alpha_i$ and $\beta$ as in Lemma \ref{lemma:freeimpliesall}. Let $P_f(X) \in k[X]$ be as in Proposition \ref{proposition:erMrs}, and set $w_1 = x_1 - y^2$. Then for each $0 \leq j < p^r$,
\begin{subequations}
\begin{multline} \label{eq:erjWuab}
\bse_r(j)(\Grf \otimes_k A,\Wuab) = \\
\sum_{\substack{j_0+\cdots+j_{r-1}=j \\ j_0,\ldots,j_{r-1} \geq 0 \\ j_i \equiv 0 \mod p^i}} \sum_{\substack{c,d \geq 0 \\ c+d = j_{r-1,r-1}}}
\Bigg( \binom{j_{r-1,r-1}}{c} (\dalpha_0^{(r)})^{s(j_0)} \cdots (\dalpha_{r-2}^{(r)})^{s(j_{r-2})} (\dalpha_{r-1}^{(r)})^c (P_f(\dalpha_{r-1}^{(r)}))^d \\
\otimes \frac{x_r^{j_{0,0}} x_{r-1}^{p j_{0,1}} \cdots x_1^{p^{r-1} j_{0,r-1}}}{(j_{0,0}!)(j_{0,1}!) \cdots (j_{0,r-1}!)}
\otimes \cdots
\otimes \frac{x_2^{p^{r-2} j_{r-2,r-2}} x_1^{p^{r-1} j_{r-2,r-1}}}{(j_{r-2,r-2}!) (j_{r-2,r-1}!)}
\otimes \frac{x_1^{p^{r-1} c} w_s^{p^{r-1}d}}{j_{r-1,r-1}!} \Bigg),
\end{multline}
\begin{multline} \label{eq:erjPiWuab}
\bse_r^\Pi(j)(\Grf \otimes_k A,\Wuab) = \\
\sum_{\substack{j_0+\cdots+j_{r-1}=j \\ j_0,\ldots,j_{r-1} \geq 0 \\ j_i \equiv 0 \mod p^i}} \sum_{\substack{c,d \geq 0 \\ c+d = j_{r-1,r-1}}}
\Bigg( \binom{j_{r-1,r-1}}{c} (\ddalpha_0^{(r)})^{s(j_0)} \cdots (\ddalpha_{r-2}^{(r)})^{s(j_{r-2})} (\ddalpha_{r-1}^{(r)})^c (P_f(\ddalpha_{r-1}^{(r)}))^d \\
\otimes \frac{x_r^{j_{0,0}} x_{r-1}^{p j_{0,1}} \cdots x_1^{p^{r-1} j_{0,r-1}}}{(j_{0,0}!)(j_{0,1}!) \cdots (j_{0,r-1}!)}
\otimes \cdots
\otimes \frac{x_2^{p^{r-2} j_{r-2,r-2}} x_1^{p^{r-1} j_{r-2,r-1}}}{(j_{r-2,r-2}!) (j_{r-2,r-1}!)}
\otimes \frac{x_1^{p^{r-1} c} w_s^{p^{r-1}d}}{j_{r-1,r-1}!} \Bigg),
\end{multline}
\begin{align}
\bsc_r(\Grf \otimes_k A, \Wuab) &= \dbeta^{(r)} \otimes 1 \otimes \cdots \otimes 1 \otimes y^{p^r}, \text{ and} \label{eq:crWuab} \\
\bsc_r^\Pi(\Grf \otimes_k A, \Wuab) &= \ddbeta^{(r)} \otimes 1 \otimes \cdots \otimes 1 \otimes y^{p^r}. \label{eq:crPiWuab}
\end{align}
\end{subequations}
where $j_i = \sum_{\ell=0}^{r-1} j_{i,\ell}p^\ell$ is the $p$-adic decomposition of $j_i$, and $s(j_i) = \sum_{\ell=0}^{r-1} j_{i,\ell}$.
\end{theorem}

\begin{proof}
Given a superspace $V$ and $a \in \set{0,1}$, set $\Pi^a(V) = V$ if $a = 0$, and set $\Pi^a(V) = \Pi(V)$ if $a = 1$. Then by Lemma \ref{lemma:freeimpliesall}, it suffices to assume that:
	\begin{itemize}
	\item $A = k$;
	\item the underlying superspace of $\Wuab$ is
		\[
		\Pi^a(k\Mrf) \cong k[u_0]/\subgrp{u_0^p} \otimes \cdots \otimes k[u_{r-2}]/\subgrp{u_{r-2}^p} \otimes \Pi^a \left( k[u_{r-1},v]/\subgrp{u_{r-1}^p+v^2,f(u_{r-1})} \right),
		\]
	with either $a = 0$ or $a = 1$; and
	\item for $0 \leq i \leq r-1$, the matrix $\alpha_i$ represents the left action of $1^{\otimes i} \otimes u_i \otimes 1^{\otimes (r-1-i)}$, and the matrix $\beta$ represents the left action of $1^{\otimes (r-1)} \otimes v$.
	\end{itemize}
As in Lemma \ref{lemma:k[Mrf]algebra}, a homogeneous basis for $k\Mrf$ is given by the set
\[
\set{u_0^{i_0} \cdots u_{r-2}^{i_{r-2}} u_{r-1}^{i_{r-1}} v^j: 0 \leq i_\ell < p \text{ for $0 \leq \ell \leq r-2$}; 0 \leq i_{r-1} < p^t; 0 \leq j \leq 1}.
\]
In particular, right multiplication by $v$ defines an odd isomorphism $(k\Mrf)_{\zero} \simeq (k\Mrf)_{\one}$. With this identification, it immediately follows for each $i$ that $\dalpha_i = \ddalpha_i$.
	
For $0 \leq i \leq r-2$, let $V_{\alpha_i}$ denote the space $k[u_i]/\subgrp{u_i^p}$ considered as a $k\G_{a(r-i)}$-module such that the generator $u_0 \in k\G_{a(r-i)}$ acts on $k[u_i]/\subgrp{u_i^p}$ as left multiplication by $u_i$, and such that the remaining generators $u_1,\ldots,u_{r-i-1} \in k\G_{a(r-i)}$ act as $0$. Similarly, let $V_{(\alpha_{r-1}|\beta)} = \Pi^a ( k\Monef )$,  considered as a $k\Monef$-supermodule as in Remark \ref{remark:paritychange} via the left action of $k\Monef$ on itself. Let $z \in \Ext_{\bsp}^\bullet(\bsir,\bsir)$. Then by Proposition \ref{proposition:Kunneth},
\[
z(\Grf,\Wuab) = \sum z_{(0)}(\Gar,V_{\alpha_0}) \otimes \cdots \otimes z_{(r-2)}(\G_{a(2)},V_{\alpha_{r-2}}) \otimes z_{(r-1)}(\Monef,V_{(\alpha_{r-1}|\beta)}),
\]
where $\Delta^r(z) = \sum z_{(0)} \otimes \cdots \otimes z_{(r-2)} \otimes z_{(r-1)} \in \Ext_{\bsp}^\bullet(\bsir,\bsir)^{\otimes r}$ denotes the $r$-fold iterated coproduct of $z$. Now as in the proof of \cite[Corollary 4.8]{Suslin:1997}, we apply the explicit description of the coproduct given in Proposition \ref{proposition:coproduct} to reduce to the calculations of Proposition \ref{proposition:erMrs} and Lemmas \ref{lemma:crMonefeta} and \ref{lemma:erGalcrGal}. For example, if $0 \leq j < p^r$ and if $z = \bse_r(j)$ (resp., if $z = \bse_r^\Pi(j)$), then $\Delta^r(z)$ can be written as a sum of monomials $z_{(1)} \otimes \cdots \otimes z_{(r)}$ such that for each $i$ either $z_{(i)} = \bse_r(j_i)$ or $z_{(i)} = \bse_r^\Pi(j_i)$, with an even (resp.\ odd) number of the $z_{(i)}$ of the latter form, and $j_0+ \cdots + j_{r-1} = j$. For $0 \leq i \leq r-2$ the space $V_{\alpha_i}$ is purely even, hence $\bse_r^\Pi(j_i)(\G_{a(i)},V_{\alpha_{r-i}}) = 0$ by Lemma \ref{lemma:erGalcrGal}. Then it follows that
\begin{multline*}
\bse_r(j)(\Grf,\Wuab) = \\
\sum_{j_0+\cdots+j_{r-1} = j} \bse_r(j_0)(\Gar,V_{\alpha_0}) \otimes \cdots \otimes \bse_r(j_{r-2})(\G_{a(2)},V_{\alpha_{r-2}}) \otimes \bse_r(j_{r-1})(\Monef,V_{(\alpha_{r-1}|\beta)}),
\end{multline*}
and
\begin{multline*}
\bse_r^\Pi(j)(\Grf,\Wuab) = \\
\sum_{j_0+\cdots+j_{r-1} = j} \bse_r(j_0)(\Gar,V_{\alpha_0}) \otimes \cdots \otimes \bse_r(j_{r-2})(\G_{a(2)},V_{\alpha_{r-2}}) \otimes \bse_r^\Pi(j_{r-1})(\Monef,V_{(\alpha_{r-1}|\beta)}).
\end{multline*}
Now applying Proposition \ref{proposition:erMrs} and Lemma \ref{lemma:erGalcrGal}, and the observation $\dalpha_i = \ddalpha_i$, the formulas \eqref{eq:erjWuab} and \eqref{eq:erjPiWuab} follow. Arguing similarly in the cases $z = \bsc_r$ and $z = \bsc_r^\Pi$, one gets
	\begin{align*}
	\bsc_r(\Grf,\Wuab) &= \bse_0(\Gar,V_{\alpha_0}) \otimes \cdots \otimes \bse_0(\G_{a(2)},V_{\alpha_{r-2}}) \otimes \bsc_r(\Monef,V_{(\alpha_{r-1}|\beta)}) \\
	&= 1 \otimes \cdots \otimes 1 \otimes \bsc_r(\Monef,V_{(\alpha_{r-1}|\beta)}), &\text{and} \\
	\bsc_r^\Pi(\Grf,\Wuab) &= \bse_0(\Gar,V_{\alpha_0}) \otimes \cdots \otimes \bse_0(\G_{a(2)},V_{\alpha_{r-2}}) \otimes \bsc_r^\Pi(\Monef,V_{(\alpha_{r-1}|\beta)}) \\
	&= 1 \otimes \cdots \otimes 1 \otimes \bsc_r^\Pi(\Monef,V_{(\alpha_{r-1}|\beta)}),
	\end{align*}
so the formulas \eqref{eq:crWuab} and \eqref{eq:crPiWuab} then follow from Lemma \ref{lemma:crMonefeta}.
\end{proof}

Suppose $(\ualpha|\beta) \in \bsVrf(\GLmn)(A)$. From the description of the homomorphism $\wh{\rho}_{(\ualpha|\beta)}$ at the beginning of Section \ref{subsection:reductionlemmas}, it follows that $\rho_{(\ualpha|\beta)}$ is the pullback of $\wh{\rho}_{(\ualpha|\beta)}$ via the homomorphism
\begin{equation} \label{eq:factorquotient}
q \times (q \circ \bsF) \times \cdots \times (q \circ \bsF^{r-2}) \times \bsF^{r-1}: \Mrf \rightarrow \Gar \times \G_{a(r-1)} \times \cdots \times \G_{a(2)} \times \Monef.
\end{equation}

\begin{theorem}  \label{theorem:charclassesVuab}
Let $A \in \calg_k$ be a purely even commutative $k$-superalgebra, and let $(\ualpha|\beta) \in \bsVrf(\GLmn)(A)$. Write the block decompositions of each $\alpha_i$ and $\beta$ as in Lemma \ref{lemma:freeimpliesall}. Let $P_f(X) \in k[X]$ be as in Proposition \ref{proposition:erMrs}, and set $w_1 = x_1-y^2$. Then for each $0 \leq j < p^r$,
\begin{subequations}
\begin{multline} \label{eq:erjVuab}
\bse_r(j)(\Mrf \otimes_k A,\Vuab) = \\
\sum_{\substack{j_0+\cdots+j_{r-1}=j \\ j_0,\ldots,j_{r-1} \geq 0 \\ j_i \equiv 0 \mod p^i}} \sum_{\substack{c,d \geq 0 \\ c+d = j_{r-1,r-1}}}
\Bigg( \binom{j_{r-1,r-1}}{c} (\dalpha_0^{(r)})^{s(j_0)} \cdots (\dalpha_{r-2}^{(r)})^{s(j_{r-2})} (\dalpha_{r-1}^{(r)})^c (P_f(\dalpha_{r-1}^{(r)}))^d \\
\otimes \frac{x_r^{j_{0,0}} x_{r-1}^{p j_{0,1}} \cdots x_1^{p^{r-1} j_{0,r-1}}}{(j_{0,0}!)(j_{0,1}!) \cdots (j_{0,r-1}!)}
\cdots
\frac{x_r^{p^{r-2} j_{r-2,r-2}} x_{r-1}^{p^{r-1} j_{r-2,r-1}}}{(j_{r-2,r-2}!) (j_{r-2,r-1}!)}
\cdot \frac{x_r^{p^{r-1} c} w_s^{p^{r-1}d}}{j_{r-1,r-1}!} \Bigg),
\end{multline}
\begin{multline} \label{eq:erjPiVuab}
\bse_r^\Pi(j)(\Mrf \otimes_k A,\Vuab) = \\
\sum_{\substack{j_0+\cdots+j_{r-1}=j \\ j_0,\ldots,j_{r-1} \geq 0 \\ j_i \equiv 0 \mod p^i}} \sum_{\substack{c,d \geq 0 \\ c+d = j_{r-1,r-1}}}
\Bigg( \binom{j_{r-1,r-1}}{c} (\ddalpha_0^{(r)})^{s(j_0)} \cdots (\ddalpha_{r-2}^{(r)})^{s(j_{r-2})} (\ddalpha_{r-1}^{(r)})^c (P_f(\ddalpha_{r-1}^{(r)}))^d \\
\otimes \frac{x_r^{j_{0,0}} x_{r-1}^{p j_{0,1}} \cdots x_1^{p^{r-1} j_{0,r-1}}}{(j_{0,0}!)(j_{0,1}!) \cdots (j_{0,r-1}!)}
\cdots
\frac{x_r^{p^{r-2} j_{r-2,r-2}} x_{r-1}^{p^{r-1} j_{r-2,r-1}}}{(j_{r-2,r-2}!) (j_{r-2,r-1}!)}
\cdot \frac{x_r^{p^{r-1} c} w_s^{p^{r-1}d}}{j_{r-1,r-1}!} \Bigg),
\end{multline}
\begin{align}
\bsc_r(\Mrf \otimes_k A, \Vuab) &= \dbeta^{(r)} \otimes y^{p^r}, \text{ and} \label{eq:crVuab} \\
\bsc_r^\Pi(\Mrf \otimes_k A, \Vuab) &= \ddbeta^{(r)} \otimes y^{p^r}. \label{eq:crPiVuab}
\end{align}
\end{subequations}
where $j_i = \sum_{\ell=0}^{r-1} j_{i,\ell}p^\ell$ is the $p$-adic decomposition of $j_i$, and $s(j_i) = \sum_{\ell=0}^{r-1} j_{i,\ell}$.
\end{theorem}

\begin{proof}
Applying the observation that $\Vuab$ is the pullback of $\Wuab$ via the homomorphism \eqref{eq:factorquotient}, the result follows from Theorem \ref{theorem:charclassesWuab} and Lemma \ref{lemma:Fandqoncohomology}.
\end{proof}

Suppose $\eta \neq 0$ and $(\ualpha|\beta) \in \bsV_{r+1;f,\eta}(\GLmn)(A)$. Then applying the algebra isomorphism $A\M_{r+1;f,\eta} \cong A\M_{r;f^p}$ of Remark \ref{remark:immediateMrf}\eqref{item:Mrfetaiso}, the $A\M_{r+1;f,\eta}$-supermodule structure on $\Amn$ defined by $(\ualpha|\beta)$ identifies with the $A\M_{r;f^p}$-supermodule structure on $\Amn$ defined by
	\[
	(\ualpha_+|\beta) := (\alpha_1,\ldots,\alpha_r,\beta) \in \bsV_{r;f^p}(\GLmn)(A).
	\]
Thus, at the level of group algebras, the homomorphism $\rho_{(\ualpha|\beta)}: \M_{r+1;f,\eta} \otimes_k A \rightarrow \GLmn \otimes_k A$ admits the factorization
\begin{equation} \label{eq:r+1factorization}
A\M_{r+1;f,\eta} \stackrel{\cong}{\longrightarrow} A\M_{r;f^p} \stackrel{\eqref{eq:factorquotient}}{\longrightarrow} A\G_{r;f^p} \stackrel{\wh{\rho}_{(\ualpha_+|\beta)}}{\longrightarrow} A\GLmn.
\end{equation}

\begin{corollary} \label{corollary:characlassesMrfp}
Suppose $0 \neq \eta \in k$. Let $A \in \calg_k$ be a purely even commutative $k$-superalgebra, and let $(\ualpha|\beta) := (\alpha_0,\ldots,\alpha_{r-1},\alpha_r,\beta) \in \bsV_{r+1;f,\eta}(\GLmn)(A)$. Then identifying $\Hbul(\M_{r+1;f,\eta},k)$ with $\Hbul(\M_{r;f^p},k) \cong \Hbul(\M_{r;s+1},k)$ as in Proposition \ref{prop:Mrcohomology}(\ref{item:Mrfetacohomology}), one has for $0 \leq j < p^r$,
\begin{subequations}
\begin{multline}
\bse_r(j)(\M_{r+1;f,\eta} \otimes_k A,\Vuab) = \\
\sum_{\substack{j_0+\cdots+j_{r-1}=j \\ j_0,\ldots,j_{r-1} \geq 0 \\ j_i \equiv 0 \mod p^i}} \sum_{\substack{c,d \geq 0 \\ c+d = j_{r-1,r-1}}}
\Bigg( \binom{j_{r-1,r-1}}{c} (\dalpha_1^{(r)})^{s(j_0)} \cdots (\dalpha_{r-1}^{(r)})^{s(j_{r-2})} (\dalpha_r^{(r)})^c (P_{f^p}(\dalpha_r^{(r)}))^d \\
\otimes \frac{x_r^{j_{0,0}} x_{r-1}^{p j_{0,1}} \cdots x_1^{p^{r-1} j_{0,r-1}}}{(j_{0,0}!)(j_{0,1}!) \cdots (j_{0,r-1}!)}
\cdots
\frac{x_r^{p^{r-2} j_{r-2,r-2}} x_{r-1}^{p^{r-1} j_{r-2,r-1}}}{(j_{r-2,r-2}!) (j_{r-2,r-1}!)}
\cdot \frac{x_r^{p^{r-1} c} w_{s+1}^{p^{r-1}d}}{j_{r-1,r-1}!} \Bigg),
\end{multline}
\begin{multline}
\bse_r^\Pi(j)(\M_{r+1;f,\eta} \otimes_k A,\Vuab) = \\
\sum_{\substack{j_0+\cdots+j_{r-1}=j \\ j_0,\ldots,j_{r-1} \geq 0 \\ j_i \equiv 0 \mod p^i}} \sum_{\substack{c,d \geq 0 \\ c+d = j_{r-1,r-1}}}
\Bigg( \binom{j_{r-1,r-1}}{c} (\ddalpha_1^{(r)})^{s(j_0)} \cdots (\ddalpha_{r-1}^{(r)})^{s(j_{r-2})} (\ddalpha_r^{(r)})^c (P_{f^p}(\ddalpha_r^{(r)}))^d \\
\otimes \frac{x_r^{j_{0,0}} x_{r-1}^{p j_{0,1}} \cdots x_1^{p^{r-1} j_{0,r-1}}}{(j_{0,0}!)(j_{0,1}!) \cdots (j_{0,r-1}!)}
\cdots
\frac{x_r^{p^{r-2} j_{r-2,r-2}} x_{r-1}^{p^{r-1} j_{r-2,r-1}}}{(j_{r-2,r-2}!) (j_{r-2,r-1}!)}
\cdot \frac{x_r^{p^{r-1} c} w_{s+1}^{p^{r-1}d}}{j_{r-1,r-1}!} \Bigg),
\end{multline}
\begin{align}
\bsc_r(\M_{r+1;f,\eta} \otimes_k A, \Vuab) &= \dbeta^{(r)} \otimes y^{p^r}, \text{ and} \\
\bsc_r^\Pi(\M_{r+1;f,\eta} \otimes_k A, \Vuab) &= \ddbeta^{(r)} \otimes y^{p^r}.
\end{align}
\end{subequations}
where $j_i = \sum_{\ell=0}^{r-1} j_{i,\ell}p^\ell$ is the $p$-adic decomposition of $j_i$, and $s(j_i) = \sum_{\ell=0}^{r-1} j_{i,\ell}$.
\end{corollary}

\begin{proof}
Applying the factorization \eqref{eq:r+1factorization}, the calculation of the characteristic classes for $\M_{r+1;f,\eta}$ follows immediately from the calculation for $\M_{r;f^p}$.
\end{proof}

\begin{corollary} \label{corollary:coefficient}
In the conditions and notations of Theorem \ref{theorem:charclassesVuab}, the coefficients with which $x_r^j$ appears in $\bse_r(j)(\Mrf \otimes_k A, \Vuab)$ and $\bse_r^\Pi(j)(\Mrf \otimes_k A, \Vuab)$, respectively, are
\[
\frac{(\dalpha_0^{(r)})^{j_0}(\dalpha_1^{(r)})^{j_1} \cdots (\dalpha_{r-1}^{(r)})^{j_{r-1}}}{(j_0!)(j_1!) \cdots (j_{r-1}!)}
\quad \text{and} \quad
\frac{(\ddalpha_0^{(r)})^{j_0} (\ddalpha_1^{(r)})^{j_1} \cdots (\ddalpha_{r-1}^{(r)})^{j_{r-1}}}{(j_0!)(j_1!) \cdots (j_{r-1}!)},
\]
where $j = \sum_{i=0}^{r-1} j_i p^i$ is the $p$-adic decomposition of $j$. (In the case $s=1$, we mean that we ignore the fact that $w_1 = x_1 - y^2$, and instead consider $x_1,\ldots,x_r,w_1$ as commuting algebraically independent indeterminants.) In particular, taking $1 \leq \ell \leq r$ and $j = p^{\ell-1}$, the coefficients with which $x_r^{p^{\ell-1}}$ appears in $\bse_\ell^{(r-\ell)}(\Mrf \otimes_k A,\Vuab)$ and $(\bse_\ell^{(r-\ell)})^\Pi(\Mrf \otimes_k A,\Vuab)$ are $\dalpha_{\ell-1}^{(r)}$ and $\ddalpha_{\ell-1}^{(r)}$, respectively.
\end{corollary}

\begin{proof}
In the notation of Theorem \ref{theorem:charclassesVuab}, one gets summands involving only $x_r^j$ in \eqref{eq:erjVuab} and \eqref{eq:erjPiVuab} if and only if $j_i = j_{i,i}p^i$ for each $i$ (with $0 \leq j_{i,i} < p$). This implies the result.
\end{proof}

As a consequence of Theorem \ref{theorem:charclassesVuab}, we can also finally pin down the structure constants of the extension algebra $\Ext_{\bsp}^\bullet(\bsir,\bsir)$.

\begin{theorem} \label{theorem:Extalgebrarelations}
Let $r \geq 1$. Then under the assumption of Convention \ref{convention}, $\Ext_{\bsp}^\bullet(\bsir,\bsir)$ is generated by the distinguished (purely even) extension classes \eqref{eq:Extgenerators} and the identity elements
\[
\bse_0 \in \Hom_{\bsp}(\bsi_0^{(r)},\bsi_0^{(r)}) \quad \text{and} \quad \bse_0^\Pi \in \Hom_{\bsp}(\bsi_1^{(r)},\bsi_1^{(r)})
\]
subject only to the relations imposed by the matrix ring decomposition \eqref{eq:matrixring} and:
	\begin{enumerate}
	\item $(\bse_r)^p = \bsc_r \circ \bsc_r^\Pi$ and $(\bse_r^\Pi)^p = \bsc_r^\Pi \circ \bsc_r$.

	\item For each $1 \leq i < r$, $(\bse_i^{(r-i)})^p = [(\bse_i^{(r-i)})^\Pi]^p = 0$.

	\item For each $1 \leq i \leq r$, $\bse_i^{(r-i)} \circ \bsc_r = \bsc_r \circ (\bse_i^{(r-i)})^\Pi$ and $(\bse_i^{(r-i)})^\Pi \circ \bsc_r^\Pi = \bsc_r^\Pi \circ \bse_i^{(r-i)}$.

	\item The subalgebra generated by $\bse_1^{(r-1)},\ldots,\bse_r,(\bse_1^{(r-1)})^\Pi,\ldots,\bse_r^\Pi$ is commutative.
	\end{enumerate}
\end{theorem}

\begin{proof}
Up to a rescaling of the algebra generators, and modulo certain structure constants that were determined to be either $+1$ or $-1$, these relations were established in \cite[Theorem 5.1.1]{Drupieski:2016a}. Now to see that the above relations hold under the assumptions of Convention \ref{convention} (i.e., to show that the undetermined constants are all equal to $+1$), it suffices to observe by Theorem \ref{theorem:charclassesVuab} that the stated relations hold upon restriction to the $(\Mrf \otimes_k A)$-supermodule $\Vuab$ for each $A \in \calg_k$ and each $(\ualpha|\beta) \in \bsVrf(\GLmn)(A)$. For example, since $(\dalpha_i)^p = 0$ for $0 \leq i \leq r-2$, the only summands in \eqref{eq:erjVuab} that will contribute to $[\bse_r(p^{r-1})(\Mrf,\Vuab)]^p$ are those with $j_0 = \cdots = j_{r-2} = 0$ and $j_{r-1} = p^{r-1}$ (hence $j_{r-1,r-1} = 1$). Then
\begin{align*}
[\bse_r(\Mrf,\Vuab)]^p &= [\bse_r(p^{r-1})(\Mrf,\Vuab)]^p \\
&= [\dalpha_{r-1}^{(r)} \otimes x_r^{p^{r-1}} + P_f(\dalpha_{r-1}^{(r)}) \otimes w_s^{p^{r-1}}]^p \\
&= (\dalpha_{r-1}^p)^{(r)} \otimes x_r^{p^r} + P_f(\dalpha_{r-1}^{(r)})^p \otimes w_s^{p^r}.
\end{align*}
Note that $P_f(X)^p = \sum_{\ell=1}^{t-1} (a_{\ell+1} a_s^{-1})^{p^r} X^{p^{\ell+1}} = \sum_{\ell=2}^t (a_\ell a_s^{-1})^{p^r} X^{p^\ell}$, so
\[
P_f(\dalpha_{r-1}^{(r)})^p  \otimes w_s^{p^r} = \begin{cases}
\left( a_s^{-1} \cdot f(\dalpha_{r-1}) \right)^{(r)} \otimes w_s^{p^r} = 0 \otimes w_s^{p^r} = 0 & \text{if $s \geq 2$,} \\
- (\dalpha_{r-1}^p)^{(r)} \otimes w_1^{p^r} = -(\dalpha_{r-1}^p)^{(r)} \otimes (x_1^p - y^{2p}) & \text{if $s = 1$.}
\end{cases}
\]
On the other hand, since $y^{p^r}$ and $\ddbeta^{(r)}$ are both of odd superdegree,
\begin{align*}
\bsc_r(\Mrf,\Vuab) \cdot \bsc_r^\Pi(\Mrf,\Vuab) &= (\dbeta^{(r)} \otimes y^{p^r}) \cdot (\ddbeta^{(r)} \otimes y^{p^r}) \\
&= - (\dbeta^{(r)} \cdot \ddbeta^{(r)}) \otimes (y^{p^r} \cdot y^{p^r}) \\
&= - (\dbeta \ddbeta)^{(r)} \otimes y^{2p^r}.
\end{align*}
By assumption $\alpha_{r-1}^p + \beta^2 = 0$, so $\dalpha_{r-1}^p = -\dbeta\ddbeta$, and hence $-(\dbeta \ddbeta)^{(r)} \otimes y^{2p^r} = (\dalpha_{r-1}^p)^{(r)} \otimes y^{2p^r}$. Since for $s \geq 2$ one has $y^{2p^r} = (y^2)^{p^r} = x_r^{p^r}$, it follows that
\[
[\bse_r(\Mrf,\Vuab)]^p = \bsc_r(\Mrf,\Vuab) \cdot \bsc_r^\Pi(\Mrf,\Vuab).
\]
The other relations are verified in a similar fashion.
\end{proof}

\section{Geometric applications}\label{section:geometricapps}

In this section we apply the results of Section \ref{section:charclasses} to obtain information about the maximal ideal spectrum of the cohomology ring $\Hbul(\GLmnr,k)$ of the $r$-th Frobenius kernel of $\GLmn$.

\subsection{The homomorphism \texorpdfstring{$\phi$}{phi}} \label{subsection:phi}

As recalled at the beginning of Section \ref{subsection:cohomology}, the cohomology ring $\Hbul(G,k)$ of a finite $k$-supergroup scheme $G$ is a graded-commutative superalgebra. In particular, if $a,b \in \Hbul(G,k)$ are homogeneous with respect to both the cohomological $\Z$-grading and the internal $\Z_2$-grading, then
	\[
	a \cdot b = (-1)^{\deg(a) \cdot \deg(b) + \ol{a} \cdot \ol{b}} b \cdot a.
	\]
This implies that the subspace
\begin{equation} \label{eq:H(G,k)}
H(G,k) := \opH^{\ev}(G,k)_{\zero} \oplus \opH^{\odd}(G,k)_{\one}
\end{equation}
of $\Hbul(G,k)$ is a commutative $k$-algebra in the ordinary sense, while the subspace
\[
\opH^{\odd}(G,k)_{\zero} \oplus \opH^{\ev}(G,k)_{\one}
\]
consists of nilpotent elements; cf.\ \cite[Corollary 2.2.5]{Drupieski:2016a}. Thus, there exists a canonical identification between the maximal (resp.\ prime) ideal spectra of $\Hbul(G,k)$ and of $H(G,k)$. As in \cite[Definition 2.3.8]{Drupieski:2016a}, we define the cohomology variety $\abs{G}$ of $G$ to be the maximal ideal spectrum of $H(G,k)$:
\[
\abs{G} = \Max \left( H(G,k) \right) = \Max \left( \Hbul(G,k) \right).
\]
By abuse of notation, we may also use the notation $\abs{G}$ to denote the $k$-scheme defined by the commutative $k$-algebra $H(G,k)$. (It should always be clear from the context whether we mean $\abs{G}$ to mean the affine variety or the affine scheme defined by $H(G,k)$.)

Fix integers $m,n,r \geq 1$. As in \cite[\S5.1]{Drupieski:2016a} and \cite[\S5.1]{Drupieski:2016}, set
\begin{align*}
\g_m &= \Hom_k(k^{m|0},k^{m|0}), & \g_{+1} &= \Hom_k(k^{0|n},k^{m|0}),\\
\g_n &= \Hom_k(k^{0|n},k^{0|n}), & \g_{-1} &= \Hom_k(k^{m|0},k^{0|n}),
\end{align*}
so that $\glzero = \g_m \oplus \g_n$, and $\glone = \g_{+1} \oplus \g_{-1}$. Then as discussed in \cite[\S5.1]{Drupieski:2016a} (cf.\ also \cite[\S5.1]{Drupieski:2016}), the characteristic classes
\begin{subequations}
\begin{equation} \label{eq:eicharclasshom}
\bse_i^{(r-i)}(\GLmnr,\kmn) \in \g_m^{(r)} \otimes \opH^{2p^{i-1}}(\GLmnr,k) \cong \Hom_k(\g_m^{\#(r)},\opH^{2p^{i-1}}(\GLmnr,k)),
\end{equation}
\begin{multline}
(\bse_i^{(r-i)})^\Pi (\GLmnr,\kmn) \in \g_n^{(r)} \otimes \opH^{2p^{i-1}}(\GLmnr,k) \\
\cong \Hom_k(\g_n^{\#(r)},\opH^{2p^{i-1}}(\GLmnr,k)),
\end{multline}
for $1 \leq i \leq r$, and
\begin{align}
\bsc_r(\GLmnr,\kmn) &\in \g_{+1}^{(r)} \otimes \opH^{p^r}(\GLmnr,k) \cong \Hom_k(\g_{+1}^{\#(r)},\opH^{p^r}(\GLmnr,k)), \text{ and} \\
\bsc_r^\Pi(\GLmnr,\kmn) &\in \g_{-1}^{(r)} \otimes \opH^{p^r}(\GLmnr,k) \cong \Hom_k(\g_{-1}^{\#(r)},\opH^{p^r}(\GLmnr,k)), \label{eq:crPicharclasshom}
\end{align}
\end{subequations}
extend multiplicatively to a homomorphism of graded superalgebras
\begin{equation} \label{eq:phi}
\phi: \left( \bigotimes_{i=1}^r S(\glzero^{\#(r)}[2p^{i-1}]) \right) \otimes S(\glone^{\#(r)}[p^r]) \rightarrow \Hbul(\GLmnr,k),
\end{equation}
One of the main results of \cite{Drupieski:2016} was that $\Hbul(\GLmnr,k)$ is finite over the image of $\phi$. Since $\phi$ preserves both the $\Z$-degree and the superdegree of elements, the image of $\phi$ is contained in the subalgebra $H(\GLmnr,k)$ of $\Hbul(\GLmnr,k)$.

Recall from Definition \ref{definition:evensubschemes} that $\Vr(\GLmn)$ is the underlying even subscheme of $\bsV_r(\GLmn)$.

\begin{proposition} \label{proposition:phibar}
The homomorphism $\phi$ of (\ref{eq:phi}) induces a homomorphism
\[
\ol{\phi}: k[\Vr(\GLmn)] \rightarrow H(\GLmnr,k).
\]
\end{proposition}

\begin{proof}
The domain of $\phi$ identifies with the coordinate algebra of
\[ \textstyle
\left( \bigoplus_{i=1}^r \glzero^{(r)} \right) \oplus \glone^{(r)} = \left( \bigoplus_{i=1}^r \glzero \right) \oplus \glone.
\]
For $1 \leq \ell \leq r$, let $X_{ij}(\ell)$ be the coordinate function that returns the $ij$-entry of the $\ell$-th copy of $\glzero$, and let $Y_{ij}$ be the coordinate function that returns the $ij$-entry of $\glone$. Then the set $\set{X_{ij}(\ell),Y_{ij}: 1 \leq \ell \leq r}$ identifies with a set of algebra generators for the domain of $\phi$; cf.\ (\ref{eq:Matmn0relations}--\ref{eq:Matmn1relations}) and \eqref{eq:Vrfetaambientscheme}. Next by Remark \ref{remark:Extmatrices}, we may view the characteristic classes (\ref{eq:eicharclasshom}--\ref{eq:crPicharclasshom}) as elements of the matrix ring $\Matmn(\Hbul(\GLmnr,k))$. Then by Remark \ref{remark:matrixtohom}, the nonzero entries
\[
\text{of } \left\{ \begin{aligned}
\bse_\ell^{(r-\ell)}(\GLmnr,\kmn) \\
(\bse_\ell^{(r-\ell)})^\Pi(\GLmnr,\kmn) \\
\bsc_r(\GLmnr,\kmn) \\
\bsc_r^\Pi(\GLmnr,\kmn)
\end{aligned} \right\}
\text{ are }
\left\{ \begin{aligned}
 &\phi(X_{ij}(\ell)), & 1 \leq i,j \leq m, \\
 &\phi(X_{ij}(\ell)), & m+1 \leq i,j \leq m+n, \\
-&\phi(Y_{ij}), & 1 \leq i \leq m; m+1 \leq j \leq m+n, \\
 &\phi(Y_{ij}), & m+1 \leq i \leq m+n; 1 \leq j \leq m.
\end{aligned} \right.
\]
Now it follows from the relations in Theorem \ref{theorem:Extalgebrarelations} that the defining relations of the scheme $\Vr(\GLmn)$ are contained in the kernel of $\phi$, and hence that $\phi$ factors through $k[\Vr(\GLmn)]$. For example, the relation $(\bse_r)^p = \bsc_r \circ \bsc_r^\Pi$ implies for $1 \leq i,j \leq m$ that
\[
\sum_{1 \leq t_1,\ldots,t_{p-1} \leq m} \phi(X_{i,t_1}(r)) \cdot \phi(X_{t_1,t_2}(r)) \cdots \phi(X_{t_{p-1},j}(r)) = \sum_{t=1}^n [-\phi(Y_{i,m+t})] \cdot \phi(Y_{m+t,j}),
\]
and hence that the polynomial relation
\[
\sum_{1 \leq t_1,\ldots,t_{p-1} \leq m} X_{i,t_1}(r) \cdot X_{t_1,t_2}(r) \cdots X_{t_{p-1},j}(r) = -\left( \sum_{t=1}^n Y_{i,m+t} \cdot Y_{m+t,j} \right)
\]
is an element of $\ker(\phi)$. The other polynomial relations defining $\Vr(\GLmn)$ are similarly verified to be elements of $\ker(\phi)$.
\end{proof}

\subsection{The homomorphism \texorpdfstring{$\psi_{r;f,\eta}$}{psi-r;f,eta}}

Let $r \geq 1$, let $0 \neq f \in k[T]$ be an inseparable $p$-polynomial, and let $\eta \in k$. Recall from Definition \ref{definition:evensubschemes} that, given an algebraic $k$-supergroup scheme $G$, $\Vrfeta(G)$ denotes the underlying purely even subscheme of $\bsVrfeta(G)$. Set $A_G = k[\Vrfeta(G)]$, the (purely even) coordinate algebra of $\Vrfeta(G)$. Then base change to $A_G$ defines a homomorphism of graded superalgebras $\iota: \Hbul(G,k) \rightarrow \Hbul(G,k) \otimes_k A_G = \Hbul(G \otimes_k A_G,A_G)$, $z \mapsto z \otimes 1$. Next, the universal purely even supergroup homomorphism $u_G: \Mrfeta \otimes_k A_G \rightarrow G \otimes_k A_G$ of Definition \ref{def:universalhom} induces a homomorphism of graded $A_G$-superalgebras $u_G^*: \Hbul(G \otimes_k A_G,A_G) \rightarrow \Hbul(\Mrfeta \otimes_k A_G,A_G)$. Finally, recall the identification of $\Hbul(\Mrfeta,k)$ from Proposition \ref{prop:Mrcohomology}. The map $\Hbul(\Mrfeta,k) \rightarrow k$ that sends the generators $y$ and $x_r$ (resp.\ $x_{r-1}$ if $r \geq 2$ and $\eta \neq 0$) of $\Hbul(\Mrfeta,k)$ each to $1$ but that sends the other generators to $0$ is an algebra homomorphism (though not a super\-algebra homomorphism). Extending scalars, one gets a homomorphism
\[
\ve: \Hbul(\Mrfeta \otimes_k A_G,A_G) = \Hbul(\Mrfeta,k) \otimes_k A_G \rightarrow k \otimes_k A_G = A_G.
\]
Now define $\psi_{r;f,\eta}: H(G,k) \rightarrow A_G$ to be the composite algebra homomorphism
\begin{equation} \label{eq:psirfeta}
\psi_{r;f,\eta}: H(G,k) \stackrel{\iota}{\longrightarrow} \Hbul(G \otimes_k A_G,A_G) \stackrel{u_G^*}{\longrightarrow} \Hbul(\Mrfeta \otimes_k A_G,A_G) \stackrel{\ve}{\longrightarrow} A_G.
\end{equation}
At the level of Hochschild complexes, the first arrow in \eqref{eq:psirfeta} is induced by the base change map $k[G] \rightarrow k[G] \otimes_k A_G = A_G[G]$, $z \mapsto z \otimes 1$, and the second arrow is induced by the comorphism $u_G^*: A_G[G] \rightarrow A_G[\Mrfeta]$, i.e., the map of coordinate algebras corresponding to $u_G$.

\begin{lemma}
Let $r \geq 1$, let $0 \neq f \in k[T]$ be an inseparable $p$-polynomial, and let $\eta \in k$. Let $G$ be an algebraic $k$-supergroup scheme. Then the homomorphism of commutative $k$-algebras
\[
\psi_{r;f,\eta}: H(G,k) \rightarrow k[\Vrfeta(G)],
\]
is natural with respect to $G$.
\end{lemma}

\begin{proof}
Set $\psi = \psi_{r;f,\eta}$. Let $\phi: G' \rightarrow G$ be a homomorphism of algebraic $k$-supergroup schemes, let $\phi_{A_G} : G' \otimes_k A_G \rightarrow G \otimes_k A_G$ be the homomorphism of $A_G$-supergroup schemes obtained from $\phi$ via base change, and let $\phi^*: k[\Vrfeta(G)] = A_G \rightarrow A_{G'} := k[\Vrfeta(G')]$ be the algebra homomorphism induced by $\phi$. Then $\phi_{A_G}$ and $\phi^*$ induce the composite map
\[
\Hbul(G \otimes_k A_G,A_G) \stackrel{(\phi_{A_G})^*}{\longrightarrow} \Hbul(G' \otimes_k A_G,A_G) \stackrel{\phi^*}{\longrightarrow} \Hbul(G' \otimes_k A_{G'},A_{G'}),
\]
and using this composite map one can check commutativity of the diagram
\[
\xymatrix{
H(G,k) \ar@{->}[r]^(.35){\iota} \ar@{->}[d]^{\phi^*} & \Hbul(G \otimes_k A_G, A_G) \ar@{->}[r]^(.45){u_G^*} \ar@{->}[d] & \Hbul(\Mrfeta \otimes_k A_G,A_G) \ar@{->}[r] \ar@{->}[d]^{\phi^*} & A_G \ar@{->}[d]^{\phi^*} \\
H(G',k) \ar@{->}[r]^(.35){\iota} & \Hbul(G' \otimes_k A_{G'}, A_{G'}) \ar@{->}[r]^(.45){u_{G'}^*} & \Hbul(\Mrfeta \otimes_k A_{G'},A_{G'}) \ar@{->}[r] & A_{G'}.
}
\]
The commutativity of the diagram implies the naturality of $\psi$ with respect to $G$.
\end{proof}

The algebra $H(G,k)$ is $\Z$-graded via the cohomological grading, while for $r,s \geq 1$, the algebra $k[\Vrs(G)]$ is $\Z[\frac{p^r}{2}]$-graded by Corollary \ref{cor:VrsGgraded}.

\begin{proposition}\label{P:grading}
Let $r,s \geq 1$, and set $\psi_{r;s} = \psi_{r;T^{p^s},0}$. Then
\[
\psi_{r;s} : H(G,k) \rightarrow k[\Vrs(G)]
\]
is a homomorphism of graded $k$-algebras that multiplies degrees by $\frac{p^r}{2}$.
\end{proposition}

\begin{proof}
Set $\psi = \psi_{r;s}$. Extending scalars if necessary, we may assume that the field $k$ is algebraically closed. Let $z \in H(G,k)$ be of cohomological degree $n$ (hence of $\Z_2$-degree $\ol{n}$). Then
\begin{equation} \label{eq:ugstar}
(u_G^* \circ \iota)(z) = \sum_{\bs{i},\bs{j},u,v} x^{\bs{i}} \lambda^{\bs{j}} w^u y^v \otimes f_{\bs{i},\bs{j},u,v}(z) \in \opH^n(\Mrs,k) \otimes_k A_G
\end{equation}
for some $f_{\bs{i},\bs{j},u,v}(z) \in A_G$. Here $x^{\bs{i}} = x_1^{i_1} \cdots x_r^{i_r}$, $\lambda^{\bs{j}} = \lambda_1^{j_1} \cdots \lambda_r^{j_r}$, and the indices run over all nonnegative values such that $2u+v + \sum_{\ell=0}^r (2i_\ell+j_\ell) = n$. By imposing the additional restrictions $0 \leq j_1,\ldots,j_r \leq 1$, and also $0 \leq v \leq 1$ if $s \geq 2$, the set $x^{\bs{i}} \lambda^{\bs{j}} w^u y^v$ becomes a basis for $\Hbul(\Mrs,k)$, and $f_{\bs{i},\bs{j},u,v}(z)$ becomes a well-defined function of $z$. Then $\psi(z)$ is the sum of the coefficients of the terms of the form $x_r^{i_r} y^v$ with $2i_r + v = n$.

Now choose $\mu,a \in k$ such that $\mu^2 = a^{p^r}$, and let $\phi = \phi_{(\mu,a)} \in \Vrs(\Mrs)(k) = \bfHom(\Mrs,\Mrs)(k)$ be the unique homomorphism whose comorphism $\phi^*: k[\Mrs] \rightarrow k[\Mrs]$ satisfies $\phi^*(\tau) = \tau \cdot \mu$, $\phi^*(\theta) = \theta \cdot a$, and $\phi^*(\sigma_i) = \sigma_i \cdot a^{p^{r-1+i}}$ for $i \geq 0$. Then it follows from the explicit description of the generators for $\Hbul(\Mrs,k)$ that
\[
(\phi^* \otimes 1) \left( (u_G^* \circ \iota)(z) \right) = \sum_{\bs{i},\bs{j},u,v} \mu^v a^{\eta(\bs{i},\bs{j},u)} x^{\bs{i}} \lambda^{\bs{j}} w^u y^v \otimes f_{\bs{i},\bs{j},u,v}(z),
\]
where $\eta(\bs{i},\bs{j},u) := p^s u + \sum_{\ell=1}^r (p^{\ell-1} j_\ell + p^\ell i_\ell)$. Next let $\phi \otimes_k A_G: \Mrs \otimes_k A_G \rightarrow \Mrs \otimes_k A_G$ be the homomorphism obtained from $\phi$ via base change to $A_G$. Then by the universal property of $u_G$, the composite homomorphism $u_G \circ (\phi \otimes_k A_G) \in \bfHom(\Mrs,G)(A_G)$ can be written in the form $u_G \otimes_\varphi A_G$ for some $\varphi \in \Hom_{\salg}(A_G,A_G)$. Specifically, $\varphi$ is the unique homomorphism such that
\[
\varphi(z) = \begin{cases}
a^i \cdot z & \text{if the $\Z[\frac{p^r}{2}]$-degree of $z$ is $i$ and $i \in \N$,} \\
a^i \mu \cdot z & \text{if the $\Z[\frac{p^r}{2}]$-degree of $z$ is $i+\frac{p^r}{2}$ and $i \in \N$.}
\end{cases}
\]
This description of $\varphi$ can be checked first when $G = \GLmn$ (cf.\ the description of the $\Z[\frac{p^r}{2}]$-grading on $A_{\GLmn}$ in the proof of Lemma \ref{lemma:Hom(Mrs,G)graded}), and can then be deduced for $G$ arbitrary by choosing an embedding $G \hookrightarrow \GLmn$. Now the identity
\[
(\phi^* \otimes 1) \circ u_G^* \circ \iota = (u_G \circ (\phi \otimes_k A_G))^* = (1 \otimes \varphi) \circ (u_G^* \circ \iota)
\]
implies that
\[
\sum_{\bs{i},\bs{j},u,v} \mu^v a^{\eta(\bs{i},\bs{j},u)} x^{\bs{i}} \lambda^{\bs{j}} w^u y^v \otimes f_{\bs{i},\bs{j},u,v}(z) = \sum_{\bs{i},\bs{j},u,v} x^{\bs{i}} \lambda^{\bs{j}} w^u y^v \otimes \varphi \left( f_{\bs{i},\bs{j},u,v}(z) \right).
\]
Since this identity holds for all the infinitely many different pairs $\mu,a \in k$ such that $\mu^2 = a^{p^r}$ (infinitely many because $k = \ol{k}$), we conclude that $f_{\bs{i},\bs{j},u,v}(z)$ is homogeneous of $\Z[\frac{p^r}{2}]$-degree
\[ \textstyle
v \cdot \frac{p^r}{2} + \eta(\bs{i},\bs{j},u) = v \cdot \frac{p^r}{2} + p^s u + \sum_{\ell = 1}^r (p^{\ell-1}j_\ell + p^\ell i_\ell).
\]
In particular, if $2 i_r + v = n$, then the coefficient of $x_r^{i_r} y^v$ in \eqref{eq:ugstar} is of degree $n \cdot \frac{p^r}{2}$.
\end{proof}

\begin{theorem} \label{theorem:Theta}
Let $r \geq 1$, let $0 \neq f \in k[T]$ be an inseparable $p$-polynomial, and let $\eta \in k$. Then the composite homomorphism
\[
k[\Vr(\GLmn)] \stackrel{\ol{\phi}}{\longrightarrow} H(\GLmnr,k) \stackrel{\psi_{r;f,\eta}}{\longrightarrow} k[\Vrfeta(\GLmn)]
\]
is equal to the composition of the $r$-th Frobenius morphism on $k[\Vr(\GLmn)]$, i.e., the algebra map that sends the defining algebra generators to their $p^r$-th powers, and the canonical quotient map $k[\Vr(\GLmn)] \twoheadrightarrow k[\Vrfeta(\GLmn)]$. In particular, if $k$ is perfect then $\psi_{r;f,\eta} \circ \ol{\phi}$ is surjective onto $p^r$-th powers. Letting
\[
\Psi_{r;f,\eta} : \Vrfeta(\GLmn) \rightarrow \abs{\GLmnr} \quad \text{and} \quad \Phi : \abs{\GLmnr} \rightarrow \Vr(\GLmn)
\]
denote the morphisms of schemes induced by $\psi_{r;f,\eta}$ and $\ol{\phi}$, respectively, the composite morphism
\[
\Theta_{r;f,\eta} : \Vrfeta(\GLmn) \stackrel{\Psi_{r;f,\eta}}{\longrightarrow} \abs{\GLmnr} \stackrel{\Phi}{\longrightarrow} \Vr(\GLmn)
\]
is equal to the composite of the inclusion $\Vrfeta(\GLmn) \hookrightarrow \Vr(\GLmn)$ and the $r$-th Frobenius twist morphism on the scheme $\Vr(\GLmn)$ (defined over the prime field $\Fp$).
\end{theorem}

\begin{proof}
Set $G = \GLmnr$, set $A = k[\Vrfeta(G)]$, and let $\set{X_{ij}(\ell),Y_{ij}: 1 \leq \ell \leq r}$ be the coordinate functions defined in the proof of Proposition \ref{proposition:phibar}. Then $\set{X_{ij}(\ell),Y_{ij}: 1 \leq \ell \leq r}$ identifies with a set of algebra generators for $A$. Let $u_G: \Mrfeta \otimes_k A \rightarrow G \otimes_k A$ be the universal purely even supergroup homomorphism from $\Mrfeta$ to $G$. Then $u_G = \rho_{(\ualpha|\beta)}$, where $(\ualpha|\beta) \in \bsVrfeta(G)(A) = \Vrfeta(G)(A)$ is the universal tuple defined in the second paragraph of the proof of Theorem \ref{theorem:Homsuperscheme}. More precisely, for $1 \leq \ell \leq r$ the $ij$-entry of $\alpha_{\ell-1}$ is equal to the image in $A$ of $X_{ij}(\ell)$, and the $ij$-entry of $\beta$ is equal to the image in $A$ of $Y_{ij}$.

Viewing the characteristic classes (\ref{eq:eicharclasshom}--\ref{eq:crPicharclasshom}) as linear maps (cf.\ Remark \ref{remark:matrixtohom}), one has
\begin{gather*}
\phi(X_{ij}(\ell)) = \bse_\ell^{(r-\ell)}(G,\kmn)(X_{ij}(\ell)) + (\bse_\ell^{(r-\ell)})^\Pi(G,\kmn)(X_{ij}(\ell)), \\
\text{and} \quad \phi(Y_{ij}) = \bsc_r(G,\kmn)(Y_{ij}) + \bsc_r^\Pi(G,\kmn)(Y_{ij})
\end{gather*}
by the definition of $\phi$. Then
\begin{gather*}
(u_G^* \circ \phi)(X_{ij}(\ell)) = \bse_\ell^{(r-\ell)}(\Mrfeta \otimes_k A,\Vuab)(X_{ij}(\ell)) + (\bse_\ell^{(r-\ell)})^\Pi(\Mrfeta \otimes_k A,\Vuab)(X_{ij}(\ell)), \\
\text{and} \quad (u_G^* \circ \phi)(Y_{ij}) = \bsc_r(\Mrfeta \otimes_k A,\Vuab)(Y_{ij}) + \bsc_r^\Pi(\Mrfeta \otimes_k A,\Vuab)(Y_{ij}).
\end{gather*}
Set $\psi = \psi_{r;f,\eta}$. First suppose $\eta = 0$. Then it follows that $(\psi \circ \phi)(X_{ij}(\ell))$ is equal to the sum of the $ij$-entries of the coefficients of $x_r^{p^{\ell-1}}$ in \eqref{eq:erjVuab} and \eqref{eq:erjPiVuab}, and $(\psi \circ \phi)(Y_{ij})$ is equal to the sum of the $ij$-entries of the coefficients of $y^{p^r}$ in \eqref{eq:crVuab} and \eqref{eq:crPiVuab}. More precisely, in the case $r = s=1$ we ignore the fact that $w_1 = x_1 - y^2$ and instead consider $w_1$ as an algebraically independent indeterminate (as in Corollary \ref{corollary:coefficient}) when calculating the coefficient of $x_1^{p^{\ell-1}}$; this does no harm because $x_1-y^2$ is in the kernel of the algebra homomorphism $\ve_A: \Hbul(\M_{1;1} \otimes_k A,A) \rightarrow A$. Now identifying $\Matmn(A)^{(r)}$ with $\Matmn(A)$ as in Remark \ref{remark:Extmatrices}, one gets by Corollary \ref{corollary:coefficient} that $(\psi \circ \phi)(X_{ij}(\ell))$ is equal to the image in $A$ of $X_{ij}(\ell)^{p^r}$, and $(\psi \circ \phi)(Y_{ij})$ is equal to the image in $A$ of $Y_{ij}^{p^r}$. The argument for the case $\eta \neq 0$ is entirely similar, using now the calculations of Proposition \ref{proposition:erMrfeta}, Lemma \ref{lemma:crMonefeta}, and Corollary \ref{corollary:characlassesMrfp} instead of Theorem \ref{theorem:charclassesVuab}.
\end{proof}

\begin{corollary} \label{corollary:VrGunion}
Suppose $k$ is algebraically closed. Then the finite morphism of schemes
\[
\Phi: \abs{\GLmnr} \rightarrow \Vr(\GLmn)
\]
induces a surjective morphism of varieties $\Phi(k): \abs{\GLmnr} \rightarrow \Vr(\GLmn)(k)$. More precisely,
\[
\Vr(\GLmn)(k) = \bigcup_{f,\eta} \im ( \Theta_{r;f,\eta} ) = \bigcup_{f,\eta} \Vrfeta(\GLmn)(k)
\]
where the union is taken over all inseparable $p$-polynomials $0 \neq f \in k[T]$ and all $\eta \in k$ (though it suffices to consider only $\eta = 0$).
\end{corollary}

\begin{proof}
Let $(\wt{\ualpha}|\wt{\beta}) := (\wt{\alpha}_0,\ldots,\wt{\alpha}_{r-1},\wt{\beta}) \in \Vr(\GLmn)(k)$. Since $k$ is algebraically closed, there exists $(\ualpha|\beta) := (\alpha_0,\ldots,\alpha_{r-1},\beta) \in \Vr(\GLmn)(k)$ such that $(\wt{\ualpha}|\wt{\beta})$ is obtained from $(\ualpha|\beta)$ by raising the individual coordinate entries of each each matrix $\alpha_0,\ldots,\alpha_{r-1},\beta$ to the $p^r$-th power. Then $(\wt{\ualpha}|\wt{\beta})$ is the image of $(\ualpha|\beta)$ under the $r$-th Frobenius twist morphism on the scheme $\Vr(\GLmn)$. Next, since $\glzero$ is a finite-dimensional $k$-vector space, the matrix $\alpha_{r-1} \in \glzero$ generates a finite-dimensional restricted Lie subalgebra of $\glzero$. Then there exist integers $1 \leq s \leq t$ such that the matrices $\alpha_{r-1}^{p^s},\alpha_{r-1}^{p^{s+1}},\ldots,\alpha_{r-1}^{p^t}$ are linearly dependent, and hence there exists an inseparable $p$-polynomial $0 \neq f \in k[T]$ such that $f(\alpha_{r-1}) =0$. Then $(\ualpha|\beta) \in \Vrf(\GLmn)(k)$, and hence $(\wt{\ualpha}|\wt{\beta})$ is in the image of the composite morphism $\Theta_{r;f}(k): \Vrf(\GLmn)(k) \rightarrow \Vr(\GLmn)(k)$.
\end{proof}

\begin{corollary}
The kernel of $\ol{\phi}: k[\Vr(\GLmn)] \rightarrow H(\GLmn,k)$ is nilpotent.
\end{corollary}

\begin{proof}
Extending scalars first if necessary, we may assume that $k$ is algebraically closed. Let $g \in \ker(\ol{\phi})$. Then $(\psi_{r;f} \circ \ol{\phi})(g) = 0$ for all inseparable $p$-polynomials $0 \neq f \in k[T]$. Since $\Vr(\GLmn)(k) = \bigcup_f \Vrf(\GLmn)(k)$ by Corollary \ref{corollary:VrGunion}, this implies that $g$ defines the zero function on the variety $\Vr(\GLmn)(k)$, and hence that $g$ is nilpotent as an element of $k[\Vr(\GLmn)]$.
\end{proof}

\appendix

\section{Base change for strict polynomial superfunctors} \label{section:basechange}

In \cite[\S2]{Drupieski:2016} the first author defined the category $\bsp = \bsp_k$ of strict polynomial superfunctors over a field $k$. (The general assumption of perfectness in \cite{Drupieski:2016} only served to simplify the discussion of Frobenius twists in \cite[\S2.7]{Drupieski:2016}.) As in \cite[\S2]{Suslin:1997}, the theory can be generalized to the context of more general coefficient rings. This section of the paper can be viewed simultaneously as a `superization' of \cite[\S2]{Suslin:1997} and as a translation from the original treatment of strict polynomial functors given by Friedlander and Suslin to the modern treatment following the exposition of Pirashvili \cite{Pirashvili:2003}. 

Throughout Appendix \ref{section:basechange}, let $A$ be a commutative superring of characteristic $p \neq 2$. In this section we will often denote the tensor product of $A$-supermodules simply by $V \otimes W$ instead of $V \otimes_A W$, except when confusion is likely. For a detailed treatment of the theory of commutative superrings and their supermodules, we refer the reader to thesis of Westra \cite{Westra:2009}.

\subsection{Strict polynomial superfunctors over commutative superrings}

Define $\bsv_A$ to be the full subcategory of $\smod_A$ having as objects the $A$-supermodules that are finitely-generated and projective in $\fsmod_A = (\smod_A)_\ev$. Thus, each object in $\bsv_A$ is a direct summand of some finitely-generated free $A$-supermodule $\Amn$, which admits a homogeneous basis $e_1,\ldots,e_{m+n}$ such that $\ol{e_i} = \zero$ if $1 \leq i \leq m$ and $\ol{e_i} = \one$ if $m+1 \leq i \leq m+n$. The category $\bsv_A$ is closed under tensor products (over $A$) and under the operation of taking $A$-linear duals: $V \mapsto V^\# := \Hom_A(V,A)$.

Now let $V \in \fsmod_A$ and let $n \in \N$. The symmetric group $\fS_n$ acts on $V^{\otimes n}$ on the right via super place permutations, making $V^{\otimes n}$ a right module for the group superalgebra $A\fS_n := \Z\fS_n \otimes_\Z A$. Set $\bsg^n(V) = \bsg_A^n(V) = (V^{\otimes n})^{\fS_n}$. Then $\bsg^n: V \mapsto \bsg^n(V)$ is an endofunctor on $\fsmod_A$. If $V \in \bsv_A$ is free, then $\bsg^n(V)$ is free and admits a basis of the form described in \cite[2.3.6]{Drupieski:2016}. The functor $\bsg^n$ also satisfies the following bi-functorial exponential formula for each $U,V \in \bsv_A$ (cf.\ \cite[\S2.5]{Drupieski:2016}):
\begin{equation} \label{eq:exponential} \textstyle
\bsg^n(U \oplus V) \cong \bigoplus_{i+j=n} \bsg^i(U) \otimes \bsg^j(V).
\end{equation}
The exponential formula can be verified first when $U$ and $V$ are free by working with bases of the type described in \cite[(2.3.4)]{Drupieski:2016}, and can then be verified for $U$ and $V$ arbitrary via naturality. (In fact, the exponential formula is an isomorphism of strict polynomial bisuperfunctors over $A$, but it is enough for us to consider it as an isomorphism of bifunctors on the category $\fsmod_A$.)

Let $U,V \in \fsmod_A$. Repeated application of the supertwist map defines an isomorphism $(U^{\otimes n}) \otimes (V^{\otimes n}) \cong (U \otimes V)^{\otimes n}$, which is an isomorphism of $\fS_n$-modules if we consider $(U^{\otimes n}) \otimes (V^{\otimes n})$ as a right $\fS_n$-module via the diagonal map $\fS_n \rightarrow \fS_n \times \fS_n$. Then it follows that $\bsg^n(U) \otimes \bsg^n(V)$ is naturally an $A$-subsupermodule of $\bsg^n(U \otimes V)$. In particular, if $\phi: U \otimes V \rightarrow W$ is an even $A$-linear map, then there is a naturally defined induced even $A$-linear map
\begin{equation} \label{eq:bsgphi}
\bsg^n(\phi): \bsg^n(U) \otimes \bsg^n(V) \rightarrow \bsg^n(W).
\end{equation}
Now define $\bsg^n(\bsv_A)$ to be the category whose objects are the same as those of $\bsv_A$, whose sets of morphisms are defined by $\Hom_{\bsg^n(\bsv_A)}(V,W) := \bsg^n \left( \Hom_A(V,W) \right)$, and in which composition of morphisms is induced as in \eqref{eq:bsgphi} by the composition of linear maps in $\bsv_A$.

\begin{remark} \label{remark:homsummand}
If $\phi \in \Hom_A(V,W)_{\zero}$, then $\phi^{\otimes n} \in \bsg^n \Hom_A(V,W)$. In particular, if $X,Y \in \bsv_A$ are free $A$-supermodules having $V$ and $W$ as direct summands, respectively, then the (even) inclusion and projection maps defining $V$ and $W$ as summands of $X$ and $Y$ can be used to show that $\Hom_{\bsg^n(\bsv_A)}(V,W)$ is a direct summand of the free $A$-supermodule $\Hom_{\bsg^n(\bsv_A)}(X,Y) = \bsg^n \Hom_A(X,Y)$. Thus, the category $\bsg^n(\bsv_A)$ is enriched over $\bsv_A$.
\end{remark}

Recall that the external tensor product $\phi \boxtimes \psi: V \otimes V' \rightarrow W \otimes W'$ of maps $\phi \in \Hom_A(V,V')$ and $\psi \in \Hom_A(W,W')$ is defined by $(\phi \boxtimes \psi)(v \otimes v') = (-1)^{\ol{v} \cdot \ol{\psi}} \phi(v) \otimes \psi(v')$. Composition of morphisms in $\bsg^n(\bsv_A)$ can then be interpreted via the following lemma:

\begin{lemma}
The external tensor product operation induces for each $V,W \in \bsv_A$ an isomorphism
\[
\bsg^n \Hom_A(V,W) \cong \Hom_{A\fS_n}(V^{\otimes n},W^{\otimes n})
\]
that is compatible with composition of morphisms in $\bsg^n(\bsv_A)$.
\end{lemma}

\begin{proof}
First let $V,W \in \bsv_A$, and let $X,Y \in \smod_A$. Then the external tensor product operation defines an isomorphism
\begin{equation} \label{eq:boxtimes}
\boxtimes: \Hom_A(V,W) \otimes \Hom_A(X,Y) \simrightarrow \Hom_A(V \otimes X,W \otimes Y).
\end{equation}
The reader can check that \eqref{eq:boxtimes} is an isomorphism by first verifying the case when $V,W \in \bsv_A$ are free, and then using naturality to deduce the case when $V,W \in \bsv_A$ are arbitrary. Next, let $\boxtimes^n : \Hom_A(V,W)^{\otimes n} \rightarrow \Hom_A(V^{\otimes n},W^{\otimes n})$ be the isomorphism obtained inductively from \eqref{eq:boxtimes}. The reader can check that $\boxtimes^n$ is $\fS_n$-equivariant, with $\fS_n$ acting on $\Hom_A(V,W)^{\otimes n}$ on the right by super place permutations, and acting on $\Hom_A(V^{\otimes n},W^{\otimes n})$ on the right by $(\phi.\sigma)(z) = [\phi(z.\sigma^{-1})].\sigma$. Taking $\fS_n$-fixed points, the result follows.
\end{proof}

\begin{definition}[Strict polynomial superfunctor]
A homogeneous degree-$n$ strict polynomial superfunctor over $A$ is an even $A$-linear functor $T: \bsg^n(\bsv_A) \rightarrow \bsv_A$, i.e., it is a covariant functor $T: \bsg^n(\bsv_A) \rightarrow \bsv_A$ such that for each $V,W \in \bsv_A$, the function $T_{V,W}: \bsg^n \Hom_A(V,W) \rightarrow \Hom_A(T(V),T(W))$ is an even $A$-linear map. A homomorphism $\eta: S \rightarrow T$ of homogeneous degree-$n$ strict polynomial superfunctors consists for each $V \in \bsv_A$ of an $A$-linear map $\eta(V) \in \Hom_A(S(V),T(V))$ such that for each $\phi \in \Hom_{\bsg^n(\bsv_A)}(V,W)$,
\[
\eta(W) \circ S(\phi) = (-1)^{\ol{\eta} \cdot \ol{\phi}} T(\phi) \circ \eta(V).
\]
We denote by $\bsp_{n,A}$ the category whose objects are the homogeneous degree-$n$ strict polynomial superfunctors over $A$ and whose morphisms are the homomorphisms between those functors, and we denote by $\bsp_A$ the category $\prod_{n \in \N} \bsp_{n,A}$ of arbitrary strict polynomial superfunctors over $A$.
\end{definition}

\begin{example}[The parity change functor $\Pi$] \label{example:paritychange}
The parity change functor $\Pi \in \bsp_{1,A}$ acts on an object $V \in \bsv_A$ by reversing the $\Z_2$-grading of $V$ but leaving the right $A$-supermodule structure of $V$ unchanged. On morphisms, $\Pi$ sends a map $\phi \in \Hom_A(V,W)$ to the same map between the underlying right $A$-supermodules. Thus, if $v^\pi$ denotes the element $v \in V$ considered as an element of $\Pi(V)$, then the right action of $a \in A$ is given by $v^\pi.a = (v.a)^\pi$, while the left action is given by $a.v^\pi = (-1)^{\ol{a}}(a.v)^\pi$. The map $v \mapsto v^\pi$ defines an odd isomorphism $V \simeq \Pi(V)$.
\end{example}

\begin{remark} \label{remark:paritychange}
If $R$ is a not necessarily commutative superalgebra (e.g., $R = k\Mrs$, which is commutative in the ordinary sense but not in the sense of superalgebras), and if $V$ is an $R$-super\-module, then we define the action of $R$ on $\Pi(V)$ via the same sign conventions as in Example \ref{example:paritychange}. Thus, if $V$ is a right $R$-supermodule, then the actions of $R$ on the underlying sets of $V$ and $\Pi(V)$ are the same, while if $V$ is a left $R$-supermodule, the left action of $R$ on $\Pi(V)$ is obtained by twisting the action on $V$ by the automorphism $r \mapsto (-1)^{\ol{r}}r$ of $R$.

Note that if $V$ is an $R$-supermodule and $v \in V$ is odd, then there exists a surjective even $R$-supermodule homomorphism $\Pi(R) \twoheadrightarrow R.v$, $r^\pi \mapsto (-1)^{\ol{r}} r.v$, where $R.v$ denotes the $R$-submodule of $V$ generated by $v$. From this it immediately follows that any finitely-generated $R$-supermodule is a quotient via an even $R$-supermodule homomorphism of a finite direct sum of copies of the free rank-one $R$-supermodules $R$ and $\Pi(R)$.
\end{remark}

\begin{example}
Given $F \in \bsp_{n,A}$, the dual functor $F^\# \in \bsp_{n,A}$ is defined on objects by $F^\#(V) = F(V^\#)^\#$. There is a natural identification $F \cong F^{\#\#}$, and the assignment $F \mapsto F^\#$ defines a (super) anti-equivalence on $\bsp_{n,A}$. For more details, see \cite[\S2.2]{Drupieski:2016}.
\end{example}

\begin{example} \label{example:classicalfunctors}
The symmetric, exterior, divided, and alter\-nating power functors defined in \cite[\S2.3]{Drupieski:2016} admit immediate generalizations to $\bsp_{n,A}$, which we denote by $\bss^n = \bss_A^n$, $\bsl^n = \bsl_A^n$, $\bsg^n = \bsg_A^n$, and $\bsa^n = \bsa_A^n$, respectively. If $\bs{X}$ is one of these functors, and if $V \in \bsv_A$ is free, then $\bs{X}(V)$ is free as well, with a basis of the form described in \cite[\S2.3]{Drupieski:2016}. Each of these functors also satisfies an exponential formula of the form \eqref{eq:exponential}. The reader can check that the duality isomorphisms $\bss \cong \bsg^\#$ and $\bsl \cong \bsa^\#$ of \cite[\S2.6]{Drupieski:2016} also generalize to the context of strict polynomial superfunctors over $A$; the duality isomorphisms can be checked first when $V \in \bsv_A$ is free by using bases, and then for $V \in \bsv_A$ arbitrary by using functoriality.
\end{example}

\begin{remark} \label{remark:GLVstructure}
Given $V \in \smod_A$, write $GL(V)$ for the group of even $A$-linear automorphisms of $V$. Now let $T \in \bsp_{n,A}$ and $V \in \bsv_A$. If $g \in GL(V) \subset \Hom_A(V,V)_{\zero}$, then $g^{\otimes n} \in \bsg^n \Hom_A(V,V)$, and hence $T_{V,V}(g^{\otimes n})$ is an invertible element in $\Hom_A(T(V),T(V))$. Thus, $T(V)$ is naturally equipped with the structure of a $GL(V)$-supermodule.
\end{remark}

The category $\bsp_A$ is not abelian, though the underlying even subcategory $(\bsp_A)_\ev$, having the same objects but only the even homomorphisms as morphisms, is an exact category in the sense of Quillen, with the admissible short exact sequences $0 \rightarrow T' \rightarrow T \rightarrow T'' \rightarrow 0$ being those that are exact when evaluated on any $V \in \bsv_A$. Now for $V \in \bsv_A$, set $\bsg^{n,V} = \bsg^n \Hom_A(V,-)$. It follows from Remark \ref{remark:homsummand} that $\bsg^{n,V} \in \bsp_{n,A}$. By Yoneda's Lemma, there exists for each $T \in \bsp_{n,A}$ a natural isomorphism $\Hom_{\bsp_A}(\bsg^{n,V},T) \cong T(V)$. This implies that $\bsg^{n,V}$ is projective in $(\bsp_{n,A})_\ev$.

\begin{proposition}
Let $T \in \bsp_{n,A}$, and set $V = A^{n|n}$. Then for all $W \in \bsv_A$, the canonical map
\[
\bsg^n \left( \Hom_A(V,W) \right) \otimes T(V) \rightarrow T(W)
\]
induced by the function $T_{V,W}: \bsg^n \Hom_A(V,W) \rightarrow \Hom_A(T(V),T(W))$ and by evaluation on $T(V)$ is a surjection. In particular, $\bsg^{n,V} \oplus (\Pi \circ \bsg^{n,V})$ is a projective generator in $(\bsp_{n,A})_\ev$.
\end{proposition}

\begin{proof}
It suffices by \cite[Proposition A.1]{Touze:2013} (or rather, by its evident generalization to a super\-com\-mutative coefficient ring) to show for each $X,Y \in \bsv_A$ that composition of morphisms in $\bsg^n(\bsv_A)$ induces a surjective map
\begin{equation} \label{eq:composeonto}
\Hom_{\bsg^n(\bsv_A)}(V,Y) \otimes \Hom_{\bsg^n(\bsv_A)}(X,V) \rightarrow \Hom_{\bsg^n(\bsv_A)}(X,Y).
\end{equation}
This can be checked first when $X,Y \in \bsv_A$ are free by arguing as in the proof of \cite[Lemma 2.3]{Touze:2010b} or \cite[Theorem 4.2]{Axtell:2013} using bases. From that case, the result can then be deduced using naturality when $X,Y \in \bsv_A$ are arbitrary.
\end{proof}

By virtue of the preceding proposition, the exact category $(\bsp_A)_\ev$ has enough projectives. This means that we can do homological algebra in $\bsp_A$, and can define the extension groups $\Ext_{\bsp_A}^\bullet(T,T')$ for each $T,T' \in \bsp_A$; cf.\ \cite[\S3.2]{Drupieski:2016}.

\subsection{Base change}

Let $A'$ be a commutative $A$-superalgebra. Given $V \in \smod_A$, write $V_{A'}$ for the $A'$-supermodule $V \otimes_A A'$ obtained via base change from $V$. Base change is left adjoint to the forgetful functor from $A'$-supermodules to $A$-supermodules, i.e., if $W' \in \smod_{A'}$, then there is a canonical identification $\Hom_A(V,W') = \Hom_{A'}(V_{A'},W')$. In particular, $\Hom_{A'}(A',A') = \Hom_A(A,A')  \cong A'$. Now for $V,W \in \bsv_A$, define
\begin{equation} \label{eq:basechangeiso}
\nu_{A'}: \Hom_A(V,W) \otimes_A A' \rightarrow \Hom_A(V, W \otimes_A A') = \Hom_{A'}(V_{A'},W_{A'})
\end{equation}
by $\nu_{A'}(\phi \otimes a')(v) = (-1)^{\ol{a}' \cdot \ol{v}} \phi(v) \otimes a'$. Identifying $A'$ with $\Hom_A(A,A')$ as a left $A$-supermodule, \eqref{eq:basechangeiso} is an isomorphism by \eqref{eq:boxtimes}. More generally, base change induces an isomorphism
\begin{equation} \label{eq:basechangebsg}
\Hom_{\bsg^n(\bsv_A)}(V,W)_{A'} \cong \Hom_{\bsg^n(\bsv_{A'})}(V_{A'},W_{A'})
\end{equation}
that is compatible with the composition of morphisms. To see this, first observe that if $V \in \bsv_A$, then $\bsg_A^n(V)_{A'} \cong \bsg_{A'}^n(V_{A'})$; this can be verified first when $V$ is free by using bases as in \cite[2.3.6]{Drupieski:2016}, and then for arbitrary $V$ by using the exponential property of $\bsg^n$. Next, another application of \eqref{eq:boxtimes} shows for $W \in \bsv_A$ that $\Hom_A(V,W) \cong W \otimes_A \Hom_A(V,A) = W \otimes_A V^\#$, and hence that $\Hom_A(V,W) \in \bsv_A$ whenever $V,W \in \bsv_A$. Together these observations imply \eqref{eq:basechangebsg}. From now on we will identify the two sides of \eqref{eq:basechangeiso} and \eqref{eq:basechangebsg}, respectively.

Now given $V' \in \bsv_{A'}$, define $\bsv_A(V')$ to be the category whose objects consist of all pairs $(V,\phi)$ such that $V \in \bsv_A$ and $\phi \in \Hom_{\bsg^n(\bsv_{A'})}(V_{A'},V')_{\zero}$, and whose morphisms are defined by
\[
\Hom_{\bsv_A(V')}((V_1,\phi_1),(V_2,\phi_2)) = \set{ \psi \in [\Hom_{\bsg^n(\bsv_A)}(V_1,V_2)]_{A'}: \phi_2 \circ \psi = \phi_1}.
\]
Note that each morphism $\psi: (V_1,\phi_1) \rightarrow (V_2,\phi_2)$ must be even because $\phi_1$ and $\phi_2$ are both even by definition. The category $\bsv_A(V')$ forms a directed system, since for any pair $(V_1,\phi_1)$ and $(V_2,\phi_2)$ of objects in $\bsv_A(V')$, there exist morphisms from the pair into $(V_1 \oplus V_2,\phi_1 \oplus \phi_2)$, where $\phi_1 \oplus \phi_2$ is defined as follows: The direct sum decomposition $(V_1 \oplus V_2)_{A'} = (V_1)_{A'} \oplus (V_2)_{A'}$ and the exponential property for $\bsg_{A'}^n$ give rise to a direct sum decomposition
\[
\Hom_{\bsg^n(\bsv_{A'})}((V_1 \oplus V_2)_{A'},V') \cong \bigoplus_{i+j=n} \Hom_{\bsg^i(\bsv_{A'})}((V_1)_{A'},V') \otimes \Hom_{\bsg^j(\bsv_{A'})}((V_2)_{A'},V').
\]
Then $\phi_1 \oplus \phi_2 \in \Hom_{\bsg^n(\bsv_{A'})}((V_1 \oplus V_2)_{A'},V')$ is the evident function induced by $\phi_1$ and $\phi_2$ that vanishes on all summands except $i=n$ and $i=0$. Now the morphism $(V_1,\phi_1) \rightarrow (V_1 \oplus V_2,\phi_1 \oplus \phi_2)$ is simply $(\iota_1^{\otimes n})_{A'}$, where $\iota_1: V_1 \rightarrow V_1 \oplus V_2$ is the inclusion, and similarly for $(V_2,\phi_2)$.

If $T \in \bsp_{n,A}$, then $(T_{V_1,V_2})_{A'} := T_{V_1,V_2} \otimes_A A'$ defines an even $A'$-linear map
\[
(T_{V_1,V_2})_{A'}: [\Hom_{\bsg^n(\bsv_A)}(V_1,V_2)]_{A'} \rightarrow [\Hom_A(T(V_1),T(V_2))]_{A'} = \Hom_{A'}(T(V_1)_{A'},T(V_2)_{A'}).
\]
Then the assignment $(V,\phi) \mapsto T(V)_{A'}$ gives rise to a functor $\bsv_A(V') \rightarrow \smod_{A'}$, which we denote by $T(-)_{A'}$. In fact, $T(-)_{A'}$ has image in the underlying even subcategory $\fsmod_{A'}$ of $\smod_{A'}$, since all homomorphisms in $\bsv_A(V')$ are even. This ensures that the direct limit
\begin{equation} \label{eq:TAprime}
T_{A'}(V') := \varinjlim_{(V,\phi) \in \bsv_A(V')} T(V)_{A'}
\end{equation}
is an $A'$-supermodule. Concretely, $T_{A'}(V')$ identifies with the quotient
\begin{multline} \label{eq:directlimitquotient}
\big(\bigoplus_{(V,\phi)} T(V)_{A'} \big) \big/ \big \langle z - [(T_{V_1,V_2})_{A'}(\psi)(z)] : z \in T(V_1)_{A'}, \psi \in \Hom_{\bsv_A(V')}((V_1,\phi_1),(V_2,\phi_2)) \big \rangle,
\end{multline}
where the direct sum is taken over all objects $(V,\phi) \in \bsv_A(V')$. In particular, since each morphism in $\bsv_A(V')$ is even, the submodule of relations is generated by the relations with $z$ homogeneous, and hence is a subsupermodule of the direct sum.

\begin{remark} \label{remark:finalobject}
If $V \in \bsv_A$, then $(V,1_{V_{A'}})$ is a final object in $\bsv_A(V_{A'})$, hence $T_{A'}(V_{A'}) = T(V)_{A'}$.
\end{remark}

\begin{remark} \label{remark:projectiveimage}
Let $W' \in \bsv_{A'}$, and let $\theta \in \Hom_{A'}(V',W')_{\zero}$. Then $\theta^{\otimes n} \in \Hom_{\bsg^n(\bsv_{A'})}(V',W')$, and the assignment $(V,\phi) \mapsto (V,\theta^{\otimes n} \circ \phi)$ evidently defines a functor $\theta: \bsv_A(V') \rightarrow \bsv_A(W')$. Passing to the directed system of $A$-supermodules $[T(V)_{A'}]_{(V,\phi) \in \bsv_A(V')}$, the functor $\theta$ induces an even $A'$-supermodule homomorphism $T_{A'}(\theta): T_{A'}(V') \rightarrow T_{A'}(W')$. In terms of the direct sum in \eqref{eq:directlimitquotient}, the map $T_{A'}(\theta)$ is induced by sending the summand $T(V)_{A'}$ corresponding to $(V,\phi) \in \bsv_A(V')$ identically onto the summand $T(V)_{A'}$ corresponding to $(V,\theta^{\otimes n} \circ \phi) \in \bsv_A(W')$. From this description it immediately follows that $T_{A'}$ defines a functor $T_{A'}: (\bsv_{A'})_\ev \rightarrow \fsmod_{A'}$.

Now let $X' \in \bsv_{A'}$ be a free $A'$-supermodule having $V'$ as a direct summand, and let $\iota: V' \hookrightarrow X'$ and $\pi: X' \twoheadrightarrow V'$ be the corresponding even inclusion and projection maps. Then $T_{A'}(\pi) \circ T_{A'}(\iota)= 1_{T_{A'}(V')}$, and hence $T_{A'}(V')$ is direct summand of $T_{A'}(X')$. Since $X'$ is free, $X' = X_{A'}$ for some free $A$-supermodule $X \in \bsv_A$. Then $T_{A'}(X') = T_{A'}(X_{A'}) = T(X)_{A'}$ by Remark \ref{remark:finalobject}. In particular, since base change takes objects in $\bsv_A$ to objects in $\bsv_{A'}$, we can conclude that $T_{A'}(V') \in \bsv_{A'}$, and hence that $T_{A'}$ defines a functor $T_{A'}: (\bsv_{A'})_\ev \rightarrow (\bsv_{A'})_\ev$.
\end{remark}

\begin{proposition}
Let $A'$ be a commutative $A$-superalgebra, let $T \in \bsp_{n,A}$, and let $T_{A'}$ be as defined in \eqref{eq:TAprime}. Then $T_{A'}$ admits the structure of a homogeneous degree-$n$ strict polynomial superfunctor over $A'$. With this structure, the base change functor $\bsp_{n,A} \rightarrow \bsp_{n,A'}$, $T \mapsto T_{A'}$, is left adjoint to the restriction functor $\bsp_{n,A'} \rightarrow \bsp_{n,A}:T' \mapsto T'(- \otimes_A A')$.
\end{proposition}

\begin{proof}
We observed in Remark \ref{remark:projectiveimage} that $T_{A'}(V') \in \bsv_{A'}$ for each $V' \in \bsv_{A'}$, so it suffices to exhibit the appropriate strict polynomial structure on $T_{A'}$. Let $V',W' \in \bsv_{A'}$, and choose free supermodules $X',Y' \in \bsv_{A'}$ having $V'$ and $W'$ as direct summands, respectively. Write $X' = X_{A'}$ and $Y' = Y_{A'}$ for some free $A$-supermodules $X,Y \in \bsv_A$. Then as in Remarks \ref{remark:homsummand} and \ref{remark:projectiveimage}, the various even inclusion and projection maps defining $V'$ and $W'$ as summands of $X'$ and $Y'$ can be used to define the first and last arrows in the following composition of even $A'$-supermodule homomorphisms:
\begin{multline} \label{eq:definepolynomialmap}
\Hom_{\bsg^n(\bsv_{A'})}(V',W') \rightarrow \Hom_{\bsg^n(\bsv_{A'})}(X',Y') \\
= \Hom_{\bsg^n(\bsv_A)}(X,Y)_{A'}
\stackrel{(T_{X,Y})_{A'}}{\longrightarrow}\Hom_A(T(X),T(Y))_{A'} \\
= \Hom_{A'}(T_{A'}(X'),T_{A'}(Y')) \rightarrow \Hom_{A'}(T_{A'}(V'),T_{A'}(W')).
\end{multline}
We will take \eqref{eq:definepolynomialmap} to be the definition of the structure morphism $(T_{A'})_{V',W'}$, once we show that it is independent of the particular choices involving $X'$ and $Y'$, and that is is compatible with the composition of morphisms in $\bsg^n(\bsv_{A'})$.

Suppose $U',Z' \in \bsv_{A'}$ are some other free modules having $V'$ and $W'$ as direct summands, and let $U,Z \in \bsv_A$ be $A$-forms for $U'$ and $Z'$. Write $X' = V' \oplus X''$ and $U' = V' \oplus U''$ for some $U'',X'' \in \bsv_{A'}$, and then define the map $\alpha_{U',X'} \in \Hom_{A'}(U',X')$ to be the identity on $V'$ and zero on $U''$. Let $\beta_{Y',Z'} \in \Hom_A(Y',Z')$ be defined analogously. Since $\Hom_{A'}(U',X') = \Hom_A(U,X)_{A'}$, $\alpha_{U',X'}$ can be expressed as an $A'$-linear combination of elements of $\Hom_A(U,X)$, which by abuse of notation we will denote $(\alpha_{U,X})_{A'}$. Similarly, we can write $\beta_{Y',Z'}$ in the form $(\beta_{Y,Z})_{A'}$. Now using these maps, one can construct a two-row commutative diagram whose top row is the sequence of maps in \eqref{eq:definepolynomialmap}, whose bottom row is the corresponding sequence constructed using $U'$ and $Z'$ instead of $X'$ and $Y'$, and whose outermost vertical maps are the identity maps. Then by the commutativity of the diagram, $(T_{A'})_{V',W'}$ is well-defined. The compatibility between composition of morphisms and the base change isomorphism \eqref{eq:basechangebsg} also allows one to show that $T_{A'}(\phi) \circ T_{A'}(\phi') = T_{A'}(\phi \circ \phi')$ whenever $\phi$ and $\phi'$ are composable morphisms in $\bsg^n(\bsv_{A'})$. Thus, $T_{A'}$ admits the structure of a degree-$n$ strict polynomial superfunctor over $A'$.

To verify the last assertion of the proposition, let $T' \in \bsp_{n,A'}$ and let $\eta' \in \Hom_{\bsp_{A'}}(T_{A'},T')$. Then for each $V \in \bsv_A$ one gets the corresponding map $\eta'(V_{A'}): T_{A'}(V_{A'}) \rightarrow T'(V_{A'})$. But $T_{A'}(V_{A'}) = T(V)_{A'}$ by Remark \ref{remark:finalobject}, so
\[
\eta'(V_{A'}) \in \Hom_{A'}(T(A)_{A'},T'(V_{A'})) = \Hom_A(T(V),T'(V_{A'})).
\]
Then the fact that $\eta'$ was a homomorphism in $\bsp_{n,A'}$ implies that the family of induced maps $\eta(V): T(V) \rightarrow T'(V_{A'})$ defines a homomorphism $\eta: T \rightarrow T'(-\otimes_A A')$ in $\bsp_A$.

Conversely, let $\eta \in \Hom_{\bsp_A}(T,T'(-\otimes_A A'))$, and let $V' \in \bsv_{A'}$. Then for each object $(V,\phi)$ in the category $\bsv_A(V')$ we get the composite $A$-supermodule homomorphism
\[
T(V) \stackrel{\eta(V)}{\longrightarrow} T'(V_{A'}) \stackrel{T'(\phi)}{\longrightarrow} T'(V'),
\]
and hence an $A'$-supermodule homomorphism $\eta'(V): T(V)_{A'} \rightarrow T'(V')$. Varying over all $(V,\phi)$, it follows that the $\eta'(V)$ induce a unique $A'$-supermodule homomorphism $\eta'(V'): T_{A'}(V') \rightarrow T'(V')$. The reader can then check that the family of maps $\eta'(V')$ for $V' \in \bsv_{A'}$ defines a homomorphism $\eta': T_{A'} \rightarrow T'$ in $\bsp_{A'}$. The two constructions $\eta' \mapsto \eta$ and $\eta \mapsto \eta'$ described in this paragraph and the preceding are mutually inverse to each other, so this proves the last assertion of the proposition.
\end{proof}

An immediate consequence of base change being left adjoint to restriction is:

\begin{corollary} \label{corollary:agreeonfree}
Let $T \in \bsp_{n,A}$. If $T(-) \otimes_A A' \cong T'(-\otimes_A A')$ as strict polynomial superfunctors in $\bsp_{n,A}$ for some $T' \in \bsp_{n,A'}$, then $T_{A'} \cong T'$.
\end{corollary}

\begin{remark}
The notation $T \mapsto T_{A'}$ for the base change operation is potentially ambiguous when it is applied to one of the classical exponential superfunctors $\bss$, $\bsl$, $\bsg$, or $\bsa$, since we sometimes used the subscript $A$ to emphasize the ambient coefficient superring. However, each of these functors is defined over the prime ring of $A$ (i.e., over $\Z$ or $\Fp$ for some prime $p$, depending on the characteristic of $A$), and then one can use Corollary \ref{corollary:agreeonfree} to check that the functor over $A$ identifies with the base change of the corresponding functor over the prime ring.
\end{remark}

\begin{remark}
The proposition defines the base change $T_{A'} \in \bsp_{n,A'}$ of a homogenous strict polynomial superfunctor $T \in \bsp_{n,A}$. The base change of a non-homogeneous strict polynomial superfunctor $T = \bigoplus_{n \in \N} T^n \in \prod_{n \in \N} \bsp_{n,A} = \bsp_A$ is then defined componentwise.
\end{remark}

\begin{proposition}
Let $A'$ be a commutative superalgebra. The base change functor $\bsp_{n,A} \rightarrow \bsp_{n,A'}$ that sends $T$ to $T_{A'}$ is exact and maps projectives in $(\bsp_{n,A})_{\ev}$ to projectives in $(\bsp_{n,A'})_{\ev}$. For each $S,T \in \bsp_{n,A}$, base change induces an isomorphism
\[
\Hom_{\bsp_A}(S,T) \otimes_A A' \cong \Hom_{\bsp_{A'}}(S_{A'},T_{A'}).
\]
\end{proposition}

\begin{proof}
The proof is essentially the same as the proof of \cite[Proposition 2.6]{Suslin:1997}.
\end{proof}

\begin{corollary} \label{cor:functorbasechange}
For any $S,T \in \bsp_A$, there exists a canonical homomorphism
\[
\Ext_{\bsp_A}^\bullet(S,T) \otimes_A A' \rightarrow \Ext_{\bsp_{A'}}^\bullet(S_{A'},T_{A'}),
\]
which is an isomorphism if $A'$ is flat over $A$.
\end{corollary}

\subsection{Frobenius twists}

Now suppose $A$ is a commutative superalgebra of characteristic $p \geq 3$. Since $A$ is commutative, the Frobenius (i.e., $p$-th power) map $\varphi: A \rightarrow A$ satisfies $\varphi(a_0 + a_1) = a_0^p$ for $a_0 \in A_{\zero}$ and $a_1 \in A_{\one}$. It follows that $\varphi$ is an $\Fp$-superalgebra homomorphism. Now for $V \in \smod_A$ and $r \in \N$, set $V^{(r)} = V \otimes_{\varphi^r} A$, the base change of $V$ with respect to the $r$-th iterate of $\varphi$. Given $v \in V$, write $v^{(r)}$ for the element $v \otimes_{\varphi^r} 1$ of $V \otimes_{\varphi^r} A$. Thus, if $V \in \bsv_A$ is free with homogeneous basis $v_1,\ldots,v_n$, then $V^{(r)}$ is free with basis $v_1^{(r)},\ldots,v_n^{(r)}$. Since the assignment $\bsir: V \mapsto V^{(r)}$ is clearly additive, this implies that $\bsir$ defines a functor $\bsir: \bsv_A \rightarrow \bsv_A$, which satisfies the following properties for $V,W \in \bsv_A$:
	\begin{itemize}
	\item $(V \otimes_A W)^{(r)} = V^{(r)} \otimes_A W^{(r)}$,
	\item $\Hom_A(V,W)^{(r)} \cong \Hom_A(V^{(r)},W^{(r)})$, and
	\item $(V^\#)^{(r)} \cong (V^{(r)})^\#$.
	\end{itemize}
Explicitly, the isomorphism $\Phi: \Hom_A(V,W)^{(r)} \simrightarrow \Hom_A(V^{(r)},W^{(r)})$ is defined by $\Phi(\phi^{(r)})(v^{(r)}) = \phi(v)^{(r)}$. Then taking $W = A$, and identifying $A^{(r)}$ with $A$ via the map $a \otimes_{\varphi^r} a' \mapsto a^{p^r}a'$, the isomorphism $\Theta: (V^\#)^{(r)} \simrightarrow (V^{(r)})^\# = \Hom_A(V^{(r)},A)$ is defined by $\Theta(\phi^{(r)})(v^{(r)}) = [\phi(v)]^{p^r}$.

Recall from Example \ref{example:classicalfunctors} the symmetric superalgebra functor $\bss  = \bigoplus_{n \in \N} \bss^n \in \bsp_A$. Then for each $V \in \bsv$, $\bss(V)$ is a commutative $A$-superalgebra, and the function $\varphi^{r}(V): \bss(V)^{(r)} \rightarrow \bss(V)$ defined by $s \otimes_{\varphi^r} a \mapsto s^{p^r} \cdot a$ is an even $A$-linear map. Via the duality isomorphism $\bsg \cong \bss^\#$, the family of maps $\varphi^{r}(V)$ for $V \in \bsv_A$ corresponds to a family of even $A$-linear maps
\[
(\varphi^r)^\#(V): \bsg(V) \rightarrow \bsg(V)^{(r)} \text{ for $V \in \bsv_A$.}
\]
Then $\bsir$ admits the structure of a homogeneous degree-$p^r$ strict polynomial superfunctor over $A$, with the action of $\bsir$ on morphisms defined by
\begin{equation} \label{eq:IrVW}
\bsir_{V,W} := (\varphi^r)^\#(\Hom_A(V,W)): \bsg^{p^r} \Hom_A(V,W) \rightarrow \Hom_A(V,W)^{(r)} = \Hom_A(V^{(r)},W^{(r)})
\end{equation}
Since for each $V \in \bsv_A$ the map $\varphi^r(V)$ has its image in the subalgebra $\bss(\Vzero)$ of $\bss(V)$, it follows for each $V,W \in \bsv_A$ that \eqref{eq:IrVW} has its image in the even subspace $\Hom_A(V^{(r)},W^{(r)})_{\zero}$. If $V \in \bsv_A$ is free, then $\varphi^r(V)$ admits an explicit description on the generators of $\bsg(V)$ as in \cite[(2.7.2)]{Drupieski:2016}.

\begin{remark}[Purely even superrings] \label{remark:Frobeniuspurelyeven}
Suppose $A$ is a purely even commu\-tative $\Fp$-superalgebra, i.e., a commutative $\Fp$-algebra in the usual sense. Then the decom\-position $V = \Vzero \oplus \Vone$ of an $A$-supermodule $V$ into its even and odd subspaces is also an $A$-module direct sum decomposition of $V$. The decomposition $V = \Vzero \oplus \Vone$ is not functorial (for example, the projection $V \mapsto \Vzero$ is not compatible with the composition of odd isomorphisms $V \simeq \Pi(V) \simeq V$), and hence does not lift to a direct sum decomposition of the identity functor $\bsi \in \bsp_{1,A}$. On the other hand, since for $r \geq 1$ the map \eqref{eq:IrVW} has its image in the subspace
\[
\Hom_A(V^{(r)},W^{(r)})_{\zero} = \Hom_A(V_{\zero}^{(r)},W_{\zero}^{(r)}) \oplus \Hom_A(V_{\one}^{(r)},W_{\one}^{(r)}),
\]
it follows that the Frobenius twist $\bsir \in \bsp_{p^r,A}$ admits a direct sum decomposition
\[
\bsir = \bsi_{\zero}^{(r)} \oplus \bsi_{\one}^{(r)}
\]
such that $\bsirzero(V) = \Vzeror$ and $\bsirone(V) = \Voner$ for each $V \in \bsv_A$. Additionally, the functors $\bsirzero$, $\bsirone$, and $\Pi$ are related by the identity $\bsirone = \Pi \circ \bsirzero \circ \Pi$. When it is necessary to emphasize the particular base ring, we will denote $\bsirzero$ and $\bsirone$ by ${\bsi_{0,A}}^{(r)}$ and ${\bsi_{1,A}}^{(r)}$, respectively.
\end{remark}

\begin{remark}[Base change of Frobenius twists] \label{remark:Frobeniusbasechange}
Let $A \in \csalg_k$, and let $({\bsi_{0,k}}^{(r)})_A \in \bsp_{p^r,A}$ be the functor obtained from ${\bsi_{0,k}}^{(r)} \in \bsp_{p^r,k}$ via base change to $A$. Then
\[
(\bsir_{0,k})_A(\Amn) = (\bsir_{0,k})_A(\kmn \otimes_k A) = \bsir_{0,k}(\kmn) \otimes_k A = (k^{m|0})^{(r)} \otimes_k A,
\]
and $(A^{m|0})^{(r)} = (k^{m|0} \otimes_k A)^{(r)}$ identifies with $(k^{m|0})^{(r)} \otimes_k A$ via $(v \otimes a)^{(r)} \mapsto v^{(r)} \otimes a^{p^r}$. Then
	\[
	({\bsi_{0,k}}^{(r)})_A(\Amn) = (A^{m|0})^{(r)}, \quad \text{and similarly,} \quad	({\bsi_{1,k}}^{(r)})_A(\Amn) = (A^{0|n})^{(r)}.
	\]
Now for $A$ purely even, one can use Corollary \ref{corollary:agreeonfree} to check that $({\bsi_{0,k}}^{(r)})_A$ identifies with ${\bsi_{0,A}}^{(r)}$, and similarly to check that $({\bsi_{1,k}}^{(r)})_A$ identifies with ${\bsi_{1,A}}^{(r)}$.
\end{remark}

\makeatletter
\renewcommand*{\@biblabel}[1]{\hfill#1.}
\makeatother

\bibliographystyle{eprintamsplain}
\bibliography{graded-analogues-of-one-parameter-subgroups}

\end{document}